\documentclass[dvipsnames,12pt,a4paper,reqno]{amsart}
\usepackage[T1]{fontenc}
\usepackage[utf8]{inputenc}
\usepackage[dvipsnames]{xcolor}
\usepackage{graphicx}
\usepackage{tikz-cd}
\usepackage{stmaryrd}
\usepackage[pagebackref]{hyperref}
\hypersetup{
    pdfpagemode=FullScreen,
    colorlinks=true,
    linkcolor=purple,
    filecolor=blue,
    urlcolor=blue,
    citecolor=ForestGreen,
    linktoc=page}

\SetSymbolFont{stmry}{bold}{U}{stmry}{m}{n}
\usepackage[scr=rsfs,bb=pazo]{mathalfa}

\usepackage[alphabetic]{amsrefs}

\usepackage{microtype,amsmath,amsfonts,amssymb,amsthm,mathrsfs,bm}
\makeatletter
\renewcommand*\env@matrix[1][*\c@MaxMatrixCols c]{%
  \hskip -\arraycolsep
  \let\@ifnextchar\new@ifnextchar
  \array{#1}}
\makeatother
\usepackage{systeme}

\theoremstyle{plain}
\newtheorem{theorem}[]{Theorem}
\newtheorem{theorema}{Theorem}

\newtheorem{proposition}[theorem]{Proposition}
\newtheorem{lemma}[theorem]{Lemma}
\newtheorem{conjecturea}{Conjecture}

\newtheorem{corollary}[theorem]{Corollary}

\theoremstyle{definition}
\newtheorem{definition}[theorem]{Definition}
\newtheorem{instance}[theorem]{Example}
\newtheorem*{assumption}{Hypothesis}

\newtheorem{remark}[theorem]{Remark}

\usepackage{enumitem}
\numberwithin{theorem}{section}
\numberwithin{equation}{section}
\setlist[itemize]{leftmargin=2.5em}

\title{On the hyperplane conjecture}
\author{Tsung-Ju Lee}
\address{Tsung-Ju Lee, Department of Mathematics, National Cheng Kung University, Tainan 70101, Taiwan.}
\email{tsungju@gs.ncku.edu.tw}
\date{\today}

\begin{document}
\begin{abstract}
We prove the hyperplane conjecture, proposed by Hosono, Lian, and Yau 
in \cite{1996-Hosono-Lian-Yau-gkz-generalized-hypergeometric-systems-in-mirror-symmetry-of-calabi-yau-hypersurfaces},
for Batyrev's mirror pairs of Calabi--Yau hypersurfaces and 
its orbifold generalization, without assuming 
that the lattice points in
the defining reflexive polytope generate the ambient lattice. 
We then identify the pushforward of the extended hypersurface 
\(I\)-function with a generating series of normalized compact 
periods of the mirror, providing a period interpretation of the hypergeometric 
series governing ambient genus-zero orbifold Gromov--Witten theory.
As an application of our method, we 
also identify the residue of the Gauss--Manin connection
and the cup product morphism on the \(A\)-side.
\end{abstract}
\maketitle
\tableofcontents

\section{Introduction}
\subsection{Motivation}
Mirror symmetry predicted the numbers of rational curves on the 
quintic threefold through period calculations for another family of mirror Calabi--Yau (CY) threefolds
\cite{1991-Candelas-de-la-Ossa-Greene-Parkes-a-pair-of-calabi-yau-manifolds-as-an-exactly-soluable-superconformal-theory}.
Since then, it has attracted sustained interest from mathematicians.
The first mirror pair was written down by 
Greene and Plesser \cite{1990-Greene-Plesser-duality-in-calabi-yau-moduli-space}.
Soon afterward, Batyrev introduced a systematic construction of mirror pairs for
CY hypersurfaces in Gorenstein toric varieties using pairs of reflexive polytopes
\cite{1994-Batyrev-dual-polyhedra-and-mirror-symmetry-for-calabi-yau-hypersurfaces-in-toric-varieties}.
Together with Borisov, they further generalized the construction to CY 
complete intersections in Gorenstein toric varieties via nef-partitions
\cite{1996-Batyrev-Borisov-on-calabi-yau-complete-intersections-in-toric-varieties}
and proved that these pairs are topological mirror pairs
\cite{1996-Batyrev-Borisov-mirror-duality-and-string-theoretic-hodge-numbers}.

The next question is whether they
are quantum mirror pairs. Let \((X,X^{\vee})\) be a topological smooth CY mirror pair.
Such a pair is called a \emph{quantum mirror pair} if \(A(X)\cong B(X^{\vee})\) and 
\(A(X^{\vee})\cong B(X)\). Here 
the \(A\) model is the genus-zero Gromov--Witten (GW) theory, 
whereas the \(B\) model is the variation of Hodge structure (VHS).
In principle, one can verify these identifications by computing
the correlation functions on both sides and showing that they agree
under the mirror map.
When \(X\) is a smooth quintic threefold, \(A(X)\cong B(X^{\vee})\)
was independently established by Givental \cites{1996-Givental-equivariant-gromov-witten-invariants,1998-Givental-a-mirror-theorem-for-toric-complete-intersections}
as well as Lian, Liu, and Yau \cite{1997-Lian-Liu-Yau-mirror-principle-i}.
Since then, the mirror theorems and the quantum Lefschetz hyperplane theorem have extended
these \(A\)-model calculations to broader classes of varieties,
including CY complete intersections in toric varieties.
These techniques allow us to calculate genus zero GW invariants explicitly.

On the \(B\)-side, as observed by Batyrev
\cite{1993-Batyrev-variations-of-the-mixed-hodge-structure-of-affine-hypersurfaces-in-algebraic-tori},
the period integrals for CY hypersurfaces in a Gorenstein toric variety
are governed by a certain type of systems of partial differential equations.
It was introduced by Gel'fand, Graev, Kapranov, and Zelevinskii
\cites{1987-Gelfand-Graev-Zelevinski-holonomic-systems-of-equations-and-series-of-hypergeometric-type,
1989-Gelfand-Kapranov-Zelevinski-hypergeometric-functions-and-toral-manifolds}
as a generalization of hypergeometric structures and is now called a GKZ system.
Although the GKZ system in this circumstance is holonomic,
its solution rank is often too large;
it includes non-period solutions. 
Nevertheless, inspired by mirror symmetry, Hosono, Lian, and Yau
proposed a conjectral recipe for recovering period integrals among the solutions
\cite{1996-Hosono-Lian-Yau-gkz-generalized-hypergeometric-systems-in-mirror-symmetry-of-calabi-yau-hypersurfaces}. 
This is also known 
as the hyperplane conjecture (cf.~Conjecture \ref{conj:HLY}).
The conjecture was tested in many examples in toric geometry
\cites{1995-Hosono-Klemm-Theisen-Yau-mirror-symmetry-mirror-map-and-applications-to-calabi-yau-hypersurfaces,
1995-Hosono-Klemm-Theisen-Yau-mirror-symmetry-mirror-map-and-applications-to-complete-intersection-calabi-yau-spaces}.
For CY hypersurfaces in \(\mathbf{P}^{n}\), it is proven by 
Lian and Zhu \cite{2021-Lian-Zhu-on-the-hyperplane-conjecture-for-periods-of-calabi-yau-hypersurfaces-in-pn}
using tautological systems
\cites{2013-Lian-Song-Yau-periodic-integrals-and-tautological-systems,2013-Lian-Yau-period-integrals-of-cy-and-general-type-complete-intersections};
in which case, the tautological system completely characterizes the
period integrals \cite{2014-Bloch-Huang-Lian-Srinivas-Yau-on-the-holonomic-rank-problem}.

In this paper, we will prove the conjecture in full generality, and its
generalization to CY hypersurfaces in toric Deligne--Mumford stacks.
Such a generalization is desirable, since MPCP desingularizations are
not smooth in higher dimension.

\subsection{Statements of the main results}
Let \(N=\mathbf{Z}^{n}\) be a rank \(n\) lattice and
\(M=\mathrm{Hom}_{\mathbf{Z}}(N,\mathbf{Z})\)
be its dual lattice. Denote by \(N_{\mathbf{R}}\) and
\(M_{\mathbf{R}}\) their scalar extensions.
Batyrev's mirror construction begins with a pair of dual reflexive 
polytopes \((\Delta\subset M_{\mathbf{R}},\nabla=\Delta^{\vee}\subset N_{\mathbf{R}})\)
which gives a pair of toric varieties \((\mathbf{P}_{\Delta},\mathbf{P}_{\nabla})\).
Denote by \(Z\) (resp.~\(\check{Z}\)) a maximal projective crepant partial
(MPCP) desingularization of \(\mathbf{P}_{\Delta}\) (resp.~\(\mathbf{P}_{\nabla}\)).
In general, \(Z\) and \(\check{Z}\) are orbifolds, and we shall 
use \(\mathcal{Z}\) and \(\check{\mathcal{Z}}\) to denote
the corresponding toric Deligne--Mumford stacks.

Let \(\Delta\cap M = \{m_{0}=\mathbf{0},m_{1},\ldots,m_{p}\}\). Define an integral matrix 
\begin{equation}
\label{eq:GKZ-A}
  A:=
  \begin{bmatrix}
    1 & 1 & \cdots & 1\\
    \mathbf{0} & m_{1} & \cdots & m_{p}
  \end{bmatrix}\in\mathrm{Mat}_{(n+1)\times (p+1)}(\mathbf{Z})
\end{equation}
and 
\begin{equation}
  -\delta:=\begin{bmatrix}
  -1\\
  \mathbf{0}
  \end{bmatrix}\in \mathbf{C}^{n+1}.
\end{equation}
As observed by Batyrev
\cite{1993-Batyrev-variations-of-the-mixed-hodge-structure-of-affine-hypersurfaces-in-algebraic-tori}, 
period integrals for
CY hypersurfaces in \(Z\) are governed by a GKZ
system \(M_{A}(-\delta)\). 
We recall the definition of \(M_{A}(\beta)\).
Let \(\beta\in\mathbf{C}^{n+1}\).
Let \(a_{0},a_{1},\ldots,a_{p}\) be coordinates on \(\mathbf{C}^{p+1}\)
labeled by the columns of \(A\) and \(\partial_{i}=
\partial\slash\partial a_{i}\). 
The GKZ system \(M_{A}(\beta)\) consists
of two types of operators:
\begin{itemize}
  \item[\textbf{(a)}] (box operators) \(\partial^{u_{+}}-\partial^{u_{-}}\) for \(u_{\pm}\in\mathbf{Z}^{p+1}_{\ge 0}\)
  with \(Au_{+}=Au_{-}\).
  \item[\textbf{(b)}] (Euler operators) \(\sum_{j=0}^{p} A_{ij}a_{j}\partial_{j}-\beta_{i}\) for \(i=1,\ldots,n+1\).
\end{itemize}
Let \(L:=\operatorname{Ker}(A)\cap\mathbf{Z}^{p+1}\). 
The naive Gamma-series solution
\begin{equation}
\label{eq:hol-zero-sol}
  \frac{1}{a_0}\sum_{\substack{\ell\in L,\ell_0\le0}}
  \frac{1}{\Gamma(\ell_{0})\prod_{i>0}\Gamma(1+\ell_i)}a^{\ell}
\end{equation}
is identically zero. Instead, Hosono, Lian, and Yau introduced
a regularized holomorphic solution 
\cite{1996-Hosono-Lian-Yau-gkz-generalized-hypergeometric-systems-in-mirror-symmetry-of-calabi-yau-hypersurfaces}
\begin{equation}\label{eq:HLYrenormalized-b-series}
 \omega(a)
 =\frac{1}{a_0}\sum_{\substack{\ell\in L,\ell_0\le0}}(-1)^{\ell_0}
 \frac{\Gamma(1-\ell_0)}{\prod_{i>0}\Gamma(1+\ell_i)}a^{\ell}.
\end{equation}
and showed that \(\omega(a)\) is non-zero
and is the holomorphic series
solution, unique up to scalar, around the so-called large complex structure limit (LCSL) point
\cite{1997-Hosono-Lian-Yau-maximal-degeneracy-points-of-gkz-systems}.
Assume for simplicity that \(\check{\mathcal{Z}}\) is smooth, and 
we will simply write \(\check{Z}\) for it.
Owing to the duality construction, the indicial ring of \(M_{A}(-\delta)\)
is shown to be isomorphic to the cohomology ring \(R=\mathrm{H}^{\bullet}(\check{Z};\mathbf{C})\)
of \(\check{Z}\). Consequently, using \(R\),
Hosono, Lian, and Yau implemented the generalized Frobenius method to
construct the full set of solutions to \(M_{A}(-\delta)\) near the LCSL point.

In practice, they deformed \(\omega(a)\) by the cohomology classes in \(R\)
and obtain a cohomology-valued series
\begin{equation}
\label{eq:HLY-renormalized-b-series}
 \mathcal{B}(a)
 =\frac{1}{a_0}\sum_{\substack{\ell\in L,\ell_0\le0}}(-1)^{\ell_0}
 \frac{\Gamma(1-\ell_0-D_0)}
 {\prod_{i>0}\Gamma(1+\ell_i+D_i)}a^{\ell+D}
\end{equation}
where \(D_{i}\) is the divisor class corresponding to the ray \(m_{i}\) for
\(1\le i\le p\) and \(D_{0}=-\sum_{i=1}^{p}D_{i}=-e\).
It can be checked that the coefficients generate the
set of local solutions to \(M_{A}(-\delta)\).

However, as we have mentioned, by a dimension count, it is clear that
not every solution to \(M_{A}(-\delta)\) comes from a period integral.
Hosono, Lian, and Yau subsequently conjectured that 
\begin{conjecturea}[The hyperplane conjecture]
\label{conj:HLY}
As a \(\mathrm{H}^{\bullet}(\check{Z};\mathbf{C})\)-valued series, the coefficients of
\begin{equation}
  \mathcal{B}(a)\cup [-c_{1}(K_{\check{Z}})]
\end{equation}
generate the full set of period integrals of CY hypersurfaces in \(Z\)
from vanishing cycles near the LCSL point.
\end{conjecturea}
Said differently, cupping with the anticanonical class
of \(\check{Z}\) restores period integrals for 
CY hypersurfaces in \(Z\) among the solutions.
The purpose of this manuscript is to prove the conjecture. 
In fact, we will not only prove the conjecture but also
its generalization to toric Deligne--Mumford stacks. In a nutshell,
we prove the conjecture for
a \(\mathrm{H}^{\bullet}_{\mathrm{CR}}(\check{\mathcal{Z}};\mathbf{C})\)-valued series 
as well as periods of residues of meromorphic top forms with higher order poles
under a weaker hypothesis which we are going to explain now.

Begin with a reflexive polytope \(\Delta\subset M_{\mathbf{R}}\). Let 
\begin{equation}
  \{m_{0}=\mathbf{0},m_{1},\ldots,m_{p}\}\subset \Delta\cap M
\end{equation}
be a subset and \(f_{a}:=\sum_{i=0}^{p} a_{i}x^{m_{i}}\) be a Laurent polynomial.
Let \(C\) be the cone over \(\Delta\), i.e., 
\begin{equation}
  C:=\operatorname{Cone}(\Delta\times\{1\}).
\end{equation}
Let \(\check{\Sigma}\) be a coherent simplicial star fan, i.e., every maximal
cone contains the distinguished vector \((\mathbf{0},1)\), 
whose support is \(C\). In this paper, we assume that 
the ray monomials occuring in \(f_{a}\) are exactly the \(1\)-cones
in \(\check{\Sigma}\) and \(f_{a}\) is nondegenerate.
We view \(f_{a}\) as a global section in 
\(\mathrm{H}^{0}(Z,-K_{Z})\) or simply a Laurent polynomial
interchangably.
To study the solutions to \(M_{A}(-\delta)\), we replace \(Z\) by
a toric resolution.
Let \(Y_{a}\subset T\) be the affine hypersurface defined by
the Laurent polynomial \(f_{a}\) and \(E_{a}=\overline{Y}_{a}\)
and \(U_{a}=Z\setminus E_{a}\). 
Denote by \(D\) the toric boundary in \(Z\).

For the matrix \(A\) as in \eqref{eq:GKZ-A}, Borisov and Horja introduce
the notion of better behaved GKZ systems. Roughly speaking, a solution of the better behaved GKZ system
is a family of scalar functions \(\{\Phi_{c}(a)\}_{c\in C\cap(M\times\mathbf{Z})}\)
such that
\begin{equation}
\begin{cases}
  \partial_{a_{i}}\Phi_{c}=\Phi_{c+v_{i}}\\
  \sum_{i} a_{i}\langle \mu,v_{i}\rangle
  \Phi_{c+v_{i}}+\langle 
  \mu,c-\beta\rangle\Phi_{c}=0,~\mu\in N\times\mathbf{Z}.
\end{cases}
\end{equation}
Here \(v_{i}=(m_{i},1)\) and \(a=(a_{0},a_{1},\ldots,a_{p})\) is the 
variable associated to the columns of \(A\).
Each \(\Phi_{c}\) is a solution to the GKZ system \(M_{A}(\beta-c)\).
The projection of \(\check{\Sigma}\) to \(M_{\mathbf{R}}\) is a complete fan
which defines a toric Deligne--Mumford stack \(\check{\mathcal{Z}}\); each
\(v_{i}\) determines a toric divisor \(D_{i}\).
One can construct an indexed family of \(\mathrm{H}^{\bullet}_{\mathrm{CR}}(
\check{\mathcal{Z}};\mathbf{C})\)-valued series 
\(\{\mathcal{B}_{c}^{\mathrm{raw}}(a)\}_{c\in C\cap(M\times\mathbf{Z})}\)
for which \(\{\lambda(\mathcal{B}_{c}^{\mathrm{raw}})\}_{c\in C\cap(M\times\mathbf{Z})}\) 
with arbitrary constant functional \(\lambda\colon \mathrm{H}^{\bullet}_{\mathrm{CR}}(
\check{\mathcal{Z}};\mathbf{C}) \to\mathbf{C}\) 
are solutions to the better behaved GKZ system with \(\beta=0\).
It could happen that the naive power series solution (cf.~\eqref{eq:hol-zero-sol})
to \(M_{A}(\beta-c)\) is identically zero. Following Hosono, Lian, and Yau,
we consider a regularized solution and consider its 
corresponding \(\mathrm{H}^{\bullet}_{\mathrm{CR}}(
\check{\mathcal{Z}};\mathbf{C})\)-valued series \(\mathcal{B}_{c}(a)\). It is straightforward
to check that 
\begin{equation}
  \mathcal{B}_{c}^{\mathrm{raw}}(a) = -e\cup \frac{\sin(\pi e)}{\pi e}\cup \mathcal{B}_{c}(a).
\end{equation}
where \(e\) is the first Chern class of \(\check{\mathcal{Z}}\) and \(c\in C\cap(M\times\mathbf{Z})\) is an 
interior lattice point.

Integral points in \(C\) have
another interpretation.
To each \(c=(m,k)\in C\cap (M\times\mathbf{Z})\) with \(k\ge 1\) we associate a
meromorphic top form
\begin{equation}
  \omega_{c}:=(-1)^{k-1}\Gamma(k)\frac{x^{m}}{f_{a}^{k}}\bigwedge_{i=1}^{n}\frac{\mathrm{d}x_{i}}{x_{i}}.
\end{equation}
When \(c\) is an \emph{interior} lattice point, one can check 
\(\omega_{c}\) can be extended to a meromorphic 
top form on \(Z\) with poles along \(E_{a}\). Set \(B_{a}=U_{a}\cap D\).
Consider the image of the morphism (over \(\mathbf{C}\))
\begin{equation}
 W:=\operatorname{Im}\left(\mathrm{H}_n(Z,E_{a})\xrightarrow{\partial}
 \mathrm{H}_{n-1}(E_{a})
 \xrightarrow{\operatorname{Tube}}\mathrm{H}_n(U_{a})
 \xrightarrow{q_{a}}\mathrm{H}_n(U_{a},B_{a})\right).
\end{equation}
For each interior integral point \(c\), we now have two maps
\begin{equation}
    \mathcal{P}_{\mathcal{B}_{c}}\colon \mathrm{H}^{\bullet}_{\mathrm{CR}}(
  \check{\mathcal{Z}};\mathbf{C})\to\mathfrak{Sol}(M_{A}(-c)),~\alpha\mapsto 
  \langle\alpha\cup \mathcal{B}_{c}\rangle_{\mathrm{CR}}
\end{equation}
and
\begin{equation}
    J_{c}\colon \mathrm{H}_n(U_{a},B_{a};\mathbf{C})\to \mathfrak{Sol}(M_{A}(-c)),~\gamma\mapsto \int_{\gamma}\omega_{c}.
\end{equation}
Our first main result is the following.
\begin{theorema}[=Theorem \ref{thm:seriesimage}]
\label{thm:main-1-intro}
For each interior integral point \(c\), we have the equality
\begin{equation}
  \mathcal{P}_{\mathcal{B}_{c}}\left(e\cup \mathrm{H}^{\bullet}_{\mathrm{CR}}(
  \check{\mathcal{Z}};\mathbf{C})\right) = J_{c}(W).
\end{equation}
\end{theorema}
By taking \(c=\delta=(0,1)\), we obtain the following.

\begin{theorema}[=Theorem \ref{thm:HLY-hyperplane}]
\label{thm:main-2-intro}
Conjecture \ref{conj:HLY} holds.
\end{theorema}

To prove Theorem \ref{thm:main-1-intro}, two spaces to compare are
\begin{equation}
\label{eq:vs-to-compare-sol}
 \mathrm{H}_{n}(U_{a},B_{a};\mathbf{C})
 ~\text{and}~
 \mathrm{H}_{\mathrm{CR}}^\bullet(\check{\mathcal Z};\mathbf{C}).
\end{equation}
By the result in \cite{2020-Lee-Zhang-a-hypergeometric-systems-and-relative-cohomology},
the homology on the left hand side of \eqref{eq:vs-to-compare-sol}
can be identified with the solution space to 
the GKZ system \(M_{A}(-\delta)\) when the columns of \(A\) generate
the full lattice and the ambient space \(Z\) is smooth, while the right hand side 
is the space of cohomology insertions
for the \(A\)-side genus-zero Gromov--Witten invariants for 
the mirror.



For the left hand side in 
the equation \eqref{eq:vs-to-compare-sol}, by a version of Poincar\'{e} 
duality (cf.~\cite{2017-Huber-Muller-Stach-periods-and-nori-motives}*{Theorem 2.4.5}), 
we have an isomorphism
\begin{equation}
\label{eq:coh-per-comparison}
  \mathrm{H}_{n}(U_{a},B_{a};\mathbf{C})\cong \mathrm{H}^{n}(T,Y_{a};\mathbf{C}).
\end{equation}
To prove the conjecture, we need a
comparison between 
\begin{equation}
\label{eq:iso-vs-need}
  \mathrm{H}^{n}(T,Y_{a};\mathbf{C})~\mbox{and}~\mathrm{H}^{\bullet}_{\mathrm{CR}}(\check{\mathcal{Z}};\mathbf{C}).
\end{equation}

Via a relative de Rham comparison,
the left hand side corresponds to a 
Koszul complex of logarithmic twisted operators on the semigroup
algebra \(S:=\mathbf{C}[C\cap (M\times\mathbf{Z})]\); this was
observed by Batyrev. Owing to this interpretation, we obtain a 
correspondence 
\begin{equation}
[c]\mapsto \mathcal{B}_{c}^{\mathrm{raw}}(a)
\end{equation}
which induces a map
\begin{equation}
\mathrm{H}^{\bullet}_{\mathrm{CR}}(\check{\mathcal{Z}};\mathbf{C})^{\vee}
\ni\lambda\mapsto
\bigl([c]\mapsto\lambda(\mathcal{B}_c^{\mathrm{raw}}(a))\bigr)
\in \mathrm{H}_{n}(T,Y_{a};\mathbf{C}).
\end{equation}
We will show that \([c]\mapsto\lambda(\mathcal{B}_c^{\mathrm{raw}}(a))\)
is the functional on \(\mathrm{H}^{n}(T,Y_{a};\mathbf{C})\)
given by 
\begin{equation}
  [c]\mapsto \lambda(\mathcal{B}_{c}^{\mathrm{raw}}(a))=\int_{\partial\Gamma_\lambda(a)}
\operatorname{Res}_{Y_a}\omega_{c}.
\end{equation}
for some relative cycle
\(\Gamma_{\lambda}(a)\in 
\mathrm{H}_n(T,Y_a;\mathbf{C})\)
and therefore it connects to a period.
Note that at the distinguished index
\(\delta\), its value is 
\begin{equation}
  \int_{\partial\Gamma_{\lambda}(a)}
  \operatorname{Res}_{Y_a}\frac{\Omega}{f_a}.
\end{equation}
Conversely, after
choosing a smooth projective model, a Hodge-theoretic splitting
shows that every compact period 
comes in this way. 
The completeness of the full indexed family identifies every 
horizontal functional on relative cohomology with a constant 
coefficient functional. The said Hodge-theoretic splitting then 
yields the reverse inclusion of period spaces
and hence Conjecture
\ref{conj:HLY}.


To establish completeness, we need a dimension count.
Nevertheless, we will develop more: we will construct a family of
twisted logarithmic de Rham complexes \(\mathcal{K}^\bullet\)
over an open subset of \(\operatorname{Spec}\mathbf{C}[s]\)
containing \(s=0\) and \(s=1\). Its fiber cohomology at \(s=1\)
gives \(\mathrm{H}^n(T,Y_a;\mathbf C)\), while its fiber
cohomology at \(s=0\) is isomorphic to
\(\mathrm{H}^\bullet_{\mathrm{CR}}(\check{\mathcal{Z}};\mathbf{C})\).
This gives a noncanonical identification between the two
vector spaces and hence the equality of their dimensions.
Note that an abstract isomorphism alone does not identify the prescribed
series with geometric periods; this compatibility comes from
the indexed family and the relative de Rham comparison.

As a result, we can show that the pushforward of the 
extended \(I\)-function for CY hypersurfaces in \(\check{\mathcal{Z}}\)
is equivalent to the generating series of compact periods for \(Y\subset\mathcal{Z}\).
Without introducing technical terms, our second main result can be stated as follows.

\begin{theorema}[=Theorem~\ref{thm:extended-i-and-periods}]
Under a change of variables and an invertible sectorwise normalization map, 
the pushforward of the extended \(I\)-function for 
CY hypersurfaces \(\check{Y}\subset\check{\mathcal{Z}}\) is
identified with the generating series of compact periods for \(Y\subset \mathcal{Z}\).
\end{theorema}

Finally, as another application, we also use \(\mathcal{K}^\bullet\) to
identify the residue of the Gauss--Manin connection with
divisor cup product on the $A$-side (cf.~Theorem \ref{thm:monodromy-residue-and-cup-prod}).

The interpolation complex bridges
the \(A\)-side and the \(B\)-side of 
Batyrev's mirror pairs of toric
CY hypersurfaces. The next is to determine
whether it identifies integral structures on 
both sides appropriately.
As explained by Hosono
\cites{2000-Hosono-local-mirror-symmetry-and-type-IIA-monodromy-of-calabi-yau-manifolds,2006-Hosono-central-charges-symplectic-forms-and-hypergeometric-series-in-local-mirror-symmetry}, 
the integral structure on the \(A\)-side needs a correction;
the \(K\)-theoretic integral structure is expected
(see \cite{2011-Iritani-quantum-cohomology-and-periods} for further developments).
We defer the discussion on the
compatibility of the interpolation construction with 
integral structures to future work.

\begin{remark}
It is expected that the proof can be generalized to 
toric Calabi--Yau complete intersections.
\end{remark}

\subsection{Mirror symmetry and relation to the \(A\)-side}
Mirror symmetry
predicts an identification between \(A\) model
and \(B\) model for mirror pairs.
Here the \(A\) model is taken to be the genus-zero Gromov--Witten (GW)
invariants, while the \(B\) model
is the VHS.

GW theory has been constructed and intensively studied in the past thirty years.
In the genus-zero case, the quantum Lefschetz hyperplane theorem (QLHT)
\cites{2007-Coates-Givental-quantum-riemann-roch-lefschetz-and-serre} relates 
the GW invariants of the ambient space and the GW invariants of hypersurfaces; 
its orbifold extension was further established in
\cites{2010-Tseng-orbifold-quantum-riemann-roch-lefschetz-and-serre}.
In our setup, since \(\Delta\) and \(\nabla=\Delta^{\vee}\) are
reflexive, \(\mathbf{P}_{\Delta}\)
and \(\mathbf{P}_{\nabla}\) as well as
their MPCP desingularizations \(Z\) and \(\check{Z}\) are Gorenstein. In particular, 
the canonical bundles \(K_{\mathcal{Z}}\)
and \(K_{\check{\mathcal{Z}}}\) are pullbacks of bundles on the coarse moduli
\(Z\) and \(\check Z\). Also, \(-K_{\check{\mathcal{Z}}}\) is numerically effective.
The QLHT for orbifolds then applies. Owing to the use of QLHT, the (orbifold) GW invariants
for CY hypersurfaces calculated above only include those with \emph{ambient insertions}.
Whenever the cohomology ring of the hypersurface is generated by restrictions of
ambient divisor classes and the hypersurface is smooth, one obtains
the full information of GW invariants by reconstruction
\cite{2004-Lee-Pandharipande-a-reconstruction-theorem-in-quantum-cohomology-and-quantum-k-theory}.

From this perspective, the hyperplane conjecture is the \(B\)-side counterpart of 
quantum Lefschetz: it identifies the period subspace with the coefficient space after 
anticanonical cup product. 
The above observation suggests that we should be able to recover 
period integrals corresponding to \emph{vanishing cohomology} 
of the hypersurface.
Our proof of Conjecture 
\ref{conj:HLY}, together with the mirror theorem for toric stacks
\cite{2015-Coates-Corti-Iritani-Tseng-a-mirror-theorem-for-toric-stacks} and quantum Lefschetz
\cite{2019-Coates-Corti-Iritani-Tseng-some-applications-of-the-mirror-theorem-for-toric-stacks}, 
establishes the corresponding identification of the genus-zero orbifold GW
invariants with hypersurface periods corresponding to vanishing cohomology for 
Batyrev's mirror pairs of Calabi--Yau hypersurfaces.

\subsection{Relation with existing works}
Our approach builds on the work of Batyrev
\cite{1993-Batyrev-variations-of-the-mixed-hodge-structure-of-affine-hypersurfaces-in-algebraic-tori} and Stienstra 
\cite{1998-Stienstra-resonant-hypergeometric-systems-and-mirror-symmetry}
on the relation between hypergeometric systems and the cohomology 
of toric hypersurfaces. Stienstra constructed a relative cohomology 
\(\mathscr{D}\)-module and, under suitable unimodularity hypotheses, described 
its trivialization by cohomology-valued \(\Gamma\)-series and related the relevant 
weight submodule to a principal ideal in the toric cohomology ring. 

From the perspective of quantum cohomology, Iritani
\cite{2011-Iritani-quantum-cohomology-and-periods} established comparisons 
between hypergeometric solutions and mirror periods (ambient/residual VHS comparison)
as well as their integral structures using quantum Lefschetz and 
the \(K\)-theoretic integral structure defined by the \(\widehat{\Gamma}\)-class
under the lattice generation assumption. 

There is also a point of view from homological mirror symmetry (HMS). In a recent preprint
in \cite{Ganatra:2025aa}, under appropriate hypotheses, Ganatra and Sheridan established 
the implication ``HMS~\(\Rightarrow\) Hodge-theoretic mirror symmetry,''
giving an identification between the genus-zero GW invariants and period integrals.
For a large class of Batyrev's mirror pairs, the required HMS theorem was proven in \cite{Ganatra:2024aa}.
This gives a categorical route to closed-string mirror symmetry in the cases covered by their results.

In this paper, we give a direct cohomological approach as explained above
to the hyperplane conjecture, which removes
the lattice-generation hypothesis through completeness of the full indexed family
and covers the orbifold setting. 

\subsection{Acknowledgments}
The author thanks Professor C.-L.~Wang for his insightful comments and 
valuable suggestions on an earlier version of this manuscript, and expresses 
sincere gratitude to Professors S.~Hosono, B.~Lian, H.-W.~Lin, and C.-L.~Wang 
for their continued support and interest in this work.

The author also thanks Si-Yang Liu for bringing to his attention 
a recent preprint by Ganatra and Sheridan that is closely related to the
topics discussed in this work.

The author acknowledges support from the National Science and Technology Council (NSTC), Taiwan, under Grant No.~115-2115-M-006-013-MY3.

\section{Notations and assumptions}
\label{sec:preliminaries}
Let \(N=\mathbf{Z}^{n}\) be a rank \(n\) lattice and 
\(M:=\mathrm{Hom}_{\mathbf{Z}}(N,\mathbf{Z})\) 
be its dual. Let \(\Delta\subset M_{\mathbf{R}}:=M\otimes_{\mathbf{Z}}\mathbf{R}\)
be a reflexive polytope which determines a toric variety \(\mathbf{P}_{\Delta}\)
with maximal torus
\(T=\operatorname{Spec}\mathbf{C}[M]\), where \(M\cong\mathbf{Z}^n\) is identified with 
the character lattice.
Often time we will fix a MPCP desingularization \(Z\to\mathbf{P}_{\Delta}\).
Let \(\{m_{0}=\mathbf{0},m_{1},\ldots,m_{p}\}\subset\Delta\cap M\)
be a subset of integral points in \(\Delta\).
We consider a Laurent polynomial
\begin{equation}
 f_a(x)=\sum_{i=0}^{p}a_i x^{m_i},~\mbox{and}~
 Y_a=\{f_a=0\}\subset T.
\end{equation}
Assume also that \(f_a\) is nondegenerate, i.e., for each nonempty
face \(\tau\subseteq\Delta\), including \(\Delta\) itself, the face
polynomial \(f_{a,\tau}\) and all its logarithmic derivatives have no
common zero in \(T\). This hypothesis implies the smoothness of \(Y_a\).
Choose a coherent star triangulation of $\Delta$ with vertex set
$\{\mathbf{0},m_1,\ldots,m_p\}$, where
$m_1,\ldots,m_p\in\partial\Delta\cap M$.
Here a star triangulation means a triangulation in which
$\mathbf{0}$ is a vertex of every maximal simplex.
Let $\check\Sigma$ be the simplicial fan obtained by taking
cones over the simplices of this triangulation in
\begin{equation}
  C=\mathrm{Cone}(\Delta\times\{1\}).
\end{equation}
Its primitive ray generators are
\begin{equation}
    v_0=(\mathbf{0},1),~\mbox{and}~v_i=(m_i,1)~\mbox{with}~1\le i\le p.
\end{equation}
Coherence means that there is a rational, positively homogeneous,
convex piecewise-linear function on \(C\) whose maximal domains
of linearity are exactly the maximal cones of \(\check\Sigma\).
Explicitly, \(\check\Sigma\) is the homogenized star fan obtained 
by lifting a star triangulation of \(\Delta\). 

The projection of \(\check\Sigma\) to \(M_{\mathbf{R}}\)
thus gives rise to a complete fan; each maximal cone
is a positive span of vertices in a maximal simplex. 
The projected fan defines a toric variety \(\check Z\); it 
is projective, simplicial and Gorenstein, and the subdivision is crepant.
Note that in our convention, the toric variety associated
with the fan \(\check\Sigma\) is the total space of \(K_{\check Z}\).

If every boundary lattice point is used, the subdivision gives a
maximal projective crepant partial (MPCP) desingularization
model. Note that the maximality of the subset \(\{m_{0},\ldots,m_{p}\}\) 
is not used in the proof. In other words,
we do not assume that \(\check{Z}\) is a MPCP desingularization of
\(\mathbf{P}_{\Delta^{\vee}}\) in the proof.
What we 
require in this manuscript is the following.
\begin{assumption}
The monomials \(x^{m_{i}}\)
with \(i\ne 0\) occurring in \(f_a\)
are exactly the ray monomials in the simplicail fan \(\check{\Sigma}\), 
including the ray generated by \((\mathbf{0},1)\)
with nonzero coefficients, and \(f_{a}\) is nondegenerate.
\end{assumption}
Thus a nonmaximal coherent star triangulation is allowed,
with the corresponding support of \(f_{a}\).
We emphasize that we do not assume that this support generates the ambient
lattice (see Remark~\ref{rmk:ex-generate}).
Let \(\check{\mathcal{Z}}\) be the smooth toric
Deligne--Mumford stack associated with the projected fan
and its primitive ray generators. Its coarse space is
\(\check Z\), and the coarse moduli map induces an isomorphism
on ordinary cohomology with \(\mathbf{C}\) coefficients.
Its Chen--Ruan cohomology also includes the cohomology
of the twisted sectors, shifted in cohomological degree
by twice their ages. These ages are integral because
\(\check{Z}\) is Gorenstein. Since the sectors have ordinary
cohomology concentrated in even degrees, the Chen--Ruan
cohomology is concentrated in even degrees.

To construct a complex to connect
spaces in \eqref{eq:iso-vs-need}, 
we introduce a homogenizing variable \(t\), and put
\begin{equation}
 \widetilde N=N\oplus\mathbf{Z},~
 \widetilde M=\operatorname{Hom}_{\mathbf{Z}}(\widetilde N,\mathbf{Z}),~\mbox{and}~
 v_i=(m_i,1).
\end{equation}
The cone \(C\) over \(\Delta\) and its semigroup algebra \(S\) are given by
\begin{equation}
 C=\operatorname{Cone}(\Delta\times\{1\})\subset\widetilde M_{\mathbf{R}},
 ~\mbox{and}~S=\mathbf{C}[C\cap\widetilde M].
\end{equation}
For \(c=(m,k)\in C\cap\widetilde M\), write \([c]\) 
for the basis monomial
\(x^m t^k\). Thus
\begin{equation}
 [c][c']=[c+c'],~\mbox{and}~\deg[c]=k.
\end{equation}
Particularly, we have \([0]=1\), whereas \([(0,1)]=t\).

Given a non-degenerate Laurent polynomial \(f_{a}\), we associate a homogeneous function
\begin{equation}
 F_a=t f_a=\sum_i a_i[v_i]\in S.
\end{equation}
For \(\mu\in\widetilde N\), the pairing \(\langle\mu,c\rangle\) is the evaluation of
the integral linear functional \(\mu\) at \(c\). The logarithmic derivation
associated with \(\mu\) is
\begin{equation}
 \theta_\mu[c]:=\langle\mu,c\rangle[c].
\end{equation}
This is a derivation of \(S\). In coordinates, if \(\mu_1,\ldots,\mu_n,\mu_t\) is the
coordinate dual basis, then we have 
\begin{equation}
 \theta_{\mu_j}=x_j\frac{\partial}{\partial x_j},~\mbox{and}~
 \theta_{\mu_t}=t\frac{\partial}{\partial t}.
\end{equation}

\begin{remark}
\label{rmk:ex-generate}
To clarify our assumptions, consider the Laurent polynomial
on \(T:=(\mathbf{C}^{\ast})^{2}\)
\begin{equation}
  f_{a} = a_{1}x^{-1}y^{2}+a_{2}x^{2}y^{-1}+a_{3}x^{-1}y^{-1}+a_{0}
\end{equation}
with \(\Delta=\mathrm{Conv}\{(2,-1),(-1,2),(-1,-1)\}\).
For a generic choice of the coefficients \(a\), the Laurent polynomial \(f_{a}\) is nondegenerate
and it defines a smooth affine hypersurface \(Y_{a}\subset T\) as
well as a smooth elliptic curve \(E_{a}=\overline{Y}_{a}\) in \(Z=\mathbf{P}_{\Delta}=\mathbf{P}^{2}\).
One can check that 
\begin{equation}
  \dim_{\mathbf{C}}\mathrm{H}^{2}(U_{a},U_{a}\cap D;\mathbf{C})=\dim_{\mathbf{C}}\mathrm{H}^{2}(T,Y_{a};\mathbf{C})=9
\end{equation}
and this is equal to \(\dim_{\mathbf{C}}\mathrm{H}^{\bullet}_{\mathrm{CR}}(\check{\mathcal{Z}};\mathbf{C})\).

Nevertheless, the GKZ system for the family of hypersurfaces \(f_{a}\) is given by
\begin{equation}
  A=\begin{bmatrix}
  1 & 1 & 1 & 1\\
  0 & 2 & -1 & -1\\
  0 & -1 & 2 & -1
  \end{bmatrix},~\mbox{and}~
  \beta=
  \begin{bmatrix}
    -1\\
    0\\
    0
  \end{bmatrix}
\end{equation}
whose holonomic rank is \(3\).
The inconsistency comes from the fact that the columns of 
\(A\) generate a sublattice of index \(3\) in \(\mathbf{Z}^{2}\).
\end{remark}

\section{Twisted operators and their cohomology}
\label{sec:relative_cohomology_and_the_associated_graded_space}

\begin{definition}
For \(\mu\in\widetilde N\), define a linear operator on \(S\) by
\begin{equation}\label{eq:operator}
 D_{\mu,a}[c]
 =\langle\mu,c\rangle[c]
  +\sum_i a_i\langle\mu,v_i\rangle[v_i+c].
\end{equation}
Equivalently,
\begin{equation}
 D_{\mu,a}=\theta_\mu+M_{\theta_\mu F_a},
\end{equation}
where \(M_g\) denotes multiplication by \(g\) and \(\theta_{\mu}\)
denotes the logarithmic derivative.
\end{definition}
In coordinates these operators are
\begin{equation}
 D_{j,a}=D_{\mu_{j},a}=x_j\frac{\partial}{\partial x_j}
       +M_{t x_j\partial f_a/\partial x_j},\quad
 D_{t,a}=t\frac{\partial}{\partial t}+M_{t f_a}.
\end{equation}

\begin{proposition}
The operators \(D_{\mu,a}\) commute.
\end{proposition}
\begin{proof}
Since \([\theta_\mu,M_g]=M_{\theta_\mu g}\), we obtain
\begin{equation}
 [D_{\mu,a},D_{\nu,a}]
 =M_{\theta_\mu\theta_\nu F_a-
       \theta_\nu\theta_\mu F_a}=0.
\end{equation}
\end{proof}

Now we come to the construction of the complex.
Choose dual bases \(\mu_1,\ldots,\mu_{n+1}\) of \(\widetilde{N}\) and
\(\epsilon_1,\ldots,\epsilon_{n+1}\) of \(\widetilde{M}\). Set
\begin{equation}
 \mathcal{L}^p=S\otimes_\mathbf{C}\bigwedge^p\widetilde{M}_{\mathbf{C}}
\end{equation}
and define
\begin{equation}
 \mathrm{d}_a(g\otimes\eta)=
 \sum_{j=1}^{n+1}D_{\mu_j,a}(g)\otimes(\epsilon_j\wedge\eta).
\end{equation}
The differential depends on \(a\).
Nevertheless, the definition is independent of the chosen dual bases and \(\mathrm{d}_a^2=0\).
After identifying elements in \(\widetilde{M}\) with logarithmic one-forms, we see that
\begin{equation}
 \mathrm{d}_a=\mathrm{d}_{\log}+\mathrm{d}_{\log}F_a\wedge.
\end{equation}
Here \(\mathrm{d}_{\log}[c]=[c]\otimes c\). We denote the complex
\((\mathcal{L}^{\bullet},\mathrm{d}_{a})\) simply by \(\mathcal{L}_{a}^{\bullet}\).

Choose a generator of \(\bigwedge^{n+1}\widetilde{M}\). Then the top cohomology is
\begin{equation}\label{eq:Q}
 Q_a:=\mathrm{H}^{n+1}(\mathcal{L}_a^\bullet)
 \cong S\Big/\sum_{j=1}^{n+1}D_{\mu_{j},a}S.
\end{equation}

Under the nondegeneracy hypotheses, 
we shall prove that the complex \(\mathcal{L}_{a}^{\bullet}\)
computes relative cohomology of the pair \((T,Y_{a})\), where
\(Y_{a}\subset T\) is the zero locus of \(f_{a}\). In other words,
there is an isomorphism
\begin{equation}\label{eq:relative}
 Q_a\cong \mathrm{H}^n(T,Y_a;\mathbf{C}).
\end{equation}

Its period description makes the normalization explicit. Fix
the logarithmic top form
\begin{equation}
\Omega=\frac{\mathrm{d}x_1}{x_{1}}\wedge\cdots\wedge \frac{\mathrm{d}x_n}{x_{n}},
\end{equation}
and let
\(\gamma_a\in \mathrm{H}_n(T,Y_a;\mathbf{C})\) 
be a collection of cycles that vary flatly. 
The generator \([0]\) is paired
with \(\gamma_a\) by
\begin{equation}
 \Pi_0(a)=-\int_{\gamma_a}\Omega.
\end{equation}
For \(c=(m,k)\) with \(k>0\), the pairing is
\begin{equation}\label{eq:period}
 \Pi_c(a)=(-1)^{k-1}\Gamma(k)
 \int_{\partial\gamma_a}
 \operatorname{Res}\left(\frac{x^m\Omega}{f_a^k}\right).
\end{equation}
The residue is understood as the cohomological residue
map. These classes enter relative cohomology through the connecting homomorphism
\(\mathrm{H}^{n-1}(Y_a)\to \mathrm{H}^n(T,Y_a)\).

\begin{proposition}
\label{prop:relativederham}
For a nondegenerate Laurent polynomial \(f_{a}\) with full-dimensional Newton
polytope \(\Delta\), the monomial classes in \(Q_a\) define an
isomorphism \(\Theta_a\) in \eqref{eq:relative}. With the normalization constant in 
\eqref{eq:period}, it sends \([0]\) to the relative class
\([(-\Omega,0)]\), and, for \(c=(m,k)\) with \(k>0\), it sends
\([c]\) to the class
\begin{equation}\label{eq:relativeclass}
 \delta_{Y_{a}}\operatorname{Res}_{Y_a}
 \left[(-1)^{k-1}\Gamma(k)\frac{x^m\Omega}{f_a^k}\right].
\end{equation}
Here \(\delta_{Y_{a}}\colon\mathrm{H}^{n-1}(Y_a;\mathbf{C})\to\mathrm{H}^n(T,Y_a;\mathbf{C})\)
is normalized so that its pairing with a relative cycle is the
pairing of the original class with the boundary cycle. 
\end{proposition}
\begin{proof}
We drop the subscript \(a\) in the proof and write \(f=f_a\), \(Y=Y_a\), and \(V=T\setminus Y\). All cohomology groups are over \(\mathbf{C}\).
Set
\begin{equation}
 S_+=\bigoplus_{k\ge1}S_k,~\mbox{where}~
 S_k=\operatorname{Span}_{\mathbf{C}}\{x^mt^k\mid m\in k\Delta\cap M\}.
\end{equation}
Note that \(S_{+}\) is preserved by every \(D_{\mu,a}\). Choose a smooth projective
toric resolution \(X\) of \(\mathbf{P}_{\Delta}\),
and let \(D=X\setminus T\) be the toric boundary. 
Denote by \(E\) the zero divisor of the ball back of the section
\(f\) on \(X\). Nondegeneracy condition implies that \(E\) is smooth and
transverse to the toric strata, so \(D+E\) is a simple normal
crossing divisor.

The divisor \(kE\) is nef, and its section spaces are the lattice-point
spaces of \(k\Delta\). Dividing by the section \(f^k\) gives
an identification
\begin{equation}
 P:=\Gamma(X,\mathcal O_X(\ast E))
 =\varinjlim_{k\ge1}
 \left\{\frac{g}{f^k}~\left|~g\in
  \bigoplus_{m\in k\Delta\cap M}\mathbf{C}x^{m}\right.\right\}.
\end{equation}
Define a linear map \(L\colon S_+\to P\) by
\begin{equation}\label{eq:heightelimination}
 L(x^mt^k)=A_k\frac{x^m}{f^k},~\mbox{where}~
 A_k=(-1)^{k-1}\Gamma(k),
\end{equation}
and extend linearly.
The height relation becomes
\begin{equation}
 D_t(x^mt^k)=kx^mt^k+fx^mt^{k+1}.
\end{equation}
Since \(A_{k+1}=-kA_k\), its image under \(L\) is zero. 

Conversely,
after rescaling degree \(k\) by the constant \(A_k\), these are exactly the
direct-limit relations above. Thus it establises an
isomorphism
\begin{equation}\label{eq:heightcokernel}
 S_+/D_tS_+\cong P.
\end{equation}
The operator \(D_t\) is injective on \(S_+\). To see this, take 
\(h\in \operatorname{Ker}(D_{t})\), the action of \(D_{t}\) on the
highest-degree term in \(h\) is multiplication by \(tf\), which a nonzero element
of the domain \(S\). Therefore that term has to
be zero and this proves the injectivity.

For each torus direction \(j\), direct calculation
gives
\begin{equation}\label{eq:heightderivative}
 \begin{aligned}
 L(D_j(x^mt^k))
 &=A_k\left(m_j\frac{x^m}{f^k}
 -k\frac{x^m\theta_{j}f}{f^{k+1}}\right)\\
 &=\theta_j L(x^mt^k).
 \end{aligned}
\end{equation}
Eliminating the one-term height Koszul direction therefore reduces
the cohomological degree by one and replaces the remaining twisted
operators by the ordinary logarithmic derivatives on \(P\).

The complex
\begin{equation}
 \label{eq:de-rham-e}
 \Omega_X^\bullet(\log D)(\ast E)
\end{equation}
computes \(\mathrm{H}^\bullet(V)=
\mathrm{H}^{\bullet}(T\setminus Y)\), by the logarithmic and meromorphic
de Rham comparisons. On a smooth toric variety,
\begin{equation}
 \Omega_X^p(\log D)(kE)
 \cong\mathcal O_X(kE)\otimes\bigwedge^pM_{\mathbf{C}}.
\end{equation}
Demazure vanishing for the nef divisor \(kE_X\) gives acyclicity in
positive sheaf degrees; see \cite[Theorem~9.2.3]{2011-Cox-Little-Schenck-toric-varieties}. 
Passing to
the direct limit preserves this vanishing. Consequently the
hypercohomology of \eqref{eq:de-rham-e}
is computed by global sections, namely
\((P\otimes\bigwedge^\bullet M_{\mathbf{C}},\mathrm{d}_{\log})\).
In particular, if \(\mathcal K_a^\bullet(S_+)\) denotes the twisted
complex with coefficient module \(S_+\), then
\begin{equation}\label{eq:positivecomplement}
 \begin{aligned}
 \mathrm{H}^{n+1}(\mathcal K_a^\bullet(S_+))
 &=S_+\Big/\bigl(D_tS_++\sum_jD_jS_+\bigr)\\
 &\cong\mathrm{H}^n(V),
 \end{aligned}
\end{equation}
and \([c]\) maps to \([A_kx^m\Omega/f^k]\)
under this isomorphism.

The decomposition
\(S=\mathbf{C}\cdot [0]\oplus S_+\) has the property that
\(D_\mu S\subseteq S_+\), so the class of \([0]\) survives as
a direct summand in the quotient \(Q_a\). The only extra relations in the
positive-height quotient are those coming from
\begin{equation}
 D_t(1)=tf,~\mbox{and}~D_j(1)=t\theta_jf.
\end{equation}
Under \eqref{eq:positivecomplement}, these become respectively
\([\Omega]\) and \([(\theta_jf/f)\Omega]\). Hence
\begin{equation}\label{eq:Qcomplementquotient}
 Q_a\cong\mathbf{C}\cdot [0]\oplus
 \left(\mathrm{H}^{n}(V)\slash 
 \left\langle[\Omega],[(\theta_jf/f)\Omega]\mid 1\le j\le n\right\rangle.\right)
\end{equation}
Now the Gysin sequence of the smooth divisor \(Y\subset T\) contains
\begin{equation}
 \mathrm{H}^n(T)\rightarrow\mathrm{H}^n(V)
 \xrightarrow{\operatorname{Res}_Y}\mathrm{H}^{n-1}(Y)
 \rightarrow\mathrm{H}^{n+1}(T)=0.
\end{equation}
Thus the residue map 
\(\operatorname{Res}_{Y}\) is surjective and its kernel is the image of
\(\mathrm{H}^n(T)=\mathbf{C}[\Omega]\).
Let \(\eta_j=\iota_{\theta_j}\Omega\). Then
\begin{equation}
 \frac{\theta_jf}{f}\Omega
 =\frac{\mathrm{d}f}{f}\wedge\eta_j,~\mbox{and}~
 \operatorname{Res}_Y\left[\frac{\theta_jf}{f}\Omega\right]
 =i^{\ast}[\eta_j],
\end{equation}
where \(i\colon Y\hookrightarrow T\) is the inclusion. 
We thus conclude that the classes \([\eta_j]\)
span \(\mathrm{H}^{n-1}(T)\) (including the case \(n=1\)).
Equation~\eqref{eq:Qcomplementquotient} therefore becomes
\begin{equation}\label{eq:Qprimitiveextension}
 Q_a\cong\mathbf{C}[0]\oplus
 \operatorname{Coker}\bigl(\mathrm{H}^{n-1}(T)
 \xrightarrow{i^{\ast}}\mathrm{H}^{n-1}(Y)\bigr).
\end{equation}

The smooth affine variety \(Y\) has complex dimension \(n-1\),
so by the affine vanishing theorem \(\mathrm{H}^n(Y)=0\). 
The relative exact sequence consequently
gives
\begin{equation}\label{eq:relativeextension}
 0\rightarrow\operatorname{Coker}i^{\ast}
 \xrightarrow{\delta}\mathrm{H}^n(T,Y)
 \rightarrow\mathrm{H}^n(T)\rightarrow0.
\end{equation}
The relative de Rham class \([(-\Omega,0)]\) maps to the generator
\(-[\Omega]\), since the restriction of 
\(\Omega\) to \(Y\) is zero. 
To construct a natural isomorphism, we 
define a morphism
\begin{equation}
\begin{aligned}
S=\mathbf{C}\oplus S_{+}&\to \mathrm{H}^{n}(T,Y)~\mbox{via}~\\
([0],[c])&\mapsto 
\left(-\Omega,\operatorname{Res}_{Y}
 \left[(-1)^{k-1}\Gamma(k)\frac{x^m\Omega}{f^k}\right]\right)~\mbox{for}~c=(m,k).
\end{aligned}
\end{equation}
Sending \([0]\) to \(-\Omega\) and
using \(\delta_{Y}\operatorname{Res}_Y\) on the second summand of
\eqref{eq:Qcomplementquotient} identifies
\eqref{eq:Qprimitiveextension} with \eqref{eq:relativeextension};
recall that 
\begin{equation}
  \delta_{Y}[\eta]=[(0,\eta)].
\end{equation}
This proves the asserted isomorphism and the period normalization.
\end{proof}

Filter \(S\) by homogenizing degree induced by \(t\) and give \(Q_a\) the quotient
filtration. 
The filtered complex \(F_{q}\mathcal{L}^{\bullet}_{a}\) is 
\begin{equation}
  F_{q}\mathcal{L}_{a}^{p}:=F_{p+q}S\otimes \bigwedge^{p} \widetilde{M}.
\end{equation}
The leading parts of the operators are multiplication by
\begin{equation}
 G_{\mu,a}=\sum_i a_i\langle\mu,v_i\rangle[v_i]
\end{equation}
and hence the graded complex is nothing but 
the Koszul complex
\begin{equation}
  K^{\bullet}(G_{1},\ldots,G_{\ell};S).
\end{equation}
Define the logarithmic Jacobian ring
\begin{equation}
 R_a:=S/(G_{\mu_1,a},\ldots,G_{\mu_{n},a},G_{\mu_{t},a}).
\end{equation}
The standard nondegenerate comparison gives an
isomorphism
\begin{equation}
 \operatorname{gr} Q_a\cong R_a
\end{equation}
between vector spaces.
\begin{remark}
The cohomology \(\mathrm{H}^{n}(T,Y_{a};\mathbf{C})\)
is connected to the solutions to 
the full GKZ system describing periods
for CY hypersurfaces in toric varieties.
In fact, there is a duality isomorphism
\begin{equation}
\label{eq:coh-hom-sol}
  \mathrm{H}^{n}(T,Y_{a};\mathbf{C})\cong 
  \mathrm{H}_{n}(U_{a},U_{a}\cap D;\mathbf{C})
\end{equation}
where \(U_{a}:=Z\setminus \overline{Y}_{a}\)
is the complement of the closure of the 
affine hypersurface \(Y_{a}\) inside 
the ambient toric manifold \(Z\) and 
\(D:=Z\setminus T\) is the toric boundary.
The right hand side of \eqref{eq:coh-hom-sol}
is shown to be the solution space
\cite{2020-Lee-Zhang-a-hypergeometric-systems-and-relative-cohomology}
when the columns of \(A\) generate the full lattice.
\end{remark}
\section{The toric fan endpoint}
\label{sec:the_toric_fan_endpoint}
Let \(\check{\Sigma}\) be the fan in \(C\) obtained by homogenizing the chosen star triangulation. By a star triangularization we mean that
every maximal cone contains the vector \((0,1)\).
On the monomial space \(S\) we define
another multiplication by
\begin{equation}
 [c]\ast_{\check\Sigma}[c']
 =\begin{cases}
 [c+c'],&c,c'\text{ belong to a common cone of }\check{\Sigma},\\
 0,&\text{otherwise}.
 \end{cases}
\end{equation}
Write \(S_{\check{\Sigma}}\) for this partial semigroup algebra. 
It is clear that the multiplication
is associative. Define for \(\mu\in \tilde{N}\) an operator
\begin{equation}
 D_{\mu,a}^{\check{\Sigma}}(g)
 =\left(\sum_i a_i\langle\mu,v_i\rangle[v_i]\right)\ast_{\check{\Sigma}} g
 =G_{\mu,a}\ast_{\check{\Sigma}} g.
\end{equation}
For clarity, we will index the rays by \(0\) and the boundary rays \(\rho\).
More concretely,
\begin{equation}
 v_0=\delta=(0,1),~\mbox{and}~v_\rho=(\rho,1)~\mbox{for}~\rho\in\Delta\cap M.
\end{equation}
Let \(B_{\check{\Sigma}}\subset S_{\check{\Sigma}}\) be the subalgebra generated by
the ray monomials \(u_0=[\delta]\) and \(u_\rho=[v_\rho]\).
Simpliciality gives
\begin{equation}
 B_{\check{\Sigma}}\cong
 \mathbf{C}[u_0,u_\rho]/I_{\mathrm{SR}}(\check \Sigma).
\end{equation}
Here \(I_{\mathrm{SR}}(\check \Sigma)\) 
is the Stanly--Reisner ideal for \(\check{\Sigma}\);
it is generated by ``non-cones'' in \(\check{\Sigma}\).

The central variable \(u_0\) occurs in no Stanley--Reisner generator.
In addition, we assume the ray coefficients are nonzero and put \(z_i=a_i u_i\).
The linear relations imposed by the \(D_{\mu,a}^{\check{\Sigma}}\) are
\begin{equation}
 z_0+\sum_\rho z_\rho=0,~\mbox{and}~
 \sum_\rho\langle\nu,\rho\rangle z_\rho=0~(\nu\in N).
\end{equation}
Eliminating \(z_0\) in the ray-generated subalgebra gives
an isomorphism
\begin{equation}\label{eq:ordinarytoric}
 B_{\check{\Sigma}}/(G_{\mu,a})
 \cong \mathrm{H}^{\mathrm{even}}(\check Z;\mathbf{C}),~\mbox{via}~
 z_\rho\mapsto D_\rho.
\end{equation}
Now we describe the orbifold data.
For a simplicial cone \(\sigma\in\check{\Sigma}\), define
\begin{equation}\label{eq:box}
 \operatorname{Box}(\sigma)=
 \left\{\sum_{v_i\in\sigma}\lambda_i v_i\in\widetilde M\mid 
 0\le\lambda_i<1\right\}
\end{equation}
and put
\begin{equation}
  \operatorname{Box}(\check{\Sigma})=\bigcup_{\sigma\in\check{\Sigma}}
  \operatorname{Box}(\sigma).
\end{equation}
Every \(c\in\sigma\cap\widetilde M\) admits a unique decomposition
\begin{equation}\label{eq:boxdecomposition}
 c=b+\sum_{v_i\in\sigma}m_i v_i,~\mbox{with}~
 b\in\operatorname{Box}(\sigma)~\mbox{and}~m_i\in\mathbf{Z}_{\ge0}.
\end{equation}
Nonzero box elements supply additional monomials in \(S_\Sigma\).

The coefficients \(a_{i}\) can be absorbed or normalized on the full algebra as well. Choose
\(\ell_i\in\mathbf{C}\) with \(\exp(\ell_i)=a_i\). For the unique expansion
\(c=\sum_i\lambda_i v_i\) in its smallest cone, set
\begin{equation}
 z(c)=\exp\left(\sum_i\lambda_i\ell_i\right)[c].
\end{equation}
These rescalings agree on common faces and are multiplicative within
each cone, hence define an algebra automorphism. In particular, we have
\(z(v_i)=a_i[v_i]\).
With this normalization, we obtain
\begin{equation}\label{eq:toric}
 S_{\check{\Sigma}}/(G_{\mu,a})
 \cong\mathrm{H}_{\mathrm{CR}}^\bullet(\check{\mathcal Z};\mathbf{C})~\mbox{via}~z_\rho\mapsto D_\rho.
\end{equation}
Here \(\check{\mathcal{Z}}\) is the toric
Deligne--Mumford stack associated to the fan obtained
by projecting \(\check{\Sigma}\) to \(M_{\mathbf{R}}\).
Since \(\check{\Sigma}\) is simplicial, the projected
fan defines 
a simplicial toric variety, or a smooth 
toric Deligne--Mumford stack. 
We denote the 
coarse moduli space by \(\check{Z}\). 
This is the presentation of toric orbifold cohomology;
see \cite{2005-Borisov-Chen-Smith-the-orbifold-chow-ring-of-toric-deligne-mumford-stacks}*{Theorem~1.1}
and \cite{2013-Horja-toric-deligne-mumford-stacks-and-the-better-behaved-version-of-the-gkz-hypergeometric-system}*{\S2 and \S4}.
Semigroup degree \(k\) corresponds to Chen--Ruan degree \(2k\).
The quotient in \eqref{eq:ordinarytoric} embeds as the untwisted sector in the
stack \(\check{\mathcal{Z}}\);
the box decomposition supplies the twisted sectors.

To relate the fan \(\check{\Sigma}\) over 
\(\Delta\times\{1\}\) to the projected fan, we put
\begin{equation}
h_\Delta(m)=\min\{r\ge0\mid m\in r\Delta\}.  
\end{equation}
That \(\Delta\) is reflexive implies that
\(h_\Delta(m)\in\mathbf{Z}\) for \(m\in M\), and it is linear on each
cone of \(\check\Sigma\). If \(\mathbf{C}[M]^{\check\Sigma}\) denotes the
partial group algebra of the projected
complete fan, there is an isomorphism
\begin{equation}
 S_{\check{\Sigma}}\cong\mathbf{C}[u_0]\otimes\mathbf{C}[M]^{\check\Sigma}
 ~\mbox{via}~
 [(m,k)]\mapsto u_0^{k-h_\Delta(m)}[m].
\end{equation}
Thus eliminating the central variable reduces the full quotient to
the usual orbifold presentation on \(\check\Sigma\). Geometrically,
\(\check\Sigma\) describes \(\operatorname{Tot}(K_{\check Z})\), with the corresponding
stack \(\operatorname{Tot}(K_{\check{\mathcal Z}})\). Local stabilizers act trivially  
on the fibers because the coarse space is Gorenstein.

\begin{remark}[Ordinary versus orbifold cohomology]
Writing \(\operatorname{mult}(\sigma)\) for the lattice index of a
maximal cone of \(\check\Sigma\), we have
\begin{equation}\label{eq:rankobstruction}
 \begin{aligned}
 \dim_\mathbf{C}\mathrm{H}^n(T,Y_a)
 &=\operatorname{Vol}_M(\Delta)
   =\sum_{\sigma\in\check\Sigma(n)}\operatorname{mult}(\sigma),\\
 \dim_\mathbf{C}\mathrm{H}^{\mathrm{even}}(\check Z)
 &=\#\check\Sigma(n).
 \end{aligned}
\end{equation}
The second equality uses vanishing of odd rational cohomology and the
toric orbit decomposition. Therefore the full relative space can have
the dimension of ordinary toric cohomology only when all maximal
cones are unimodular. The orbifold endpoint has the required full rank.
\end{remark}

\section{An interpolation complex \texorpdfstring{\(\mathcal{K}^{\bullet}\)}{K}}
\label{sec:VHS-complex}
In this section, we will connect
\(\mathrm{H}^{\mathrm{\bullet}}_{\mathrm{CR}}(\check{\mathcal{Z}};\mathbf{C})\)
with the relative cohomology \(\mathrm{H}^{n}(T,Y_{a};\mathbf{C})\)
by constructing a global complex.
Choose an integral, positively homogeneous, convex piecewise-linear
function \(\psi\colon C\to\mathbf{R}\), with linearity cones exactly \(\check\Sigma\).
For \(c,c'\in C\cap\widetilde M\), we put
\begin{equation}
 b(c,c'):=\psi(c)+\psi(c')-\psi(c+c')\in\mathbf{Z}_{\ge0}.
\end{equation}
On \(\mathcal{S}=\bigoplus_{c\in C\cap\widetilde M}\mathbf{C}[s][c]\), we set
\begin{equation}\label{eq:product}
 [c]\ast_s[c']:=s^{b(c,c')}[c+c'].
\end{equation}
The product \(\ast_{s}\) on \(\mathscr{S}\) is associative as it follows from
\begin{equation}
 b(c,c')+b(c+c',c'')
 =b(c',c'')+b(c,c'+c'').
\end{equation}
Thus \(\mathcal{S}\) is a free \(\mathbf{C}[s]\)-algebra connecting two relevant
spaces in question, i.e.~
\(\mathcal{S}/(s-1)\cong S\) and \(\mathcal{S}/(s)\cong S_{\check \Sigma}\).
Now we define
\begin{equation}\label{eq:familyD}
 D_{\mu,a,s}[c]=D_{\mu,s}[c]
 =s\langle\mu,c\rangle[c]
 +\sum_i a_i\langle\mu,v_i\rangle[v_i]\ast_s[c].
\end{equation}
We will omit the subscript \(a\) if it causes no confusion.

The Euler operators are derivations of \(\ast_s\). If \(M_g^{(s)}\)
denotes multiplication by \(g\in\mathscr{S}\) in this algebra, then
\begin{equation}
 D_{\mu,s}=s\theta_\mu+M_{G_{\mu,a}}^{(s)},~\mbox{and}~
 [D_{\mu,s},D_{\nu,s}]
 =sM_{\theta_\mu G_{\nu,a}-\theta_\nu G_{\mu,a}}^{(s)}=0.
\end{equation}
Hence these operators define a \(\mathbf{C}[s]\)-linear Koszul complex
\(\mathcal{K}^\bullet\). Its specialized complexes have top cohomology
\begin{equation}
 \begin{aligned}
 \mathrm{H}^{n+1}(\mathcal{K}^\bullet|_{s=1})
   &\cong \mathrm{H}^n(T,Y_a;\mathbf{C}),\\
 \mathrm{H}^{n+1}(\mathcal{K}^\bullet|_{s=0})
   &\cong \mathrm{H}_{\mathrm{CR}}^\bullet(\check{\mathcal Z};\mathbf{C}).
 \end{aligned}
\end{equation}

\section{Cohomology of \texorpdfstring{\(\mathcal{K}^{\bullet}\)}{K}}
Let \(\ell=n+1\) be the rank of \(\widetilde{N}\). In what follows, we
fix a basis \(\mu_1,\ldots,\mu_\ell\) of \(\widetilde N\)
and write \(G_j=\sum_i a_i\langle\mu_j,v_i\rangle[v_i]\) for brevity. We
assume that the parameter \(s\) carries no weight, and monomials have their height degree.
Following the finite-complex argument in Hartshorne, 
fiberwise exactness of each graded strand below degree \(\ell\) implies that its top 
cohomology is locally free and commutes with base change.

\begin{theorem}\label{thm:comparison}
Under the hypotheses of \S\ref{sec:preliminaries}, there is a \(g\in\mathbf{C}[s]\)
with \(g(0)g(1)\ne0\) such that, for
\begin{equation}
 A=\mathbf{C}[s,g^{-1}],~
 \mathcal{K}_A^\bullet=\mathcal{K}^\bullet\otimes_{\mathbf{C}[s]}A,
 ~\mbox{and}~\mathcal{H}=\mathrm{H}^\ell(\mathcal{K}_A^\bullet),
\end{equation}
the following hold.
\begin{enumerate}
\item \(\mathrm{H}^j(\mathcal{K}_A^\bullet)=0\) for \(j\ne\ell\).
\item \(\mathcal{H}\) is finite free over \(A\), of rank
\(\operatorname{Vol}_M(\Delta)
=\dim_\mathbf{C} \mathrm{H}_{\mathrm{CR}}^\bullet(\check{\mathcal Z};\mathbf{C})\).
Here we use the normalized volume.
\item For every \(b\in\mathbf{C}\) with \(g(b)\ne0\), specialization gives
\begin{equation}\label{eq:basechange}
 \mathcal{H}/(s-b)\mathcal{H}
 \cong \mathrm{H}^\ell(\mathcal{K}^\bullet|_{s=b}).
\end{equation}
\item A basis of \(\mathcal{H}\) adapted to its degree filtration gives
a noncanonical filtered vector-space isomorphism
\begin{equation}\label{eq:endpointcomparison}
 \mathrm{H}^n(T,Y_a;\mathbf{C})\cong \mathrm{H}_{\mathrm{CR}}^\bullet(\check{\mathcal Z};\mathbf{C}),
\end{equation}
where the filtration on the right is
\(F_k=\bigoplus_{j\le k}\mathrm{H}_{\mathrm{CR}}^{2j}(\check{\mathcal Z};\mathbf{C})\).
\end{enumerate}
\end{theorem}

\begin{proof}
Apply the box decomposition \eqref{eq:boxdecomposition} in each
maximal cone. Multiplication within a cone has exponent \(b(c,c')=0\),
so every basis monomial of \(\mathcal{S}\) is a product of ray
monomials and one of finitely many box monomials. Thus \(\mathcal{S}\)
is a finitely generated positively graded \(\mathbf{C}[s]\)-algebra.
Each graded piece is still
finite free over \(\mathbf{C}[s]\), since its basis is indexed by
\(k\Delta\cap M\).

For \(b\ne0\), the map
\begin{equation}
 \mathcal{S}/(s-b)\rightarrow S~\mbox{defined by}~
 [c]\mapsto b^{\psi(c)}[c]
\end{equation}
is an algebra isomorphism. These fibers are Cohen--Macaulay because
\(S\) is a normal affine semigroup algebra. At the point \(b=0\),
the full partial semigroup ring \(S_{\check\Sigma}\) is also Cohen--Macaulay and 
of dimension \(\ell\). This follows from the Cohen--Macaulay result for the
partial semigroup ring of a coherent simplicial subdivision of a
convex cone with parameter zero.
Moreover, the
unweighted ray linear forms form a regular sequence. The coefficient
rescaling in \S\ref{sec:the_toric_fan_endpoint} 
transfers this statement to the forms \(G_j\).
See \cite{2003-Borisov-Mavlyutov-string-cohomology-of-calabi-yau-hypersurfaces-via-mirror-symmetry}*{Theorem 6.3 and Lemma 6.4}
for relevant detail.
All fibers have dimension \(\ell\). Here we write \(G_{\ell}\equiv G_{t}=G_{t,a}\).

Consider the closed projective subscheme
\begin{equation}
 \mathcal V=\operatorname{Proj}_{\mathbf{C}[s]}
 \bigl(\mathcal{S}/(G_1,\ldots,G_\ell)\bigr)\rightarrow
 \operatorname{Proj}_{\mathbf{C}[s]}
 \mathcal{S}.
\end{equation}
It is projective over \(\operatorname{Spec}\mathbf{C}[s]\).
and its image in
\(\operatorname{Spec}\mathbf{C}[s]\) is therefore closed.
This image misses both \(0\) and \(1\), for the
quotient at \(0\) is the finite-dimensional toric orbifold cohomology ring, and
the quotient at \(1\) is the finite-dimensional nondegenerate logarithmic
Jacobian ring. The image is thus a finite set. Choose \(g\in\mathbf{C}[s]\) vanishing
on that set and nowhere at \(0,1\).

Put \(A=\mathbf{C}[s]_{g}=\mathbf{C}[s,g^{-1}]\).
On every fiber over \(\operatorname{Spec}A\), the \(G_j\) are a
homogeneous system of parameters in a Cohen--Macaulay algebra. They
form a regular sequence.

Base change to \(A\) and set
\begin{equation}
 \mathcal{R}_A=\mathcal{S}_A/(G_1,\ldots,G_\ell),~\mbox{and}~
 \mathcal{S}_A=\mathcal{S}\otimes_{\mathbf{C}[s]}A.
\end{equation}
Fix an integer \(q\), and consider the complex
\begin{equation}
 E^p=(\mathcal{S}_A)_{q+p}\otimes\bigwedge^p\widetilde M_\mathbf{C},
 \quad 0\le p\le\ell.
\end{equation}
A graded piece with negative index is zero be convention. The
multiplication by \(G_{j}\) gives the differential, denoted by \(\mathrm{d}_G\).
We obtain a bounded complex
\begin{equation}\label{eq:strand}
 0\longrightarrow E^0\xrightarrow{\mathrm{d}_G^0}E^1
 \longrightarrow\cdots\longrightarrow E^{\ell-1}
 \xrightarrow{\mathrm{d}_G^{\ell-1}}E^\ell
 \longrightarrow0
\end{equation}
of finite free \(A\)-modules. Its last cokernel is isomorphic to
\((\mathcal{R}_A)_{q+\ell}\).
For every maximal ideal \(\mathfrak m=(s-b)\) of \(A\), we already proved
that \eqref{eq:strand} tensored with
\(\kappa(\mathfrak m)=\mathbf{C}\) is exact in degrees below \(\ell\).
We now explain why this fiberwise assertion gives the required assertion
over \(A\).

Localize at \(\mathfrak m\), and write \(B=A_{\mathfrak m}\) and
\(E_B^p=E^p\otimes_A B\).
The map
\(\mathrm{d}_G^0\otimes_B\kappa(\mathfrak m)\) is injective.
If \(r=\operatorname{rank}_B E_B^0>0\), its matrix therefore has
an \(r\times r\) minor nonzero modulo \(\mathfrak mB\). That
determinant is a unit of the local ring \(B\). After invertible
changes of bases, the matrix over \(B\) has the form
\begin{equation}
 \mathrm{d}_G^0=
 \begin{bmatrix} I_r\\0\end{bmatrix}.
\end{equation}
Consequently \(\mathrm{d}_G^0\) is
a split injection and
\begin{equation}
 E_B^1=\mathrm{d}_G^0(E_B^0)\oplus C^1
\end{equation}
for a finite free \(B\)-module \(C^1\).
If \(r=0\), take \(C^1=E_B^1\), and the same conclusion holds.
The identity \(\mathrm{d}_G^1\mathrm{d}_G^0=0\)
ensures that \(\mathrm{d}_G^1\) factors through \(C^1\).
Because the displayed decomposition is split, it remains a direct-sum
decomposition after tensoring with \(\kappa(\mathfrak m)\).
Fiberwise exactness at \(E^1\) now says that the induced map
\begin{equation}
 C^1\otimes_B\kappa(\mathfrak m)
 \longrightarrow E_B^2\otimes_B\kappa(\mathfrak m)
\end{equation}
is injective.

Set \(C^0=E_B^0\). Suppose the earlier terms have been split off, so
that the remaining complex begins
\begin{equation}
 0\longrightarrow C^p\longrightarrow E_B^{p+1}
 \longrightarrow\cdots\longrightarrow E_B^\ell\longrightarrow0,
\end{equation}
where \(C^p\) is finite free and the first arrow is injective on the
residue-field fiber. The same argument splits this arrow
and produces a finite free cokernel \(C^{p+1}\), i.e.,
\begin{equation}\label{eq:splitinduction}
 0\longrightarrow C^p\longrightarrow E_B^{p+1}
 \longrightarrow C^{p+1}\longrightarrow0.
\end{equation}
For \(p+1<\ell\), fiberwise exactness at \(E^{p+1}\), together with
the splitting, makes the induced map from \(C^{p+1}\) to
\(E_B^{p+2}\) injective on the fiber. This is the induction hypothesis
for the next step. At \(p=\ell-1\), it proves that the last cokernel
\begin{equation}
 C^\ell
 \cong(\mathcal{R}_A)_{q+\ell}\otimes_A B
\end{equation}
is finite free over \(B\). 

The argument works at every maximal ideal of \(A\). Since all terms
in a fixed strand are finite modules, its cohomology vanishes below
degree \(\ell\) if and only if it vanishes after these localizations.
Its last cokernel is finitely presented and locally free, hence is
a finite projective \(A\)-module. The local splittings
\eqref{eq:splitinduction} also show that tensoring a strand with any
\(A\)-module preserves exactness below its last term. In particular,
for every \(b\) with \(g(b)\ne0\) and every \(k\),
\begin{equation}\label{eq:gradedbasechange}
 (\mathcal{R}_A)_k\otimes_A A/(s-b)
 \cong
 \bigl(\mathcal{S}_b/(G_{1,b},\ldots,G_{\ell,b})\bigr)_k
\end{equation}
and no lower cohomology appears in the complex associated
with the index \(q=k-\ell\). Varying \(q\) proves exactness below
degree \(\ell\) for the entire graded multiplication Koszul complex.

Write \(\mathcal{R}_b=\mathcal{S}_b/(G_{1,b},\ldots,G_{\ell,b})\) and
\(L_b^{(j)}=\mathcal{S}_b/(G_{1,b},\ldots,G_{j,b})\), with
\(L_b^{(0)}=\mathcal{S}_b\). In each fiber the degree-one regular
sequence gives for each \(1\le j\le\ell\) an exact sequence
\begin{equation}
 0\longrightarrow L_b^{(j-1)}[-1]
 \xrightarrow{\ G_{j,b}\ }L_b^{(j-1)}
 \longrightarrow L_b^{(j)}\longrightarrow0.
\end{equation}
Here \(L[-1]_k=L_{k-1}\) is the shifted complex. Induction on \(j\) therefore multiplies
the Hilbert series by \(1-T\) at each step. Since every fiber
\(\mathcal{S}_b\) has the same lattice-point basis as \(S\), this gives
\begin{equation}\label{eq:hilbert}
 \begin{aligned}
 \operatorname{Hilb}(\mathcal{R}_b;T)
 &=(1-T)^\ell\operatorname{Hilb}(S;T)
   =h^{\ast}_\Delta(T)\\
 &=\sum_{k=0}^n
   \dim_\mathbf{C}\mathrm{H}_{\mathrm{CR}}^{2k}(\check{\mathcal Z};\mathbf{C})T^k.
 \end{aligned}
\end{equation}
The last equality is obtained by evaluating at \(b=0\).
For \(k>n\), \eqref{eq:gradedbasechange} says that the finite
\(A\)-module \((\mathcal{R}_A)_k\) has zero fiber at every maximal
ideal. Nakayama's lemma at each such ideal implies
\((\mathcal{R}_A)_k=0\).
Finally, \(A\) is a PID, so
\begin{equation}
 \mathcal{R}_A=\bigoplus_{k=0}^n(\mathcal{R}_A)_k
\end{equation}
is finite free whose degree-\(k\) rank is the coefficient of \(T^k\)
in \(h^{\ast}_\Delta(T)\). 

Let \(F_k\mathcal{S}_A\) be the span of monomials of height at most
\(k\), with the convention \(F_k=0\) for \(k<0\). Filter the twisted complex by
\begin{equation}\label{eq:shiftedfiltration}
 F_q\mathcal{K}_A^p
 =F_{q+p}\mathcal{S}_A\otimes\bigwedge^p\widetilde M_\mathbf{C}.
\end{equation}
Multiplication by \(G_j\) preserves this shifted filtration degree,
whereas \(s\theta_{j}\) lowers the shifted filtration degree by one. Thus
in the graded complex the differential \(D_{\mu,a}\) becomes \(G_{\mu,a}\) and 
\begin{equation}
 \operatorname{gr}_F\mathcal{K}_A^\bullet
 =K^\bullet(G_1,\ldots,G_\ell;\mathcal{S}_A).
\end{equation}
Its cohomology is concentrated at degree \(\ell\).
For a cocycle in a
cohomological degree below \(\ell\), its highest filtered part is
exact in the graded complex. Lift a graded primitive and subtract its
differential. The remainder has smaller filtered degree. The iteration
terminates because the filtration is bounded below and each element
has finite degree; hence the original cocycle is exact.

Suppose
\(\mathrm{d}\eta\in F_q\mathcal{K}_A^\ell\), with
\(\eta\in F_r\mathcal{K}_A^{\ell-1}\) and \(r>q\).
Its degree-\(r\) symbol is a cocycle in the graded complex at
degree \(\ell-1\). Graded exactness gives a lift
\(\theta\in F_r\mathcal{K}_A^{\ell-2}\) such that
\(\eta-\mathrm{d}\theta\) has smaller filtered degree and the same
differential. Repeating gives a primitive in
\(F_q\mathcal{K}_A^{\ell-1}\). Thus the last differential is strict, i.e.
\begin{equation}
 \operatorname{im}(\mathrm{d})\cap F_q\mathcal{K}_A^\ell
 =\mathrm{d}(F_q\mathcal{K}_A^{\ell-1}).
\end{equation}
Consequently taking the quotient by boundaries commutes with passing
to the associated graded. Together these arguments prove the vanishing
of the cohomology
\begin{equation}
 \mathrm{H}^j(\mathcal{K}_A^\bullet)=0\ (j\ne\ell)
\end{equation}
and 
\begin{equation}
 \operatorname{Gr}_{\deg}\mathcal{H}\cong\mathcal{R}_A.
\end{equation}
Here the top filtration is reindexed from \(q\) to \(q+\ell\), so
that it is the quotient monomial-degree filtration. Its successive
quotients are finite free and vanish outside \(0,\ldots,n\).
The resulting finite filtration splits over \(A\); hence \(\mathcal{H}\)
is finite free.
\end{proof}

\begin{remark}
The theorem is a comparison of filtered complex vector spaces at a
fixed coefficient vector \(a\). It relies on the relative de Rham
comparison of \S\ref{sec:relative_cohomology_and_the_associated_graded_space} 
and the toric orbifold presentation of
\S\ref{sec:the_toric_fan_endpoint} as inputs. 
It is clear that the isomorphism given in
Theorem \ref{thm:comparison} (4) is not 
canonical.
\end{remark}

\begin{remark}
We can bypass the use of the interpolation complex 
\(\mathcal{K}^{\bullet}\) and Theorem \ref{thm:comparison} for the
proof of the hyperplane conjecture; the only result
we need is the equality
\begin{equation}
  \dim \mathrm{H}^{n}(T,Y_{a};\mathbf{C}) = \dim \mathrm{H}^{\bullet}_{\mathrm{CR}}(\check{\mathcal{Z}};\mathbf{C}).
\end{equation}
However, we still include it because \(\mathcal{K}^{\bullet}\)
provides a bridge to connection the residue of the Gauss--Manin
connection with the cup product morphism on the \(A\)-side
(cf.~Theorem~\ref{thm:monodromy-residue-and-cup-prod}).
\end{remark}

\section{Proof of the hyperplane conjecture}
\label{sec:proof_of_the_hyperplane_conjecture}
We will prove Theorem \ref{thm:main-1-intro} and Theorem \ref{thm:main-2-intro}
(and hence Conjecture \ref{conj:HLY}) in this section.

For clarity first we assume that \(\check Z\) is smooth, and put
\(R=\mathrm{H}^\bullet(\check Z;\mathbf{C})\). 
(We will remove the assumption later.)
Let us recall the definition of the \(B\)-series for 
CY hypersurfaces in \(Z\).
Let \(\Delta\subset M_{\mathbf{R}}\) be a reflexive polytope 
and \(\nabla=\Delta^{\vee}\subset N_{\mathbf{R}}\) be its dual. Put
\(\Delta\cap M=\{m_{0}=\mathbf{0},m_{1},\ldots,m_{p}\}\).
Each non-zero integral point \(m_{i}\)
corresponds to a \(1\)-cone in a MPCP desingularization
\(\check{Z}\to\mathbf{P}_{\nabla}\) whose
divisor class in \(R\) is denoted by \(D_{i}\).
Let \(A\) be the integral matrix
\begin{equation}
  A:=
  \begin{bmatrix}
    1 & 1 & \cdots & 1\\
    \mathbf{0} & m_{1} & \cdots & m_{p}
  \end{bmatrix}\in\mathrm{Mat}_{(n+1)\times (p+1)}(\mathbf{Z})
\end{equation}
and 
\begin{equation}
  \beta:=\begin{bmatrix}
  -1\\
  \mathbf{0}
  \end{bmatrix}\in \mathbf{C}^{n+1}.
\end{equation}
As observed by Batyrev, period integrals for
CY hypersurfaces in \(Z\) are governed by the GKZ
system \(M_{A}(\beta)\). 
Let \(L:=\operatorname{Ker}(A)\). 
Inspired by mirror symmetry,
Hosono, Lian, and Yau introduced
\begin{equation}\label{eq:HLYrenormalizeddelta}
 \mathcal B(a)
 =\frac{1}{a_0}\sum_{\substack{\ell\in L,\ell_0\le0}}(-1)^{\ell_0}
 \frac{\Gamma(1-\ell_0-D_0)}
      {\prod_{i>0}\Gamma(1+\ell_i+D_i)}a^{\ell+D}.
\end{equation}
This is a cohomology-valued series whose coefficients are 
(formal) functions in \(a\) which are convergent near
a large complex structure limit (LCSL) point.

Let \(\mathfrak{Sol}\) be the local solution space for 
the GKZ system \(M_{A}(\beta)\).
The Poincar\'{e} pairing turns \(\mathcal{B}(a)\) into a solution-valued
functional
\begin{equation}
\label{eq:Bfunctional}
 \mathcal P_{\mathcal{B}(a)}\colon R\rightarrow\mathfrak{Sol}~\mbox{via}~
 \alpha\mapsto\int_{\check Z}\alpha\cup\mathcal B(a).
\end{equation}
Recall that the space of geometric cycles considered in this paper is defined by
\begin{equation}
   W=\operatorname{Im}\left(
 \mathrm{H}_n(Z,E_{a})\xrightarrow{\partial}\mathrm{H}_{n-1}(E_{a})
 \xrightarrow{\operatorname{Tube}}\mathrm{H}_n(U_{a})
 \xrightarrow{q_{a}}\mathrm{H}_n(U_{a},B_{a})\right).
\end{equation}
\begin{proposition}\label{prop:Bimagecriterion}
Cup product on the series corresponds to precomposition on its
functional:
\begin{equation}\label{eq:Bprecomposition}
 \mathcal P_{e\,\cup\,\mathcal B}
 =\mathcal P_{\mathcal B}\circ(e\cup-).
\end{equation}
Consequently the span of the scalar components of
\(e\cup\mathcal{B}\) is \(\mathcal{P}_{\mathcal B}(e\cup R)\).
Let \(\mathfrak{Per}(W)\subset\mathfrak{Sol}\) 
denote the geometric period map with the
same parameter and normalization conventions. Whenever
\(\mathcal P_{\mathcal B}\) is injective and
\(\dim\mathfrak{Per}(W)=\operatorname{rank}(e\cup-)\), the inclusion
\begin{equation}\label{eq:Bremaininginclusion}
 \mathcal P_{\mathcal B}(e\cup R)\subseteq\mathfrak{Per}(W)
\end{equation}
implies equality. In particular, the dimension condition holds if
\(\dim W=\operatorname{rank}(e\cup-)\) and the period map is injective
on \(W\).
\end{proposition}
\begin{remark}
This is the strategy adopted by Lian and Zhu for \(\mathbf{P}^{n}\) case
\cite {2021-Lian-Zhu-on-the-hyperplane-conjecture-for-periods-of-calabi-yau-hypersurfaces-in-pn}.
In the \(\mathbf{P}^{n}\) situation, we have \(W=\mathrm{H}_{n}(U_{a};\mathbf{C})\)
and the period sheaf \(\mathfrak{Per}(W)\) is completely characterized 
by the tautological system \cites{2014-Bloch-Huang-Lian-Srinivas-Yau-on-the-holonomic-rank-problem,2016-Huang-Lian-Zhu-period-integrals-and-riemann-hilbert-correspondence}.
Hence it is possible to verify the inclusion.
Then Lian and Zhu subsequently computed the dimension rank of \(e\cup -\).
\end{remark}

However, an abstract isomorphism may not be compatible with the \(B\)-series.
We will follow a different approach. Instead of one cohomology-valued
series, we consider an indexed family of such series; they form a solution
to a so-called better behaved GKZ systems.
The full raw (undeformed) indexed family, which we will define shortly, 
gives rise to horizontal functionals on the monomial quotient.
Relative de Rham comparison identifies these functional with flat relative cycles,
and the distinguished index gives compact hypersurface periods. 


Set \(\delta=v_0=(0,1)\) and 
recall that \(\ell=n+1\). 
Suppose that we
have an integral matrix \(A\in\mathrm{Mat}_{\ell\times (p+1)}(\mathbf{Z})\) 
and a parameter \(\beta\in\mathbf{C}^{\ell}\).
Let \(a_{0},a_{1},\ldots,a_{p}\) be coordinates on \(\mathbf{C}^{p+1}\)
labeled by the columns of \(A\) and let \(\partial_{i}=
\partial\slash\partial a_{i}\). 
The GKZ system \(M_{A}(\beta)\) consists
of two types of operators:
\begin{itemize}
  \item[\textbf{(a)}] (box operators) \(\partial^{u_{+}}-\partial^{u_{-}}\) for \(u_{\pm}\in\mathbf{Z}^{p+1}_{\ge 0}\)
  with \(Au_{+}=Au_{-}\).
  \item[\textbf{(b)}] (Euler operators) \(E_\mu-\langle\mu,\beta\rangle\) with \(E_\mu=
  \sum_i\langle\mu,v_i\rangle a_i\partial_{a_i}\), \(\mu\in\widetilde{N}\).
\end{itemize}
In particular, the parameter for
periods for CY hypersurfaces is \(-\delta\).

Assume further that \(A\) is homogeneous. In this case, 
we denote \(C\) the cone generated by columns of \(A\) in
\(\mathbf{R}^{\ell}\) and identify \(\mathbf{Z}^{\ell}\) with \(\widetilde{M}\).
A solution of the better-behaved system \(\operatorname{bbGKZ}(A,C,\beta)\)
is a family of scalar functions \(\{\Phi_{c}\}_{c\in C\cap\widetilde M}\)
such that
\begin{equation}\label{eq:bbindexed}
 \partial_{a_i}\Phi_c=\Phi_{c+v_i},\quad
 E_\mu\Phi_c=\langle\mu,\beta-c\rangle\Phi_c.
\end{equation}
Each component solves the GKZ system \(M_A(\beta-c)\). Indeed, repeated differentiation
gives \(\partial_a^u\Phi_c=\Phi_{c+Au}\), and hence the box relations hold.
A cohomology-valued indexed series packages such scalar solutions.
Applying any fixed linear functional on the cohomology space gives
one scalar solution family. It is the entire family, not an individual
component, that is a solution of the better-behaved system.

First work in the smooth case. Put
\(R=\mathrm H^\bullet(\check Z;\mathbf{C})\). There are divisor classes
\(D_0,\ldots,D_r\) satisfying
\begin{equation}
 \sum_i v_iD_i=0,~\mbox{and}~D_0=-e.
\end{equation}
For \(c\in C\cap\widetilde M\), set
\begin{equation}
 L_c=\{\ell\in\mathbf{Z}^{r+1}\mid A\ell=-c\},
\end{equation}
and define the cohomology-valued series (cf.~
\cite{2013-Horja-toric-deligne-mumford-stacks-and-the-better-behaved-version-of-the-gkz-hypergeometric-system})
\begin{equation}
  \mathcal B_c^{\mathrm{raw}}(a)
  =\sum_{\ell\in L_c}\prod_{i=0}^r\frac{a_i^{\ell_i+D_i}}{\Gamma(1+\ell_i+D_i)}.
\end{equation}
Here the supscript ``raw'' indicates that we do not regularize the series.

The raw series is interpreted using the intersection theory on \(R\) 
and its convergence
chamber. Reciprocal Gamma functions are entire, and we put
\(a_i^{\ell_{i}+D_i}=\exp((\ell_{i}+D_i)\log a_i)\). 
The identity of Gamma functions and the linear relations
give
\begin{equation}\label{eq:rawindexed}
 \begin{aligned}
 \partial_{a_i}\mathcal B_c^{\mathrm{raw}}
   -\mathcal B_{c+v_i}^{\mathrm{raw}}&=0,\\
 \langle\mu,c\rangle\mathcal B_c^{\mathrm{raw}}
   +\sum_i a_i\langle\mu,v_i\rangle
       \mathcal B_{c+v_i}^{\mathrm{raw}}&=0.
 \end{aligned}
\end{equation}
In particular, the full family of cohomology-valued series \((\mathcal B_{c}^{\mathrm{raw}})_{c\in 
C\cap\widetilde{M}}\) solves
the better-behaved GKZ system \(\mathrm{bbGKZ}(A,C,0)\), whereas 
its \(c\)-component solves the ordinary system \(M_{A}(-c)\).
In addition, we have
\begin{equation}\label{eq:rawderivative}
 \partial_{a_0}\mathcal B_0^{\mathrm{raw}}
 =\mathcal B_\delta^{\mathrm{raw}}.
\end{equation}

The raw degree-zero scalar components
are \(1\) at \(c=0\) and zero at positive-height integral indices.

We can relate our \(\mathcal{B}_{\delta}^{\mathrm{raw}}(a)\) with
the cohomology-valued series \(\mathcal{B}_{\delta}(a)\) introduced by Hosono, Lian, and Yau.

Writing \(L=L_0\), the shifted expression at \(\delta\) is
\begin{equation}
\label{eq:rawdeltaindexed}
 \mathcal B_\delta^{\mathrm{raw}}
 =\frac1{a_0}\sum_{\ell\in L}
 \frac{a^{\ell+D}}{\Gamma(\ell_0+D_0)
 \prod_{i>0}\Gamma(1+\ell_i+D_i)}.
\end{equation}
Following Hosono, Lian, and Yau, we define the regularized series by
\begin{equation}\label{eq:renormalizeddelta}
 \mathcal B_\delta
 =\frac1{a_0}\sum_{\substack{\ell\in L\\\ell_0\le0}}(-1)^{\ell_0}
 \frac{\Gamma(1-\ell_0-D_0)}
      {\prod_{i>0}\Gamma(1+\ell_i+D_i)}a^{\ell+D}.
\end{equation}
The Euler's reflection formula gives
\begin{equation}\label{eq:rawreflection}
 \mathcal B_\delta^{\mathrm{raw}}
 =-e\cup U(e)\cup\mathcal B_\delta,~\mbox{where}~
 U(e)=\frac{\sin(\pi e)}{\pi e}.
\end{equation}

\begin{remark}
For a general \(c=(m,k)\), the Euler's reflection formula formally gives
\begin{equation}
 \mathcal B_c(a)=\sum_{\ell\in L_c}(-1)^{\ell_0+1}
 \frac{\Gamma(-\ell_0-D_0)}
      {\prod_{i>0}\Gamma(1+\ell_i+D_i)}a^{\ell+D}.
\end{equation}
\end{remark}

For \(\ell_0=-N\le 0\), the two central factors in \eqref{eq:renormalizeddelta}
and \eqref{eq:rawdeltaindexed} are
\begin{equation}
 \begin{aligned}
 \frac1{\Gamma(D_0-N)}
 &=\frac{D_0\prod_{j=1}^N(D_0-j)}{\Gamma(1+D_0)},\\
 (-1)^N\Gamma(1+N-D_0)
 &=\Gamma(1-D_0)\prod_{j=1}^N(D_0-j).
 \end{aligned}
\end{equation}
Thus the finite-product expression already exhibits the extra factor
\(D_0=-e\) in the raw series. The remaining ratio is
\(1/(\Gamma(1+D_0)\Gamma(1-D_0))\); this gives 
the identity \eqref{eq:rawreflection}.

The limiting prescription of \cite[(3.3)]{1996-Hosono-Lian-Yau-gkz-generalized-hypergeometric-systems-in-mirror-symmetry-of-calabi-yau-hypersurfaces} 
at \(\gamma_0=-1\) is
\begin{equation}
 \frac{\Gamma(\epsilon)}{\Gamma(\ell_0+\epsilon)}
 =(-1)^{\ell_0}
   \frac{\Gamma(1-\ell_0-\epsilon)}{\Gamma(1-\epsilon)}.
\end{equation}
For \(\ell_0\le0\) its limit is \((-1)^{\ell_0}\Gamma(-\ell_0+1)\).
The sine factors from the refletion formula cancel in this
ratio. Multiplication of the raw deformed series by \(\Gamma(D_0)\)
alone gives \(\mathcal B_\delta/\Gamma(1-D_0)\), rather than
\(\mathcal B_\delta\). This distinction matters before specialization.

For a relation basis \(\ell^{(\alpha)}\), write
\(D_i=\sum_\alpha\ell_i^{(\alpha)}p_\alpha\), and choose
\begin{equation}
 \log q_\alpha=\sum_i\ell_i^{(\alpha)}\log a_i
                  +\pi\mathrm i\,\ell_0^{(\alpha)}.
\end{equation}
The Frobenius series \(w_0(q,p)\) of \cite[(3.5)]{1996-Hosono-Lian-Yau-gkz-generalized-hypergeometric-systems-in-mirror-symmetry-of-calabi-yau-hypersurfaces} 
then satisfies
\begin{equation}
 w_0(q,p)=a_0\exp(-\pi\mathrm i\, e)\cup \mathcal B_\delta(a).
\end{equation}
All comparisons of image spaces must use the same scalar normalization;
the displayed branch factor is invertible and preserves \(e\cup R\).

\subsection{The Gauss--Manin connection on the monomial quotient}
Now we vary \(a\) in the nondegenerate parameter locus 
and consider the bundle \(\mathcal{Q}\).
Over which we define
\begin{equation}\label{eq:monomialGM}
 \nabla_{\partial/\partial a_i}[c]=[c+v_i]
\end{equation}
and extend by the Leibniz rule. This descends to the quotient
and defines a connection on the bundle \(\mathcal{Q}\). 

\begin{proposition}\label{prop:horizontalTheta}
The comparison \(\Theta_a\) in Proposition~\ref{prop:relativederham}
intertwines \eqref{eq:monomialGM} with the relative Gauss--Manin connection.
\end{proposition}
\begin{proof}
For \(c=(m,k)\) with positive height, put
\begin{equation}
\label{eq:def-omegac}
  \omega_c=(-1)^{k-1}\Gamma(k)\frac{x^m\Omega}{f_a^k}.
\end{equation}
Then \(\partial_{a_i}\omega_c=\omega_{c+v_i}\). Residue and the connecting
map commute with the induced Gauss--Manin connections, proving the
assertion on these classes.
For the height-zero generator \([0]\), we consider its representative \((-\Omega,0)\).
Locally choose a vector field \(V_i\) lifting parameter motion along
\(Y_a\), so that \((\mathrm df_a)(V_i)=-x^{m_i}\) on \(Y_a\).
The residue of \(x^{m_i}\Omega/f_a\) is simply
\begin{equation}
  \operatorname{Res}_{Y_{a}}\left(\frac{x^{m_{i}} \Omega}{f_{a}}\right) = -\iota_{V_i}\Omega|_{Y_a}.
\end{equation}
Differentiation of the integral over a moving relative chain \(\Gamma_{a}\), we obtain
\begin{equation}
   -\frac{\partial}{\partial a_{i}}\int_{\Gamma_a}\Omega=
   -\int_{\Gamma_{a}}\mathcal{L}_{V_{i}}\Omega.
\end{equation} 
Applying Cartan's formula and Stokes' theorem, we obtain
\begin{equation}
  \int_{\Gamma_{a}}\mathcal{L}_{V_{i}}\Omega=
  \int_{\Gamma_{a}}(\mathrm{d}\iota_{V_{i}}\Omega+\iota_{V_{i}}\mathrm{d}\Omega)
  =-\int_{\partial\Gamma_{a}} 
   \operatorname{Res}_{Y_{a}}\left(\frac{x^{m_{i}} \Omega}{f_{a}}\right).
\end{equation}
This gives the integral of this residue
over \(\partial\Gamma_a\). With the connecting-map convention fixed
in \eqref{eq:period}, this is exactly \(\Theta_a([v_i])\).
\end{proof}

\subsection{Functionals and geometric periods}
So far, every calculation is on the affine torus \(T\).

To identify geometric periods, we note that an MPCP 
desingularization \(Z\to \mathbf{P}_{\Delta}\) may not be
smooth. We take a toric resolution \(\pi\colon \widetilde{Z}\to Z\). Regarding
\(T\) as a maximal torus in \(\widetilde{Z}\) and \(\widetilde{E}_{a}:=\overline{Y}_{a}\subset
\widetilde{Z}\), we have a duality isomorphism
\begin{equation}
   \mathrm{H}_{n}(\widetilde{Z}\setminus \widetilde{E}_{a},
   (\widetilde{Z}\setminus \widetilde{E}_{a})\cap \widetilde{D};\mathbf{C})
   \cong \mathrm{H}^{n}(T,Y_{a};\mathbf{C}).
\end{equation} 
Here \(\widetilde{D}\) is the union of toric divisors on \(\widetilde{Z}\).

Recall that for \(c=(m,k)\in C\cap\widetilde{M}\) with \(k\ge 1\), we define 
\begin{equation}
  \omega_{c}:=(-1)^{k-1}\Gamma(k)\frac{x^m\Omega}{f_a^k}.
\end{equation}
We can view \(\omega_{c}\) as a meromorphic top form on \(\widetilde{Z}\). 
Under the nondegeneracy assumption, \(\pi^{\ast}E_{a}=\widetilde{E}_{a}\)
and hence \(\omega_{c}\) is a meromorphic top form
with a pole along \(\widetilde{E}_{a}\). Notice that 
\(E_{a}\sim -K_{Z}\), so the polytope for \(\widetilde{E}_{a}\)
remains \(\Delta\); it is the same as \(E_{a}\).
On the other hand,
\begin{equation}
  K_{\widetilde{Z}} = \pi^{\ast} K_{Z} + A,\quad A\ge 0.
\end{equation}
In general, \(\widetilde{E}_{a}\) is not an anticanonical
hypersurface in \(\widetilde{Z}\). The rational form 
\(\omega_{c}\), viewed as a meromorphic top form on 
\(\widetilde{Z}\), has a pole along \(\widetilde{E}_{a}\)
and would acquire additional zeros from the resolution.
In fact, 
\begin{equation}
\begin{aligned}
    \operatorname{div}_{\widetilde{Z}}\left(\omega_{c}\right)+k\tilde{E}_{a}
    &=-k \pi^{\ast}K_{Z}+K_{\widetilde{Z}}+
    \sum \langle m,\widetilde{\rho}_{i}\rangle \widetilde{D}_{i}\\
    &=-(k-1)\pi^{\ast}K_{Z}+A+\pi^{\ast}\left(\sum\langle m,\rho_{i}\rangle D_{i}\right)\\
    &\ge 0
\end{aligned}
\end{equation}
as long as \(c=(m,k)\in C\cap\widetilde{M}\) is in the interior.
Therefore, for an interior lattice point \(c\),
the notion of integrating \(\omega_{c}\) on \(\widetilde{Z}\)
against relative cycles in \(\mathrm{H}_{n}(\widetilde{U}_{a},
\widetilde{B}_{a};\mathbf{C})\) is thus defined. 
There is a commutative diagram
\begin{equation}
  \begin{tikzcd}
    &\mathrm{H}_{n-1}(E_{a};\mathbf{C})\ar[r,""] &\mathbf{C}\ar[d,equal]\\
    &\mathrm{H}_{n-1}(\widetilde E_{a};\mathbf{C})
    \ar[u,"\rho_{\ast}"]\ar[r,""]  &\mathbf{C}
  \end{tikzcd}
\end{equation}
induced from \(\omega_{c}\) where \(\rho\colon\widetilde{E}_{a}\to E_{a}\)
is the restriction of \(\pi\) and 
the upper (resp.~lower) horizontal map is 
\begin{equation}
  \int_{\bullet}\operatorname{Res}_{E_{a}}\left(\omega_{c}\right)\quad \left(
    \mbox{resp}.~\int_{\bullet}\operatorname{Res}_{\widetilde E_{a}}\pi^{\ast}(\omega_{c})\right).
\end{equation}
The integration on the (CY) orbifold \(E_{a}\)
is given by the pairing between
\begin{equation}
  \mathrm{H}^{n-1}(E_{a};\mathbf{C})~\mbox{and}~\mathrm{H}_{n-1}(E_{a};\mathbf{C}).
\end{equation}
Here we write the pullback \(\pi^{\ast}(\omega_{c})\)
to emphasize that this is regarded as an object on \(\widetilde{Z}\).

To describe period integrals explicitly,
we will work on \(\widetilde{Z}\) by passing to a toric
resolution and henceforth assume that \(Z\) is smooth.
We put \(U_{a}=Z\setminus E_{a}\) and \(B_{a}=U_{a}\cap D\) as before.
We shall keep in mind that \(E_{a}\) may not be an
anticanonical hypersurface in this setting.

Let \(R\) be a finite-dimensional cohomology ring.
(Recall that \(R\) will be the full Chen--Ruan cohomology ring
\(\mathrm{H}^{\bullet}_{\mathrm{CR}}(\check{\mathcal{Z}};\mathbf{C})\) at the end.)
\begin{proposition}
\label{prop:raw-period-identification}
Let \(\{G_{c}(a)\}_{c}\) be an analytic
family of \(R\)-valued series satisfying \eqref{eq:rawindexed}. For each constant, i.e.,
independent of the parameter \(a\), functional \(\lambda\colon R\to\mathbf{C}\), there is a unique locally flat family
\(\Gamma_\lambda(a)\in\mathrm H_n(T,Y_a;\mathbf{C})\) such that
\begin{equation}\label{eq:allindexedperiods}
 \begin{aligned}
 \lambda(G_0(a))&=-\int_{\Gamma_\lambda(a)}\Omega,\\
 \lambda(G_c(a))&=
 \int_{\partial\Gamma_\lambda(a)}\operatorname{Res}_{Y_a}\omega_c
 \quad (\deg c>0).
 \end{aligned}
\end{equation}
In particular, for \(c\) in the interior of \(C\),
this function belongs to \(J_{c}(W)\), where
\(J_{c}\) is integration of \(\omega_{c}\) defined
in \eqref{eq:def-omegac} over flat relative cycles
and \(W\) is the image of
\begin{equation}
\label{eq:tubesubspace}
 \mathrm{H}_n(Z,E_{a})\xrightarrow{\partial}
 \mathrm{H}_{n-1}(E_{a})
 \xrightarrow{\operatorname{Tube}}\mathrm{H}_n(U_{a})
 \xrightarrow{q_{a}}\mathrm{H}_n(U_{a},B_{a}).
\end{equation}
\end{proposition}
\begin{proof}
The Euler relation shows that
\(\ell_\lambda(a)([c])=\lambda(G_c(a))\) annihilates
\(\sum_\mu D_{\mu,a}S\); it is therefore a functional on \(Q_a\).
The differential relation immediately implies that
\begin{equation}
 \partial_{a_i}\bigl(\ell_\lambda([c])\bigr)
 =\ell_\lambda([c+v_i])
 =\ell_\lambda(\nabla_{\partial/\partial a_i}[c]).
\end{equation}
Thus \(\ell_\lambda\) is horizontal for the dual connection.
Proposition~\ref{prop:horizontalTheta} and the perfect integration pairing
identify it with a unique flat complex relative homology class. The
explicit images under \(\Theta\) in Proposition 
\ref{prop:relativederham} give \eqref{eq:allindexedperiods}.

Let \(\beta_\lambda=j_{\ast}\partial\Gamma_\lambda\in\mathrm H_{n-1}(E_a)\)
where \(j=j_{a}\colon Y_{a}\to E_{a}\).
The inclusion \((T,Y_a)\hookrightarrow(Z,E_a)\) of pairs shows that
\(\beta_\lambda\) is in the image of the boundary map for \((Z,E_a)\). 
Let \(c\in C\) be an interior lattice point.
Then \(\omega_{c}\) extends to the compactification without
introducing new poles other than \(E_{a}\).
Observe that \(\operatorname{Res}_{Y_a}(\omega_{c})\) is the restriction of
\(\operatorname{Res}_{E_a}(\omega_{c})\).
Consequently
\begin{equation}
 \lambda(G_{c}(a))
 =\int_{\beta_\lambda(a)}\operatorname{Res}_{E_a}(\omega_{c})
 =\frac{1}{2\pi\mathrm i}
   \int_{\operatorname{Tube}(\beta_\lambda(a))}\omega_{c}.
\end{equation}
Its relative tube class lies in \(W\); the
nonzero scalar factor does not affect the period subspace.
\end{proof}

\begin{remark}
According to the argument, we see that 
\begin{equation}
  \frac{1}{2\pi\mathrm i}
  \int_{\operatorname{Tube}(\beta_\lambda(a))}\frac{\Omega}{f_a}=
  \int_{\partial\Gamma_\lambda(a)}\operatorname{Res}_{Y_a}\frac{\Omega}{f_{a}}
\end{equation}
and the right hand side does not depend on the compactification.
\end{remark}

\subsection{The orbifold \(\mathcal{B}_{c}^{\mathrm{raw}}\) series and 
the orbifold HLY series}
\label{subsec:sectorreflection}
We now make the support prescription explicit and include the twisted
sectors. 
Let \(\check{\Sigma}\) be the homogenized star fan satisfying the
hypothesis in \S\ref{sec:preliminaries}, and let
\(b\in\operatorname{Box}(\check{\Sigma})\). Write
\begin{equation}
 b=\sum_{i=0}^r\alpha_i(b)v_i,~\mbox{where}~
 0\le\alpha_i(b)<1,
\end{equation}
using the minimal cone containing \(b\). Set
\begin{equation}
 L_{c,b}=\{\ell\mid\ell_{i}+\alpha_{i}(b)\in\mathbf{Z},~A\ell=-c\}.
\end{equation}
Write \(\mathbf{1}_b\) for the sector unit, and multiplication
by the untwisted divisors on each sector 
is understood by its restriction. 

Let \(R=\mathrm{H}_{\mathrm{CR}}^{\bullet}(\check{\mathcal{Z}};\mathbf{C})\) and
\(D_{1},\ldots,D_{p}\) be the divisor classes cooresponding to 
the \(1\)-cones in \(\check{\Sigma}\) and hence also \(1\)-cones
in the projected fan \(\check{\Sigma}_{M_{\mathbf{R}}}\).
In this identification, \(D_{0}=-\sum_{i=1}^{p}D_{i}\).
\begin{definition}
\label{def:raw-b-series}
Define
\begin{equation}
\label{eq:fullrawsectors}
 \mathcal B_c^{\mathrm{raw}}(a)
 :=\sum_{b\in\operatorname{Box}(\check\Sigma)}\sum_{\ell\in L_{c,b}}
 \left(\prod_i\frac{a_i^{\ell_i+D_i}}{\Gamma(1+\ell_i+D_i)}\right)\mathbf{1}_b.
\end{equation}
\end{definition}
A summand is zero unless the support
of \(b\) and all rays with negative integral \(\ell_i\) lie in a
common cone. It converges in the chosen fan chamber, as in
\cite[Theorems~4.6 and~4.8]{2013-Horja-toric-deligne-mumford-stacks-and-the-better-behaved-version-of-the-gkz-hypergeometric-system}.

Let
\begin{equation}
    L:=\operatorname{Ker}(A)\cap\mathbf{Z}^{p+1}.
\end{equation}
Whenever \(L_{c,b}\neq\emptyset\), we choose
\(\gamma_{c,b}\in L_{c,b}\). Then
\begin{equation}
    L_{c,b}=\gamma_{c,b}+L.
\end{equation}
Thus we have
\begin{equation}
\label{eq:raw-indexed-shifted}
  \mathcal B_c^{\mathrm{raw}}(a)=
  \sum_{\substack{b\in\operatorname{Box}(\check\Sigma)\\
  L_{c,b}\neq\varnothing}}
  a^{\gamma_{c,b}}
  \sum_{l\in L}
  \left(
    \prod_{i=0}^p
    \frac{a_i^{\,l_i+D_i}}
         {\Gamma(1+l_i+\gamma_{c,b,i}+D_i)}
  \right)\mathbf 1_b.
\end{equation}
In particular, for \(c=\delta=v_0\), we have
\begin{equation}
    L_{\delta,b}=(-1,0,\ldots,0)+L_{0,b}.
\end{equation}
Consequently,
\begin{equation}\label{eq:raw-delta-shifted}
  \mathcal B_\delta^{\mathrm{raw}}(a)
  =
  \frac{1}{a_0}
  \sum_{b\in\operatorname{Box}(\check\Sigma)}
  \sum_{\substack{l\in L_{0,b}\\l_0\leq0}}
  \frac{a^{l+D}}
       {\Gamma(l_0+D_0)
        \prod_{i>0}\Gamma(1+l_i+D_i)}
  \mathbf 1_b,
\end{equation}

\begin{lemma}\label{lem:centralbox}
For the reflexive Gorenstein star fan as in \S\ref{sec:preliminaries},
\(\alpha_0(b)=0\) for every box element. Moreover, for \(c\in C\)
in the interior and \(\ell\in L_{c,b}\) for which
\begin{equation}
  \prod_{i}\frac{a_i^{\ell_i+D_i}}{\Gamma(1+\ell_i+D_i)}\mathbf{1}_b
\end{equation}
is nonzero, we have
\(\ell_0\le -1\).
\end{lemma}
\begin{proof}
Choose a maximal boundary simplex containing the boundary rays in the
support of \(b\). They lie in a facet of the reflexive polytope. There
is an integral linear form \(u\) on \(M\) taking value one on this
facet and at most one on \(\Delta\). The integral functional
\(h(m,k)=k-u(m)\) has value one on \(\delta\) and zero on the boundary
rays of this cone. Hence \(h(b)=\alpha_0(b)\) is integral, forcing
\(\alpha_0(b)=0\).

Now if \(c\in C\)
is an interior lattice point, then \(h(c)>0\), while
\begin{equation}
  -h(c)=h(A\ell)=\ell_0+\sum_{i>0}\ell_i h(v_i)\ge\ell_0.
\end{equation}
The inequality holds because the potentially negative boundary exponents lie where \(h(v_i)=0\). 
Consequently \(\ell_0<0\), hence \(\ell_0\le-1\).
\end{proof}

\begin{definition}
\label{def:HLZ-b-series}
For an interior lattice point \(c\in C\), define
\begin{equation}
\label{eq:sector-normalized}
  \mathcal B_c(a)=
  \sum_{b\in\operatorname{Box}(\check\Sigma)}
  \sum_{\substack{\ell\in L_{c,b}\\ \ell_0\leq-1}}
  \frac{(-1)^{\ell_0+1}\Gamma(-\ell_0-D_0)}{\prod_{i>0}\Gamma(1+\ell_i+D_i)}
  a^{\ell+D}\mathbf{1}_b.
\end{equation}
\end{definition}
Since \(D_0=-e\), the reflection formula gives
\begin{equation}
\label{eq:orbifold-ref}
    \mathcal B_c^{\mathrm{raw}}(a)=-e\cup\frac{\sin(\pi e)}{\pi e}\cup\mathcal B_c(a).
\end{equation}
For \(c=\delta\), using
\(L_{\delta,b}=(-1,0,\ldots,0)+L_{0,b}\), we obtain
\begin{equation}
  \mathcal B_\delta(a)=\frac{1}{a_0}
  \sum_{b\in\operatorname{Box}(\check\Sigma)}
  \sum_{\substack{l\in L_{0,b}\\l_0\leq0}}
  \frac{(-1)^{l_0}\Gamma(1-l_0-D_0)}{\prod_{i>0}\Gamma(1+l_i+D_i)}
  a^{l+D}\mathbf 1_b,
\end{equation}
This is the extension of the HLY regularized
series to orbifolds.
Its coefficients are regular since \(\ell_0\le0\). The
same Gamma estimates as for the fan series, after removing the fixed
central zero, give convergence on a sufficiently small chamber domain.
Equivalently, one may first use formal chamber expansions and then
their convergent analytic representatives. Finite derivatives in the
nilpotent variables change the estimates only by polynomial and
logarithmic factors.

Termwise reflection, now valid on each sector with integral central
exponent, and Lemma~\ref{lem:centralbox} give
\begin{equation}\label{eq:orbifoldreflection}
 \mathcal B_\delta^{\mathrm{raw}}
 =-e\cup\frac{\sin(\pi e)}{\pi e}\cup\mathcal B_\delta.
\end{equation}
There is no fractional sine correction in the central factor: its
fractional part is zero. The class \(e\) is untwisted and Cartier and
acts on every sector by its restriction. The factor
\(\sin(\pi e)/(\pi e)\) is invertible on the full cohomology space.

\begin{remark}
The support prescription \(\ell_{0}\le 0\)
in \(\mathcal{B}_{\delta}\)
is essential. For \(\check Z=\mathbf{P}^2\),
the relation \(\ell=d(-3,1,1,1)\) at \(d=-1\) would contribute
\begin{equation}
 -\frac{\Gamma(-2+3H)}{\Gamma(H)^3}
 =-\frac{H^2}{6}+O(H^3)
\end{equation}
if negative degrees were admitted after cancellation. That term is
excluded by \(\ell_0\le0\). Its raw counterpart vanishes, so raw
reflection alone cannot specify the component in \(\operatorname{Ann}(e)\).
Our theorem concerns \(e\cup\mathcal B_\delta\) and is unchanged by
such an ambiguity, but the series itself requires
the stated convention.
\end{remark}

\subsection{Completeness for constant functionals}
\label{subsec:sparsecomplete}
Now we deal with the general situation; namely 
the columns span a proper finite-index sublattice.
We first remark that the raw family uses all indices in
\(C\cap\widetilde M\), not just the monoid generated by the columns.
The convergence and differential identities of
\eqref{eq:fullrawsectors} use Gamma recurrence, fan support, and the
estimates of Gamma functions in a chamber. These arguments are sectorwise
and do not require integral generation by the columns. In particular,
the allowed exponent sets remain finite unions of translates of the
same relation lattice. The finite-index issue concerns completeness,
which we shall verify directly below,
which is essentially due to Horja 
\cite{2013-Horja-toric-deligne-mumford-stacks-and-the-better-behaved-version-of-the-gkz-hypergeometric-system}*{Proposition 4.7 and Theorem 4.8}. 

\begin{lemma}\label{lem:support-complete}
The coefficient functionals of the full raw family give all horizontal
functionals on \(Q\).
\end{lemma}
\begin{proof}
Let \(\lambda\colon R\to \mathbf{C}\) be a constant functional.
Suppose \(\lambda(\mathcal B_c^{\mathrm{raw}})=0\) for every index
\(c\). For a box element \(b=\sum_i\alpha_i(b) v_i\) as before, we choose
\begin{equation}
 \ell_{i}=-\alpha_{i}(b),~\mbox{for}~1\le i\le p,~\mbox{and}~\ell_{0}=0.
\end{equation}
This belongs to \(L_{b,b}\), and every \(1+\ell_i\) is positive.
Indeed, for \(\ell=(\ell_{i})\)
\begin{equation}
  A\ell = -A\alpha = -b.
\end{equation}
For this \(\ell\), the corresponding coefficient is
\begin{equation}
 a^\ell\exp\left(\sum_iD_i\log a_i\right)
 g_b(D)\mathbf{1}_b,~\mbox{where}~
 g_b(D)=\prod_i\Gamma(1+\ell_i+D_i)^{-1}.
\end{equation}
The class \(g_b(D)\) is invertible on the sector. 
Now \(\lambda(\mathcal B_c^{\mathrm{raw}})=0\) and 
Lemma \ref{lem:zeroterms} below imply
\begin{equation}
 \lambda\left(g_b(D)\mathbf{1}_b
       \exp\left(\sum_iD_i\log a_i\right)\right)=0.
\end{equation}
The assumption on the convergence of
the series in Lemma \ref{lem:zeroterms} holds 
(cf.~\cite{2006-Borisov-Horja-mellin-barnes-integrals-as-fourier-mukai-transforms}).
It follows that \(\lambda\) must annihilate
\(g_b(D)\mathbf{1}_b\prod_iD_i^{u_i}\) for every multi-index \(u\).
The sector is generated as a module over the divisor algebra by
\(\mathbf{1}_b\); invertibility of \(g_b\) therefore shows that
\(\lambda\) vanishes on this sector. Applying this to every box element
proves the injectivity of the map from \(R^\vee\) to indexed solutions.

The indexed equations identify each such solution with a horizontal
functional on \(Q\), as in Proposition~\ref{prop:raw-period-identification}.
The relative comparison and the rank calculation of
Theorem~\ref{thm:comparison} give \(\dim Q=\dim R\).
Thus it is an isomorphism. 
\end{proof}

Recall that for \(a=(a_{0},\ldots,a_{p})\in\mathbf{C}^{p+1}\), a polyannulus region is
a region in \(\mathbf{C}^{p+1}\) of the form
\begin{equation}
  A:=\{a\in\mathbf{C}^{p+1}\mid r_{i}< |a_{i}| <R_{i}~\mbox{for every}~i\}~\mbox{with}~0<r_{i}<R_{i}.
\end{equation}
The logarithmic covering of \(A\) is defined as the region of the form
\begin{equation}
  \{w\in\mathbf{C}^{p+1}\mid \log r_{i}<\operatorname{Re}(w_{i})<\log R_{i}~\mbox{for every}~i\}
\end{equation}
with a projection to \(A\) via \(w\mapsto e^{w}\).

\begin{lemma}
\label{lem:zeroterms}
Let \(a=(a_{0},\ldots,a_{p})\). Let
\begin{equation}
  F(a) := \sum_{\alpha\in \Lambda} a^{\alpha} P_{\alpha}(\log a)
\end{equation}
where \(P_{\alpha}\) is a polynomial of degree \(\deg P_{\alpha}\le K\) and 
\(\Lambda\) is a sublattice in \(N^{-1}\mathbf{Z}^{p+1}\) for some \(N\ge K>0\).
Suppose that \(F(a)\) is convergent on every compact subset in the logarithmic covering
over a polyannulus region. If \(F(a)=0\), then \(P_{\alpha}(\log a)=0\) for every \(\alpha\in \Lambda\).
\end{lemma}
\begin{proof}
First make a change of variable to the logarithmic covering. Let 
\(a_{j}=e^{w_{j}}\). Then we can write \(F(a)\) as 
\begin{equation}
\label{eq:log-poly}
  F(e^{w})=\sum_{\alpha\in\Lambda} e^{\langle \alpha,w\rangle} P_{\alpha}(w).
\end{equation}
For any vector \(\tau\in\mathbf{Z}^{p+1}\), we have
\begin{equation}
  F(e^{w+2N\pi\mathsf{i}\,\tau})=\sum_{\alpha\in\Lambda} 
  e^{\langle \alpha,w\rangle} P_{\alpha}(w+2N\pi\mathsf{i}\,\tau).
\end{equation}
Take \(I=\{0,\ldots,N\}^{p+1}\). For each element \(k\in I\) we define
\begin{equation}
  L_{k}(s):=\prod_{i=0}^{p} \prod_{j\ne k_{i}} \frac{s_{i}-j}{k_{i}-j},\quad s=(s_{0},\ldots,s_{p}).
\end{equation}
Then \(L_{k}(s)=\delta_{ks}=\prod_{i=0}^{p} \delta_{k_{i}s_{i}}\) for \(s,k\in I\).
In other words,
\(L_{k}(s)\) is a charactersitic function supported at \(k\in I\). Now we have
\begin{equation}
  P_{\alpha}(y)=\sum_{k\in I} L_{k}\left(\frac{y-w}{2\pi N\mathsf{i}}\right)P_{\alpha}(w+2N\pi\mathsf{i}\, k).
\end{equation}
Thus we obtain a closed formula for \(P_{\alpha}(y)\).
Now consider
\begin{equation}
\begin{aligned}
  &\sum_{\alpha\in\Lambda} e^{\langle \alpha,w\rangle} P_{\alpha}(y)\\
  &=
  \sum_{\alpha\in\Lambda}\sum_{k\in I} L_{k}\left(\frac{y-w}{2\pi N\mathsf{i}}\right)
  e^{\langle \alpha,w\rangle}P_{\alpha}(w+2N\pi\mathsf{i}\, k)\\
  &=\sum_{k\in I} L_{k}\left(\frac{y-w}{2\pi N\mathsf{i}}\right)\sum_{\alpha\in\Lambda}
  e^{\langle \alpha,w+2N\pi\mathsf{i}\,k\rangle}P_{\alpha}(w+2N\pi\mathsf{i}\, k)\\
  &=0
\end{aligned}
\end{equation}
owing to the hypothesis \(F(a)=0\). This is 
true for any \(w\) as long as \(w\) stays in the
polyannulus region. Let \(w=b+\mathsf{i}\,\theta\).
Then 
\begin{equation}
\begin{aligned}
  0&=\frac{1}{(2\pi N)^{p+1}}\int_{[0,2\pi N]^{p+1}}
  \left(\sum_{\eta\in\Lambda} e^{\langle \eta,w\rangle} 
  P_{\eta}(y)\right)e^{-\mathsf{i}\langle\alpha,\theta\rangle}\mathrm{d}\theta\\
  &=e^{\langle \alpha,b\rangle}P_{\alpha}(y).
\end{aligned}
\end{equation}
This proves that \(P_{\alpha}(y)=0\) as desired.
\end{proof}

\begin{remark}
In the proof of Lemma \ref{lem:support-complete}, we only use
\(\dim R=\dim Q\) in Theorem \ref{thm:comparison}.
We can give a rather direct argument by considering the filtered complex.
\end{remark}

\begin{remark}
Notice that this argument
does not assert completeness of the \emph{single} zero-index component
for an ordinary GKZ system defined using a smaller generated lattice.
\end{remark}

\subsection{The reverse period inclusion}
\label{subsec:reverseperiod}
First let \(d=n-1\ge1\), and work on a sufficiently small simply connected
parameter domain on which the smooth projective pair \((E_a,E_a\cap D)\)
is locally trivial. The variations of mixed Hodge structures
and their morphisms below have locally constant ranks. 
Recall that 
\begin{equation}
  j_{a}\colon Y_{a}\to E_{a}~\mbox{and}~\delta_{Y_{a}}\colon \mathrm{H}^{d}(Y_{a};\mathbf{C})
  \to \mathrm{H}^{d+1}(T,Y_{a};\mathbf{C})
\end{equation}
is the connecting homomorphism.
Write
\begin{equation}
 \kappa_a=\delta_{Y_{a}}j_{a}^{\ast}\colon
 \mathrm{H}^d(E_a;\mathbf{C})\rightarrow
 \mathrm{H}^{d+1}(T,Y_a;\mathbf{C}).
\end{equation}

\begin{lemma}
\label{lem:compact-period-factor}
Let \(c\in C\) be an interior lattice point.
Every local period of the form \(\eta_{c}:=\operatorname{Res}_{E_{a}}(\omega_{c})\) on \(E_a\)
is obtained by evaluating \(\kappa_{a}([\eta_{c}])\) against a flat
functional on \(\mathrm H^{n}(T,Y_a)\). 
\end{lemma}
\begin{proof}
All cohomology groups are over \(\mathbf{C}\).
Note that \(\omega_{c}\) extends to \(Z\) without introducing new poles
(cf.~the paragraph preceding Proposition \ref{prop:raw-period-identification}).
Let \(D_i\) be the irreducible components of the toric boundary,
put \(B_{a,i}=E_a\cap D_i\), and denote the inclusions by
\begin{equation}
    j_{a}\colon Y_{a}\rightarrow E_{a},~\mbox{and}~\iota_{a,i}\colon B_{a,i}\rightarrow E_a.
\end{equation}
The localization sequence for the SNC boundary
gives
\begin{equation}
    K_{a}:=\operatorname{Ker} j_{a}^{\ast}=\sum_{i} 
    (\iota_{a,i})_{\ast}\left(\mathrm{H}^{d-2}(B_{a,i})\right).
\end{equation}
Indeed, the SNC Gysin spectral sequence has
\begin{equation}
  E_1^{-r,q}=\bigoplus_{|I|=r} \mathrm{H}^{q-2r}(B_{a,I})(-r)\Rightarrow
  \mathrm{H}^{q-r}(Y_a).
\end{equation}
Here the parentheses are the Tate twist
and the set \(I\) could be empty. In this case
\(B_{a,\emptyset}\) in understood as \(E_{a}\).
The spectral sequence
degenerates at \(E_2\), since \(E_a\) and its divisor strata 
are all smooth projective (cf.~\cite{1971-Deligne-theorie-de-hodge-ii}*{\S3.2}). 
At \(E_1^{0,d}=\mathrm H^d(E_a)\), the differential is
\begin{equation}
  \bigoplus_i\mathrm{H}^{d-2}(B_{a,i})(-1)\to\mathrm{H}^d(E_a).
\end{equation}
Consequently,
\begin{equation}
  K_a=\operatorname{Im}\!\left(\bigoplus_i\mathrm H^{d-2}(B_{a,i})(-1)\rightarrow\mathrm H^d(E_a)\right).
\end{equation}
Next we show that \(\eta_{c,a}=\operatorname{Res}_{E_{a}}(\omega_{c})\in 
\mathrm{H}^{d}(E_{a})\) 
is orthogonal to \(K_{a}\) for \(c=(m,k)\)
in the interior of \(C\).
Recall that \(\omega_{c}\) is a global holomorphic \(n\)-form on
\(Z\setminus E_{a}\). Its pullback to the divisor \(D_{i}\setminus B_{a,i}\)
vanishes. This implies that 
\begin{equation}
    \iota_{a,i}^{\ast}\eta_{c,a}=0.
\end{equation}
Consequently, for every \(\xi\in \mathrm{H}^{d-2}(B_{a,i})\), the projection
formula gives
\begin{equation}
  \int_{E_a}\eta_{c,a}\cup(\iota_{a,i})_{\ast}(\xi)=
  \int_{B_{a,i}}\iota_{a,i}^{\ast}(\eta_{c,a})\cup\xi=0.
\end{equation}
Hence \(\eta_{c,a}\) is orthogonal to \(K_a\) for the intersection
pairing.

Choose a polarization on \(E_{a}\) using the restriction \(L_a\) of a fixed
ample toric divisor on \(Z\). Writing
\(L_a=\sum_{i} c_i[B_{a,i}]\), we also obtain
\begin{equation}
  \operatorname{Im}\left(L_{a}\cup -\colon \mathrm{H}^{d-2}(E_{a})
  \to \mathrm{H}^{d}(E_{a})\right)\subset K_{a},
\end{equation}
because \((\iota_{a,i})_{\ast}\mathbf{1}=[B_{a,i}]\).
Now we use the Lefschetz decomposition
\begin{equation}
  K_a=\bigl(K_a\cap P^d(E_a)\bigr) \oplus \bigl(L_{a}\mathrm{H}^{d-2}(E_a)\bigr).
\end{equation}
By hard Lefschetz, the intersection pairing on the second summand is 
nondegenerate, while it is also nondegenerate on the first because \(K_{a}\)
is a real Hodge substructure.
Then it follows that the polarization gives a flat splitting
\begin{equation}
    \mathrm{H}^d(E_{a})=K_{a}\oplus K_{a}^\perp.
\end{equation}
Recall that
\begin{equation}
  \kappa_{a}=\delta_{Y_a}j_{a}^{\ast}\colon
  \mathrm{H}^d(E_a)\rightarrow \mathrm{H}^{d+1}(T,Y_{a}).
\end{equation}
The image of \(\mathrm H^d(T)\) in \(\mathrm H^d(Y_a)\) is pure of
weight \(2d\), whereas the image of \(\mathrm H^d(E_a)\) is pure of
weight \(d\). Strictness of morphisms of mixed Hodge structures and
\(d>0\) imply that these images have zero intersection. The long exact
sequence of the pair \((T,Y_a)\) therefore gives
\begin{equation}
 \operatorname{Ker}\kappa_{a}=\operatorname{Ker} j_{a}^{\ast}.
\end{equation}
In particular, \(\kappa_{a}\) identifies \(K_{a}^\perp\) with its image
as a local system.

Now let \(\beta_{a}\) be a flat compact \(d\)-cycle on \(E_{a}\).
Restrict its integration functional to \(K_{a}^\perp\) and transport
it through \(\kappa_{a}\). On our simply connected parameter domain,
this flat functional extends to a flat functional \(\Lambda_{a}\)
on the whole \(\mathrm{H}^{d+1}(T,Y_a)\). Since
\(\eta_{c,a}\in K_a^\perp\), we have
\begin{equation}
  \Lambda_{a}\left(\kappa_{a}(\eta_{c,a})\right)=\int_{\beta_a}\eta_{c,a}.
\end{equation}
This proves the lemma.
\end{proof}

In particular, if \(c=\delta=(0,1)\), we obtain the formula for
the holomorphic top form 
\begin{equation}
  \Omega_{a}:=\eta_{\delta}=\operatorname{Res}_{E_{a}}\frac{\Omega}{f_{a}}.
\end{equation}

\begin{instance}[The case \(n=1\)]
Take \(\Delta=[-1,1]\). The compactification is
\(Z=\mathbf{P}^1\), and let
\begin{equation}
  f_a(x)=a_{1}x^{-1}+a_0+a_{2}x,~\mbox{with}~a_{1}a_{2}\ne 0.
\end{equation}
be a nondegenerate Laurent polynomial.
Suppose \(f_{a}\) has both roots lie in \(T=\mathbf{C}^{\ast}\), so
\begin{equation}
  Y_a=E_a=\{p_+(a),p_-(a)\}.
\end{equation}
Work on a sufficiently small simply connected parameter domain on
which the roots can be labeled consistently. 

Since \(Y_a=E_a\), the map \(\kappa_{a}\) is the connecting
homomorphism in the exact sequence
\begin{equation}
  0\rightarrow \mathrm{H}^0(T)
\xrightarrow{i_a^{\ast}}\mathrm{H}^0(E_a)
\xrightarrow{\kappa_{a}}\mathrm{H}^1(T,E_a)
\rightarrow \mathrm{H}^1(T)
\rightarrow 0.
\end{equation}
Under the identification \(\mathrm{H}^0(E_{a})\cong \mathbf{C}^2\), the
restriction map \(i_{a}^{\ast}\) sends a constant \(c\) to \((c,c)\). Consequently,
\begin{equation}
  \operatorname{Ker}\kappa_{a}=\mathbf{C}(1,1),~\mbox{and}~
  \mathrm{H}^0(E_a)=\mathbf{C}(1,1)\oplus\mathbf{C}(1,-1).
\end{equation}
Both summands are flat under the chosen local labeling, and
\(\kappa_a\) identifies \(\mathbf{C}(1,-1)\) with its image.

The meromorphic form
\begin{equation}
  \frac{\Omega}{f_a}=\frac{1}{f_{a}(x)}\frac{\mathrm{d}x}{x}
\end{equation}
is regular at \(0\) and \(\infty\), and has simple poles precisely
at \(p_+\) and \(p_-\). Its residues have zero sum. Thus
\begin{equation}
  \Omega_{a}:=\operatorname{Res}_{E_a}\frac{\Omega}{f_a}=\bigl(r(a),-r(a)\bigr)\in \mathbf{C}(1,-1).
\end{equation}
Choose a flat relative homology class
\begin{equation}
  \Gamma(a)\in \mathrm{H}_1(T,E_a)
\end{equation}
whose boundary is
\begin{equation}
  \partial\Gamma(a)=[p_+(a)]-[p_-(a)].
\end{equation}
Such a class exists by the homology exact sequence of the pair.
The boundary map surjects onto the kernel of
\begin{equation}
  \mathrm{H}_0(E_a)\rightarrow \mathrm{H}_0(T),
\end{equation}
which is the anti-diagonal line. One may choose a path at a base
point and transport its relative homology class flatly.
With our convention for the connecting homomorphism,
\begin{equation}
  \left\langle\kappa_{a}([\Omega_{a}]),\Gamma(a)\right\rangle
  =\int_{\partial\Gamma(a)}\Omega_{a}=2r(a).
\end{equation}
Now let
\begin{equation}
  \beta(a)=c_+[p_+(a)]+c_-[p_-(a)]
\end{equation}
be any flat compact zero-cycle, where \(c_{\pm}\) are
constant. Its period is
\begin{equation}
  \int_{\beta(a)}\Omega_{a}=(c_+-c_-)r(a).
\end{equation}
Hence
\begin{equation}
  \int_{\beta(a)}\Omega_{a}=\left\langle\kappa_{a}([\Omega_{a}]),\frac{c_+-c_-}{2}\Gamma(a)\right\rangle.
\end{equation}
The class \(\Gamma(a)\) on the right is flat, so the evaluation defines
the required flat functional on \(\mathrm{H}^1(T,E_{a})\).

Thus every compact period is obtained from
\(\kappa_{a}([\Omega_{a}])\), exactly as in the previous Lemma. Here the
diagonal kernel is handled by the residue theorem, which places
the distinguished residue class in a flat complementary line.
\end{instance}

\begin{proposition}\label{prop:rawcompactequality}
Suppose the full raw indexed family defined in Definition \ref{def:raw-b-series} 
is complete, i.e., its constant
coefficient functionals give all horizontal functionals on \(Q\).
Then we have equalities
\begin{equation}
\label{eq:rawcompactequality}
\begin{aligned}
  J_{c}(W)
  &=\operatorname{Span}\left\{\int_{\beta_a}\eta_{c}~\left|~
  \beta_a~\text{flat in}~\mathrm H_{n-1}(E_a)\right.\right\}\\
  &=\operatorname{Span}\{\lambda(\mathcal B_{c}^{\mathrm{raw}})\mid\lambda\colon R\to\mathbf{C}\}
\end{aligned}
\end{equation}
for every interior lattice point \(c\in C\).
Recall that \(W\) and \(J_{c}(W)\) are defined in Proposition \ref{prop:raw-period-identification}.
\end{proposition}
\begin{proof}
Proposition~\ref{prop:raw-period-identification} gives the inclusion
\begin{equation}
   \operatorname{Span}\{\lambda(\mathcal B_{c}^{\mathrm{raw}})\mid\lambda\colon R\to\mathbf{C}\}\subset J_{c}(W).
\end{equation} 
Tube integration identifies the latter space with a
subspace of compact hypersurface periods. The inclusion 
\begin{equation}
  J_{c}(W)
  \subset\operatorname{Span}\left\{\int_{\beta_a}\eta_{c}~\left|~
  \beta_a~\text{flat in}~\mathrm H_{n-1}(E_a)\right.\right\}
\end{equation}
is also obvious.
Finally, Lemma~\ref{lem:compact-period-factor} expresses every compact period as
the evaluation of \(\Theta_{a}([c])=\kappa_{a}[\eta_{c}]\) by a flat
functional on relative cohomology. Transfer this functional to \(Q\)
and use completeness of the full raw family. Its value on \([c]\)
is \(\lambda(\mathcal B_{c}^{\mathrm{raw}})\) for one fixed
coefficient functional \(\lambda\). This proves
\begin{equation}
\begin{aligned}
  &\operatorname{Span}\left\{\int_{\beta_a}\eta_{c}~\left|~
  \beta_a~\text{flat in}~\mathrm H_{n-1}(E_a)\right.\right\}\\
  &\qquad \subset\operatorname{Span}\{\lambda(\mathcal B_{c}^{\mathrm{raw}})\mid\lambda\colon R\to\mathbf{C}\}.  
\end{aligned}
\end{equation}
Hence all the three spaces are equal.
\end{proof}

\begin{theorem}[Hyperplane conjecture]
\label{thm:seriesimage}
Let \(R=\mathrm H^\bullet_{\mathrm{CR}}(\check{\mathcal{Z}};\mathbf{C})\)
and let \(\mathcal B_{c}\) be the explicitly supported sectorwise
series \eqref{eq:sector-normalized} in Definition 
\ref{def:HLZ-b-series} with
\(c\) being an interior lattice point. Then
\begin{equation}\label{eq:hyperplanefinal}
 \mathcal{P}_{\mathcal B_{c}}(e\cup R)=J_{c}(W).
\end{equation}
\end{theorem}
\begin{proof}
Lemma~\ref{lem:support-complete} gives completeness of the full raw
indexed family for the support allowed in \S\ref{sec:preliminaries}. 

By \eqref{eq:orbifold-ref}, the Frobenius property of the
nondegenerate Chen--Ruan pairing, and invertibility of
\(U(e)=\sin(\pi e)/(\pi e)\),
\begin{equation}
 \begin{aligned}
 \operatorname{Span}\{\lambda(\mathcal B_{c}^{\mathrm{raw}})\mid 
 \lambda\colon R\to\mathbf{C}\}
 &=\mathcal{P}_{\mathcal B_{c}}(e\cup U(e)\cup R)\\
 &=\mathcal{P}_{\mathcal B_{c}}(e\cup R).
 \end{aligned}
\end{equation}
Combining this with \eqref{eq:rawcompactequality} proves the result.
\end{proof}

Take \(c=\delta\). We recover the HLY hyperplane conjecture.
\begin{theorem}[HLY hyperplane conjecture]
\label{thm:HLY-hyperplane}
\begin{equation}
 \mathcal{P}_{\mathcal{B}_{\delta}}(e\cup R)=J_{\delta}(W).
\end{equation}
\end{theorem}

\begin{remark}
The sector separation in Lemma~\ref{lem:support-complete} supplies the
finite-index step explicitly. We emphasize that
we do not replace the ambient lattice
by the lattice generated by the support, since that might change the
torus and the relative cohomology. This extension concerns the full
indexed system, not the completeness of a single ordinary GKZ component.
\end{remark}

\begin{corollary}
We have 
\begin{equation}
  \dim W = \operatorname{rank}(e\cup -) = \dim(e\cup R).
\end{equation}
\end{corollary}
\begin{proof}
  Consider the morphism \(\Phi\colon R^{\vee}\to \prod_{c\in C^{\circ}\cap\widetilde{M}}\mathscr{O}\) by
  \begin{equation}
    \lambda\to \{\lambda(\mathcal{B}_{c}^{\mathrm{raw}})\}_{c\in C^{\circ}\cap\widetilde{M}}.
  \end{equation}
  If \(\lambda\) annihilates \(eR\), then \(\lambda\in\operatorname{Ker}(\Phi)\).

  Conversely, suppose that \(\lambda(\mathcal{B}_{c}^{\mathrm{raw}})=0\) for every \(c
  \in C^{\circ}\cap\widetilde{M}\).
  For each box element \(b=\sum_{i}\alpha_{i}v_{i}\), 
  Then \(\delta+b\in C^{\circ}\cap\widetilde{M}\).
  We subsequently take 
  \begin{equation}
    \ell_{i}=
    \begin{cases}
      -\alpha_{i},~&i\ne 0\\
      -1,~&i=0.
    \end{cases}
  \end{equation}
  Then for \(\ell=(\ell_{i})\) we have \(A\ell = -\tilde{b}-\delta\).
  Note that 
  \begin{equation}
    \frac{1}{\Gamma(D_{0})}=\frac{D_{0}}{\Gamma(1+D_{0})}=\frac{-e}{\Gamma(1-e)}.
  \end{equation}
  By the same argument, we see that 
  \begin{equation}
    \lambda\left(-e\cup h_{b}(D)\mathbf{1}_{b}\exp\left(\sum D_{i}\log a_{i}\right)\right)=0
  \end{equation}
  where 
  \begin{equation}
    h_{b}(D):=\frac{1}{\Gamma(1-e)}\frac{1}{\prod_{i>0}\Gamma(1+\ell_{i}+D_{i})}
  \end{equation}
  is invertible. It follows that \(\lambda(e\cup R)=0\). 
  We deduce that 
  \begin{equation}
    \dim (e\cup R) =\dim R - \dim\operatorname{Ker}(\Phi) = \dim\operatorname{Im}(\Phi).
  \end{equation}
  Next we investigate 
  \(S^{\circ}:=\mathbf{C}[C^{\circ}\cap\widetilde{M}]\subset S\)
  the interior subspace. One can show that (cf.~Lemma \ref{lem:interior-derham})
  \begin{equation}
    \mathrm{H}^{n}(U_{a},B_{a};\mathbf{C})=\operatorname{Span}_{\mathbf{C}}\{[\omega_{c}]\mid 
    c\in S^{\circ}\}.
  \end{equation}
  If \(\lambda\in R^{\vee}\) with \(\lambda(\mathcal{B}_{c}^{\mathrm{raw}})=0\),
  by Proposition \ref{prop:raw-period-identification}, 
  there exists \(\Gamma_{\lambda}(a)\in \mathrm{H}_{n}(T,Y_{a};\mathbf{C})\)
  such that 
  \begin{equation}
    \lambda(\mathcal{B}_{c}^{\mathrm{raw}})=\int_{\partial \Gamma_{\lambda}(a)}
    \operatorname{Res}_{Y_{a}}(\omega_{c})=0.
  \end{equation}
  Therefore \(\lambda(\mathcal{B}_{c}^{\mathrm{raw}})=0\)
  if and only if \(\Gamma_{\lambda}(a)\) is in the kernel of
  \begin{equation}
    \mathrm{H}_{n}(T,Y_{a})\to 
    \mathrm{H}_n(Z,E_{a})\xrightarrow{\partial}
    \mathrm{H}_{n-1}(E_{a})
    \xrightarrow{\operatorname{Tube}}\mathrm{H}_n(U_{a})
    \xrightarrow{q_{a}}\mathrm{H}_n(U_{a},B_{a}).
  \end{equation}
  Call the image \(W'\). Then
  we can conclude that 
  \begin{equation}
  \begin{aligned}
    \dim W' &= \dim \mathrm{H}_{n}(T,Y_{a};\mathbf{C})-\dim\operatorname{Ker}(\Phi)\\
    &=\dim R-\dim\operatorname{Ker}(\Phi)\\
    &=\dim(e\cup R).
  \end{aligned}
  \end{equation}
  Clearly we have \(W'\subset W\). For \(\beta\in\mathrm{H}_{n}(Z,E_{a})\),
  \(\partial\beta\in\mathrm{H}_{n-1}(E_{a})\) defines a linear functional 
  on \(K_{a}^{\perp}\) (cf.~Lemma~\ref{lem:compact-period-factor}). By the proof
  of Lemma \ref{lem:compact-period-factor}, there exists a relatvie cycle
  \(\Gamma\in\mathrm{H}_{n}(T,Y_{a})\) such that 
  \begin{equation}
    \int_{\partial\beta}\eta_{c}=\int_{\Gamma} \kappa_{a}([\eta_{c}])
  \end{equation}
  for \(\eta_{c}=\operatorname{Res}_{E_{a}}\omega_{c}\) and every interior lattice point 
  \(c\). Lemma~\ref{lem:interior-derham} implies that \(\beta\) and \(\Gamma\)
  have the same image in \(\mathrm{H}_{n}(U_{a},B_{a})\). This demonstrates the opposite
  inclusion \(W\subset W'\).
\end{proof}

Recall that \(S^{\circ}\) is also stable under the operators
\(D_{\mu_{j},a}\) with \(j=1,\ldots,n+1\). Consider 
the quotient
\begin{equation}
    Q^{\circ}=S^{\circ}\Big/\sum_{j=1}^{n+1}D_{\mu_{j},a}S^{\circ}.
\end{equation}

\begin{lemma}
\label{lem:interior-derham}
There is an isomorphism
\begin{equation}
\label{eq:interiorrelative}
 Q^\circ\cong\mathrm{H}^n(U_{a},B_{a};\mathbf{C}).
\end{equation}
Under the forgetful map to \(\mathrm{H}^n(U_{a};\mathbf{C})\), the class of an
interior monomial \(c=(m,k)\) with \(k>0\), is represented by
\begin{equation}\label{eq:interiorform}
 \omega_c=(-1)^{k-1}\Gamma(k)\frac{x^m\Omega}{f_a^k}.
\end{equation}
\end{lemma}

\begin{proof}
The logarithmic relative de Rham complex 
for the right hand side is
\begin{equation}
 \Omega_{Z}^\bullet(\log D)(-D)(\ast E).
\end{equation}
On \(U\) it is the logarithmic de Rham model for the pair \((U,B)\);
allowing arbitrary poles along \(E\) computes its cohomology on \(Z\).
For a smooth toric variety,
\begin{equation}
 \Omega_{Z}^p(\log D)(-D)(kE)
 \cong\mathcal{O}_Z(K_{Z}+kE)\otimes\bigwedge^p M_{\mathbf{C}}.
\end{equation}
For \(k>0\), applying Kawamata--Viehweg vanishing 
theorem and passing to 
the direct limit, these sheaves
are acyclic in positive degrees, so global sections compute the
hypercohomology of the displayed complex.

On the other hand, the global sections 
of \(\mathcal O_Z(K_Z+kE)\) are indexed by
the interior lattice points of \(k\Delta\).
Indeed, the space of global sections of \(\mathcal O_Z(K_Z+kE)\) is 
the same as that of the model before the toric resolution, 
we thus can compute the polytope there.
On this model, along a toric
divisor with primitive inward normal \(\rho\), the order of
\(x^m\Omega/f_a^k\) is
\begin{equation}
 \langle m,\rho\rangle+k-1\ge0.
\end{equation}
Consequently, the coefficient space of global relative forms is the direct
limit of the interior monomial spaces whose transitions are given by
multiplying \(f_a\).

The proof follows almost the same lines as that of Proposition \ref{prop:relativederham}.
For the sake of completeness, we will give the precise argument below.

We first eliminate the height operator \(D_{t}\)
in the interior twisted complex.
It acts on a monomial \(c=(m,k)\) by
\begin{equation}
 D_t[c]=k[c]+\sum_i a_i[v_i+c].
\end{equation}
This operator is injective. 
Its leading term in height is
multiplication by \(F_a\), which is already injective on \(S^\circ\).
Its cokernel is identified with the above direct limit by
\begin{equation}
 (-1)^{k-1}\Gamma(k)\frac{x^m}{f_a^k}.
\end{equation}
The relation is precisely
\(k(-1)^{k-1}(k-1)!+(-1)^kk!=0\). Conversely, these height
relations generate the direct-limit identifications, so this map
has no futher kernel. 
Finally, for a torus direction \(\mu\), the remaining
operator becomes the ordinary logarithmic derivative
\begin{equation}
 \begin{aligned}
 \left[D_\mu[c]\right]\mapsto~
 &(-1)^{k-1}\Gamma(k)
 \left(\langle\mu,m\rangle\frac{x^m}{f_a^k}
 -k\frac{x^m\theta_\mu f_a}{f_a^{k+1}}\right)\\
 &=\theta_\mu\left((-1)^{k-1}\Gamma(k)
 \frac{x^m}{f_a^k}\right).
 \end{aligned}
\end{equation}
Elimination of the height direction decreases the cohomological
degree by one. The resulting complex is then identified with
the global relative
logarithmic de Rham complex above. This proves
\eqref{eq:interiorrelative} and \eqref{eq:interiorform}.
\end{proof}

\section{Relations to genus-zero orbifold GW invariants}
In this section, we explain how our 
family \(\{\mathcal{B}^{\mathrm{raw}}_{c}(a)\}\) raw \(B\)-series 
is related to orbifold GW invariants for CY hypersurfaces in \(\check{\mathcal{Z}}\).
To switch to the fan point of view, we put \(\rho_{i}=m_{i}\) with \(1\le i\le p\).

To simplify our notation, we denote by \(\check{\Sigma}_{M_{\mathbf{R}}}\)
the projection of \(\check{\Sigma}\) to \(M_{\mathbf{R}}\)
under the natural map
\begin{equation}
  \widetilde{M}\to M;
\end{equation}
it is the defining fan for the toric stack \(\check{\mathcal{Z}}\).
Note that the projectioin induces a bijection between \(\mathrm{Box}(\check{\Sigma}_{M_{\mathbf{R}}})\)
and \(\mathrm{Box}(\check{\Sigma})\). If 
\(b\in\mathrm{Box}(\check{\Sigma}_{\mathbf{R}})\), then we can take a minimal cone \(\sigma\)
containing \(b\) and write
\begin{equation}
  b = \sum_{\rho\in\sigma(1)} \alpha_{\rho}\rho~\mbox{with}~0\le \alpha_{\rho}<1.
\end{equation}
Then for \(v_{\rho}=(\rho,1)\), since \(\check{\mathcal{Z}}\) is Gorenstein, 
we have \(\operatorname{age}(b)=\sum_{\rho\in\sigma(1)}\alpha_{\rho}\in\mathbf{Z}\) and
\begin{equation}
  \sum_{\rho\in\sigma(1)} \alpha_{\rho}v_{\rho}\in \widetilde{M}.
\end{equation}
To avoid unnecessary complications, we shall use the same
notation \(b\) to denote an element in either 
\(\mathrm{Box}(\check{\Sigma}_{M_{\mathbf{R}}})\) 
or \(\mathrm{Box}(\check{\Sigma})\). 

\subsection{\(S_{\mathrm{ext}}\)-extended \(I\)-function}

In \cites{2015-Coates-Corti-Iritani-Tseng-a-mirror-theorem-for-toric-stacks,2019-Coates-Corti-Iritani-Tseng-some-applications-of-the-mirror-theorem-for-toric-stacks}, 
Coates, Corti, Iritani, and Tseng introduced 
the notion of \(S_{\mathrm{ext}}\)-extended \(I\)-function; it
has additional Novikov variables \(x_{1},\ldots,x_{m}\) with \(m=|S_{\mathrm{ext}}|\)
to keep track of box elements. We briefly recall their construction.

Let us take \(S_{\mathrm{ext}}=\mathrm{Box}(\check{\Sigma}_{M_{\mathbf{R}}})=
\{b_{1},\ldots,b_{m}\}\).
The extended fan map 
\(\rho^{S_{\mathrm{ext}}}\colon\mathbf{Z}^{p+m}\rightarrow M\)
is defined by
\begin{equation}
  \begin{cases}
  e_i\mapsto\rho_{i},\\
  e_{p+j}\mapsto b_j,~\mbox{for}~1\le j\le m.
  \end{cases}
\end{equation}
Let \(\mathbb{L}^{S_{\mathrm{ext}}}=\operatorname{Ker}(\rho^{S_{\mathrm{ext}}})\).
For a cone \(\sigma\in\check{\Sigma}_{M_{\mathbf{R}}}\), denote by \(\Lambda_{\sigma}^{S_{\mathrm{ext}}}
\subset \mathbb{L}^{S_{\mathrm{ext}}}\otimes\mathbf{Q}\) the 
subset consisting of elements of the form
\begin{equation}
  \sum_{i=1}^{p+m} \lambda_{i} e_{i}
\end{equation}
such that \(\lambda_{p+j}\in\mathbf{Z}\) for \(1\le j\le m\), and 
\(\lambda_{i}\in\mathbf{Z}\) for \(\rho_{i}\notin\sigma(1)\) with \(1\le i\le p\).
Set \(\Lambda^{S_{\mathrm{ext}}}=\cup_{\sigma\in\check{\Sigma}}\Lambda_{\sigma}^{S_{\mathrm{ext}}}\).
They also introduced the \emph{reduction function}
\begin{equation}
\begin{aligned}
    v^{S_{\mathrm{ext}}}\colon \Lambda^{S_{\mathrm{ext}}}&\to \mathrm{Box}(\check{\Sigma}_{M_{\mathbf{R}}})\\
    \sum_{i=1}^{p+m}\lambda_{i}e_{i}&\mapsto\sum_{i=1}^{p}\lceil\lambda_{i}\rceil\rho_{i}+
    \sum_{j=1}^{m}\lceil\lambda_{p+j}\rceil b_{j}.
\end{aligned}
\end{equation}
Note that \(\overline{v^{S_{\mathrm{ext}}}(\lambda)}=v^{S_{\mathrm{ext}}}(\lambda)\in 
\mathrm{Box}(\check{\Sigma}_{M_{\mathbf{R}}})\) because our lattice \(M\) is torsion free.
Additionally, 
\begin{equation}
  v^{S_{\mathrm{ext}}}(\lambda) = \sum_{i=1}^{p}\lceil\lambda_{i}\rceil\rho_{i}-
  \sum_{i=1}^{p}\lambda_{i}\rho_{i}=\sum_{i=1}^{p}\{-\lambda_{i}\}\rho_{i}
\end{equation}
where \(\{x\}:=x-\lfloor x\rfloor\). For each \(b\in 
\mathrm{Box}(\check{\Sigma}_{M_{\mathbf{R}}})\), we set
\begin{equation}
  \Lambda_{b}^{S_{\mathrm{ext}}}:=\{\lambda \in \Lambda^{S_{\mathrm{ext}}}
  \mid v^{S_{\mathrm{ext}}}(\lambda)=b\}
\end{equation}
and
\begin{equation}
  \Lambda E_{b}^{S_{\mathrm{ext}}}:=\Lambda_{b}^{S_{\mathrm{ext}}}\cap \mathrm{NE}^{S_{\mathrm{ext}}}(\check{\mathcal{Z}}).
\end{equation}
In \cite{2015-Coates-Corti-Iritani-Tseng-a-mirror-theorem-for-toric-stacks}*{Definition~28},
the \(S_{\mathrm{ext}}\)-extended \(I\)-function is introduced.
It is a cohomology-valued series 
\begin{equation}
    \begin{aligned}
    I_{\check{\mathcal{Z}}}^{S_{\mathrm{ext}}}(Q,x;z):=
    &z e^{\sum_{i=1}^{p} D_{i}t_{i}\slash z}\sum_{b\in \mathrm{Box}(\check{\Sigma}_{M_{\mathbf{R}}})}
    \sum_{\lambda=(d,k)\in \Lambda E_{b}^{S_{\mathrm{ext}}}}
    Q^{d}x^{k}e^{dt}\\
    &\qquad \times\prod_{i=1}^{p+m}\frac{\prod_{\{a\}=\{\lambda_{i}\},~a\le 0}(D_{i}+az)}{\prod_{\{a\}=\{\lambda_{i}\},~a\le \lambda_{i}}(D_{i}+az)}\,\mathbf{1}_{b}
    \end{aligned}
\end{equation}
\begin{enumerate}
    \item \(Q^{d}\) are Novikov variables;
    \item \(x_{1},\ldots,x_{m}\) are auxiliary variables for box elements;
    \item \(t_{1},\ldots,t_{p}\) are divisor variables.
    \item \(e^{dt}:=\prod_{i=1}^{p} e^{\langle d, D_{i}\rangle t_{i}}\);
    \item \(\mathbf{1}_{b}\) is the generator for the corresponding sector.
    \item \(D_{p+j}=0\) for \(j=1,\ldots,m\).
\end{enumerate}
Notice that the \(i\)\textsuperscript{th} component \(\lambda_{i}\) in \(\lambda
\in \Lambda E_{b}^{S_{\mathrm{ext}}}\)
has the property that \(\{\lambda_{i}\}=\{-b_{i}\}\). To be precise, we 
recall that 
the notation \(\lambda=(d,k)\) means the equality 
\begin{equation}
    \lambda=d+\sum_{j=1}^m k_j\left(e_{p+j}-\sum_{i=1}^p \alpha_{i}(b_{j})e_i\right),
\end{equation}
where the scalars \(\alpha_{i}(b_{j})\) satisfy
\begin{equation}
    b_j=\sum_{i=1}^p \alpha_{i}(b_{j})\rho_{i}
\end{equation}
using its minimal cone.
Also observe that 
\begin{equation}
  \sum_{i=1}^{p}\lambda_{i}+\sum_{j=1}^{m}k_{j}\operatorname{age}(b_{j})=\langle d, e\rangle
\end{equation}
where \(e=\sum_{i=1}^{p}D_{i}\).

To compute genus-zero orbifold GW invariants, we introduce the 
hypersurface series
\begin{equation}
\label{eq:i-hyper-sf}
\begin{aligned}
\mathcal{I}^{S_{\mathrm{ext}}}_{\mathrm{hyp}}&=z e^{\sum_iD_it_i/z}
\sum_{\lambda=(d,k)}
Q^d x^k e^{dt}
\left[\prod_{i=1}^{p+m}\frac{\prod_{\{a\}=\{\lambda_i\},\,a\leq0}(D_i+az)}
{\prod_{\{a\}=\{\lambda_i\},\,a\leq\lambda_i}(D_i+az)}\right]\\
&\qquad \qquad \qquad \qquad \qquad 
\times\prod_{r=1}^{\langle d, e\rangle}(e+rz)\,\mathbf{1}_{v^{S_{\mathrm{ext}}}(\lambda)}.
\end{aligned}
\end{equation}
For the hypersurface inclusion \(i\colon \check{Y}\to\check{\mathcal{Z}}\), we have
\begin{equation}
  I_{\check{Y}}^{S_{\mathrm{ext}}}=i^{\ast}\mathcal{I}^{S_{\mathrm{ext}}}_{\mathrm{hyp}},~\mbox{and}~
  i_{\ast}I_{\check{Y}}^{S_{\mathrm{ext}}}=e\cup\mathcal{I}^{S_{\mathrm{ext}}}_{\mathrm{hyp}}
\end{equation}
The goal is to match the series with \(\mathcal{B}_{c}^{\mathrm{raw}}(a)\).

\subsection{Comparison with \(\mathcal{B}^{\mathrm{raw}}_{c}(a)\)}
First of all, we match the summand described in \eqref{eq:i-hyper-sf}. 
Let \(\lambda=(d,k)\in \Lambda E_{b}^{S_{\mathrm{ext}}}\).
Define \(\ell=(\ell_{i})\) via
\begin{equation}
    \ell_{i}=\lambda_{i},~\mbox{for}~1\le i\le p,~\mbox{and}~\ell_{0}=-1-\sum_{i=1}^{p}\lambda_{i}
    -\sum_{j=1}^{m}k_{j}\operatorname{age}(b_{j}).
\end{equation}
Then for \(v_{i}=(\rho_{i},1)\) and \(v_{0}=(\mathbf{0},1)\) we have
\begin{equation}
\begin{aligned}
    \sum_{i=0}^{p} \ell_{i}v_{i}&=\left(\sum_{i=1}^{p}\lambda_{i} \rho_{i},-1-\sum_{j=1}^{m}k_{j}
    \operatorname{age}(b_{j})\right)\\
    &=-\sum_{j=1}^{m} k_{j}\tilde{b}_{j}-\delta.
\end{aligned}
\end{equation}
If we put \(c=\sum_{j=1}^{m} k_{j}\tilde{b}_{j}+\delta\), then \(\ell_{i}+b_{i}\in\mathbf{Z}\)
and we have \(A\ell = -c\). Each element in \(\lambda\in \Lambda E_{b}^{S_{\mathrm{ext}}}\)
corresponds to an element in \(L_{c,b}\).

Conversely, for each \(c=\sum_{j=1}^{m} k_{j}\tilde{b}_{j}+\delta\) and \(b\in
\mathrm{Box}(\check{\Sigma}_{M_{\mathbf{R}}})\), we take 
\(\ell\in L_{c,b}\) such that the corresponding term
\begin{equation}
  \prod_{i=0}^{p}\frac{a_i^{\ell_i+D_i}}{\Gamma(1+\ell_i+D_i)}~\mathbf{1}_b
\end{equation}
is nonzero. Then
\begin{equation}
  \ell_{i}+\alpha_{i}(b)\in\mathbf{Z}~\mbox{and}~\sum_{i=0}^{p} \ell_{i}v_{i} = -c.
\end{equation}
Set \(\lambda_{i}=\ell_{i}\) for \(1\le i\le p\). We observe that \((d,k)
\in \Lambda E_{b}^{S_{\mathrm{ext}}}\), where
\begin{equation}
  d_{i} = \lambda_{i}+\sum_{j=1}^{m} k_{j}\alpha_{i}(b_{j}).
\end{equation}
Moreover, by the support condition in \(\mathcal{B}_{c}^{\mathrm{raw}}(a)\), an
element \(\ell\in L_{c,b}\) contributes to a nonzero term to the sector \(\mathbf{1}_{b}\)
only if there exists a cone \(\sigma\) containing \(b\) and 
all rays \(\rho_{i}\) with \(\lambda_{i}=\ell_{i}\in\mathbf{Z}_{<0}\). Therefore 
\(\lambda_{i}=\ell_{i}\in\mathbf{Z}_{\ge 0}\) for \(\rho_{i}\notin\sigma(1)\).
For a numerically effective Cartier divisor \(H\), we can choose 
a representative \(H=\sum_{i=1}^{p} r_{i}D_{i}\) with
\(r_{i}\ge 0\) and \(r_{i}=0\) if \(\rho_{i}\in \sigma(1)\). 
Compute the intersection number
\begin{equation}
  \langle H,d\rangle=\sum_{i=1}^{p} r_{i}\lambda_{i}+\sum_{j=1}^{m}\sum_{i=1}^{p} r_{i} k_{j}\alpha_{i}(b_{j})\ge 0.
\end{equation}
Thus \(d\in\operatorname{NE}(\check{\mathcal{Z}})\) represents a curve class.

Now we match the monomials. For \(k=(k_{1},\ldots,k_{m})\),
we introduce
\begin{equation}
M_k(a):=\prod_{i=1}^{p}a_i^{\sum_{j}k_{j}\alpha_{i}(b_j)}.
\end{equation}
Then
\begin{equation}
a_0M_k(a)a^{\ell+D}=a^D\,a_0^{-\langle e,d\rangle}\prod_{i=1}^{p}a_i^{d_i}.
\end{equation}
After matching those variables, i.e. 
\begin{equation}
  t_i=\log a_i-\log a_0,~\mbox{and}~Q^d\mapsto 1,
\end{equation}
on the chosen branches and convergence domain, we obtain
\begin{equation}
  e^{\sum_{i}D_it_i}=a^D,~\mbox{and}~e^{dt}=a_0^{-N}\prod_{i} a_i^{d_i}.
\end{equation}
Define
\begin{equation}
  I_{\check Y,k}=\left.\partial_{x_1}^{k_1}\cdots\partial_{x_m}^{k_m}I_{\check{Y}}^{S_{\mathrm{ext}}}
  \right|_{x=0,z=1}.
\end{equation}
Define an operator \(\mathsf{C}\colon\mathrm{H}_{\mathrm{CR}}^{\bullet}(\check{\mathcal{Z}};\mathbf{C})
\to \mathrm{H}_{\mathrm{CR}}^{\bullet}(\check{\mathcal{Z}};\mathbf{C})\) sectorwisely by
\begin{equation}
  \left.\mathsf{C}\right|_{b}:=\Gamma(1-e)\prod_{i=1}^{p}\Gamma(1-\alpha_{i}(b)+D_{i})\cup -.
\end{equation}
It is a multiplication on each sector. Then 
\begin{equation}
  \begin{aligned}
    \mathsf{C}\left(\mathcal{B}_{c}^{\mathrm{raw}}(a)\right)=\sum_{b\in\operatorname{Box}(\check{\Sigma})}
    \sum_{\ell\in L_{c,b}}\frac{\Gamma(1-e)}{\Gamma(1+\ell_{0}-e)}\prod_{i>0}\frac{\Gamma(1-\alpha(b)_{i}+D_{i})}{\Gamma(1+\ell_{i}+D_{i})} a^{\ell+D} \mathbf{1}_{b}.
  \end{aligned}
\end{equation}
We can calculate for \(\ell_{0}\le 0\), 
\begin{equation}
  \begin{aligned}
    \frac{\Gamma(1-e)}{\Gamma(1+\ell_{0}-e)}&=(-e)(-e-1)\cdots(-e+1+\ell_{0})\\
    &=(-1)^{\ell_{0}}\prod_{a=0}^{-1-\ell_{0}}(e+a)\\
    &=(-1)^{1+\langle d,e\rangle}\prod_{a=0}^{\langle d,e\rangle}(e+a).
  \end{aligned}
\end{equation}
The resulting formula is
\begin{equation}
  \mathsf{C}\left(a_0 M_k(a)B_{c_k}^{\mathrm{raw}}\right)=-\mathsf{S}\left(i_{\ast}I_{\check Y,k}\right),
\end{equation}
where \(c_{k}=\delta+\sum_{i=1}^{m}c_{i}\tilde{b}_{i}\) and 
\begin{equation}
  \mathsf S(Q^d)=(-1)^{\langle d,e\rangle}Q^d.
\end{equation}
\begin{theorem}
\label{thm:extended-i-and-periods}
Under the mirror change of variables and an invertible sectorwise normalization \(\mathsf{C}\), 
the pushforward of the extended \(I\)-function for 
CY hypersurfaces \(\check{Y}\subset\check{\mathcal{Z}}\)
coincides with the generating series of compact periods for \(Y\subset \mathcal{Z}\).
\end{theorem}

\section{More on the interpolation complex \(\mathcal{K}^{\bullet}\)}
In \S\ref{sec:VHS-complex}, we introduced an interpolation complex
\(\mathcal{K}^{\bullet}\) to connect relevant objects.
Indeed, \(\mathcal{K}^{\bullet}\) is equipped with a 
flat connection outside \(s=0\);
it comes from the pullback of the Gauss--Manin connection,
which we shall explain now.

\subsection{Gauss--Manin connection}
We introduced a fiberwise isomorphism 
\begin{equation}
  T_{b}\colon \mathcal{S}_{b}\cong S,\quad [c]\mapsto b^{\psi(c)}[c]
\end{equation} 
for \(b\ne 0\). Then
\begin{equation}
  T_{b} \circ D_{\mu,b}\circ T_{b}^{-1} = b D_{\mu,a(b)}
\end{equation}
where \(a_{i}(s):=a_{i}s^{\psi(v_{i})-1}\).
Thus if \(b\ne 0\), we can invert \(b\) and the complex
is isomorphic to the pullback of the geometric twisted 
complex. At the complex level, the collection \(\{b^{-p}T_{b}\}_{p}\)
gives the desired isomorphism of chain complexes. 
This identification induces a Gauss--Manin connection along the base 
\begin{equation}
\begin{aligned}
  \nabla_{\partial_{s}}^{\mathrm{GM}} &= \partial_{s} + \sum_{i} a'_{i}(s)[v_{i}]\cdot~\\
  &=\partial_{s} + \sum_{i} a_{i}(\psi(v_{i})-1)s^{\psi(v_{i})-2}[v_{i}]\cdot~,
\end{aligned}
\end{equation}
where the product is taken in \(S\).
Explicitly, we have
\begin{equation}
    \nabla_{\partial_{s}}^{\mathrm{GM}} 
    = \partial_{s} + \frac{\Psi}{s} - \frac{\ell\cdot \operatorname{id}}{s}+ 
    \frac{1}{s^{2}}\sum_{i} a_{i}(\psi(v_{i})-1)M_{[v_{i}]}^{(s)}.
\end{equation}
Here \(\Psi([c])=\psi(c)[c]\) is the logarithmic product.
The connection \(\nabla^{\mathrm{GM}}\)
allows us to transport the integral lattice
\begin{equation}
  \mathrm{H}^{n}(T,Y_{a};\mathbf{Z})
  \slash\mathrm{tor}\subset\mathrm{H}^{n}(T,Y_{a};\mathbf{C})
\end{equation}
to a nearby fiber around \(s=0\).

Consider an ample toric divisor \(D=\sum_{i=1}^{p} c_{i}D_{i}\) with \(c_{i}>0\)
on the coarse moduli space \(\check{Z}\).
It determines a piecewise-linear support 
function \(\varphi_{D}\) on the support of its fan.
Explicitly, if \(\{n_{\sigma}\}_{\sigma\in\check{Z}}\) is the Cartier data, then 
it is given by
\begin{equation}
  \varphi_{D}(m)=-\langle n_{\sigma}, m\rangle~\mbox{for}~n\in \sigma.
\end{equation}
Therefore, if \(m_{j}\) is a primitive generator of a \(1\)-cone in \(\sigma\), then
\begin{equation}
  \varphi_{D}(m_{j})=-\langle n_{\sigma}, m_{j}\rangle=c_{j}.
\end{equation}
Note that our convention for convexity is different from the one used in \cite{2011-Cox-Little-Schenck-toric-varieties}.
Recall that \(\check{\Sigma}\) is a star triangularization obtained by lifting cones 
in the defining fan of \(\check{Z}\).
We can define a convex function \(\psi_{D}\) on \(\check{\Sigma}\) 
using the Cartier data \(\{(n_{\sigma},1)\}\).
Explicitly, if \((m,k)\) belongs to a cone in \(\check{\Sigma}\) whose
projection to \(M_{\mathbf{R}}\) lies in \(\sigma\), then
\begin{equation}
  \psi_{D}(m,k):=\langle (-n_{\sigma},1),(m,k)\rangle = \varphi_{D}(m)+k.
\end{equation}
One can easily check
\begin{eqnarray}
  \psi_{D}(v_{i})=\varphi_{D}(m_{i})+1 = c_{i}+1~\mbox{for}~1\le i\le p,
\end{eqnarray}
and \(\psi_{D}(v_{0})=1\). Therefore 
\begin{equation}
\begin{cases}
    a_{i}(s)=a_{i}s^{\psi(v_{i})-1}=a_{i}s^{c_{i}},~\mbox{for}~1\le i\le p,\\
    a_{0}(s)= a_{0}.
\end{cases}
\end{equation}
When \(s\to 0\), we are literally approaching the LCSL point.

\begin{proposition}
The connection has a regular singulary at \(s=0\).
\end{proposition}
\begin{proof}
  Although this follows from a gerenral theory of
  Gauss--Manin connections for family of smooth algebraic varieties,
  we would like to give a direct and computational proof, for we
  will compare the residue connection with the cup product on the \(A\)-side.

  Recall that \(\mathcal{H}:=\mathrm{H}^{\ell}(\mathcal{K}^{\bullet}_{A})\).
  We work over the analytic local ring \(\mathcal{O}:=\mathbf{C}\{s\}\).
  Let \(\mathcal{H}_{\mathcal{O}}=\mathrm{H}^{\ell}(\mathcal{K}^{\bullet}_{\mathcal{O}})\).
  By Theorem \ref{thm:comparison}, we
  choose homogeneous representatives \(g_{\alpha}\) of degree \(d_{\alpha}\) whose
  classes form a graded \(\mathcal{O}\)-basis for \(\mathcal{R}_{\mathcal{O}}\). They also
  form a \(\mathcal{O}\)-basis for \(\mathcal{H}_{\mathcal{O}}\).

  Set \(e_{\alpha}:=s^{-d_{\alpha}}\bar{g}_{\alpha}\), where \(\bar{g}_{\alpha}\)
  is the image of \(g_{\alpha}\) in \(\mathcal{H}_{\mathcal{O}}\). We are going to show that 
  \(\{e_{\alpha}\}\) is the needed logarithmic frame for the meromorphic bundle 
  \(\mathcal{H}_{\mathcal{O}}[s^{-1}]\). Compute
  \begin{equation}
    \begin{aligned}
      \nabla_{s\partial_{s}}^{\mathrm{GM}}(e_{\alpha})&=s^{-d_{\alpha}}\left(
      \overline{s\partial_{s}g_{\alpha}
      +(\Psi-d_{\alpha}-\ell)g_{\alpha}}\right)\\
      &\qquad +s^{-d_{\alpha}-1}\left(\overline{
      \textstyle \sum_{i} a_{i}(\psi(v_{i})-1)[v_{i}]\ast_{s} g_{\alpha}}\right).
    \end{aligned}
  \end{equation}
  The element before taking its class 
  in the first parenthesis has homogeneous degree \(d_{\alpha}\). We can do 
  the following induction.
  For a homogeneous element \(g\in S_{\mathcal{O}}\) of degree \(k\),
  there are \(c_\beta(s)\in\mathcal O\) and homogeneous \(g_j\) of degree \(k-1\) such that
  \begin{equation}
    g=\sum_{d_\beta=k}c_\beta(s)g_\beta+\sum_{j}G_{j}\ast_{s} g_{j}.
  \end{equation}
  Passing to twisted cohomology \(\mathcal{H}_{\mathcal{O}}\) gives
  \begin{equation}
    \bar g=\sum_{d_\beta=k}c_\beta(s)\bar g_\beta-s~\overline{\sum_{j}\theta_{j}g_{j}}.
  \end{equation}
  The expression inside the last bar has degree \(k-1\). 
  Equivalently, 
 \begin{equation}
    s^{-k}\bar g=s^{-k}\sum_{d_\beta=k}c_\beta(s)\bar g_\beta-s^{-(k-1)}~\overline{\sum_{j}\theta_{j}g_{j}}.
  \end{equation} 
  Repeating the process shows that \(s^{-k}\bar{g}\) can be expressed as a
  linear combination of \(e_{\alpha}\).
  The same argument applies to the latter term.
  This implies that 
  \(\nabla_{s\partial_{s}}^{\mathrm{GM}}(e_{\alpha})\) can be expressed as 
  a linear combination of \(\{e_{\alpha}\}\) over \(\mathcal{O}\) as desired.
\end{proof}

Let \(\mathcal{E}:=\bigoplus_{\alpha} 
\mathcal{O}\varepsilon_{\alpha}\subset \mathcal{H}[s^{-1}]\) 
where \(\varepsilon_{\alpha}=s^{\ell}e_{\alpha}\)
be the shift of the logarithmic lattice.
The residue connection \(\tau\)
becomes an endomorphism on \(\mathcal{E}\slash s\mathcal{E}\).
To make the presentation more concise, define 
\begin{equation}
  \operatorname{H}\colon \mathcal{S}_{\mathcal{O}}\to \mathcal{E}
  \quad\mbox{via}\quad g\mapsto s^{\ell-\deg g} \bar{g}
\end{equation}
for any homogeneous element \(g\in\mathcal{S}_{\mathcal{O}}\). We can compute
the residue connection explicitly. 
Then the proof shows that 
\begin{equation}
  \begin{aligned}
    \nabla_{\theta_{s}} \operatorname{H}(g) &= 
    \operatorname{H}(\theta_{s} g +\Psi g-(\deg g)g)\\
    &\qquad +\operatorname{H}(\textstyle\sum_{i} a_{i}(\psi(v_{i})-1)[v_{i}]\ast_{s} g).
  \end{aligned}
\end{equation}
Specializing the right hand side to \(s=0\), we 
obtain a formula for residue connection.
\begin{proposition}
The residue connection \(\tau\) is given by
\begin{equation}
  \Psi - \deg + \sum_{i} a_{i}(\psi(v_{i})-1)[v_{i}]\ast_{\check{\Sigma}}\colon 
  \mathcal{E}\slash s\mathcal{E}\to \mathcal{E}\slash s\mathcal{E}.
\end{equation}
\end{proposition}

\subsection{An auxiliary space}
Now we can use this to construct a conjectural correspondence between
the residue of the monodromy of period integrals around a LCSL point 
and the cup product on the Chen--Ruan cohomology ring of the mirror.


Inspired by the action \((\Psi-\deg)\) in the residue connection,
we will introduce an isomorphism
\begin{equation}
\label{eq:identifying-connection}
  \operatorname{\mathfrak{M}}\colon \mathcal{E}\slash s\mathcal{E}\to 
  \mathrm{H}_{\mathrm{CR}}^{\bullet}(
  \check{\mathcal{Z}};\mathbf{C})
\end{equation}
as follows.

On the algebra \(S_{\check{\Sigma}}\) we define
\begin{equation}
  \Xi_{\mu}:=\theta_{\mu}+G_{\mu,a}\ast_{\check{\Sigma}}\colon S_{\check{\Sigma}}\to S_{\check{\Sigma}}.
\end{equation}
Consider an auxiliary vector space
\begin{equation}
  V:=S_{\check{\Sigma}}\,\big\slash\, \sum_{\mu} \Xi_{\mu}(S_{\check{\Sigma}}).
\end{equation}
The vector space \(V\) mixes the structure
of the fiber at \(s=1\) (the derivations \(\Xi_{\mu}\)) and that of \(s=0\) (the product).

\qquad \\
\noindent \textbf{Notation and convention}.~Let \(u=z^{q}[b]\) with \(b\in\mathrm{Box}(\check{\Sigma})\). We take
a minimal cone containing \(b\) and write
\(b=\sum_{i=0}^{p}d_{i}v_{i}\) with \(0\le d_{i}<1\) for \(0\le i\le p\)
using this cone. Thus \(\{v_{i}\mid d_{i}\ne 0\}\)
generates the minimal cone containing \(b\).
Set \(q=(q_{0},\ldots,q_{p})\). 
\quad \\

Whenever \(u=z^{q}[b]\) is taken, we always pick \(d_{0},\ldots,d_{p}\)
and \(q_{0},\ldots,q_{p}\) as above and follow this notation and convention.

\begin{lemma}
There is a linear isomorphism \(V\cong \mathcal{E}\slash s\mathcal{E}\).
\end{lemma}
\begin{proof}
  Define \(\varphi\colon S_{\check\Sigma}\to \mathcal{E}\slash s\mathcal{E}\) by
  \begin{equation}
    u\mapsto \mathrm{H}(u)~\pmod{s\mathcal{E}}
  \end{equation}
  for homogeneous element \(u\) and extend it to \(S_{\check\Sigma}\) by linearity.

  Take \(u=z^{q}[b]\) of degree \(k\). Compute
  \begin{equation}
    \Xi_{\mu}\left(z^{q}[b]\right)=\theta_{\mu}u+
    \sum_i\langle\mu,v_i\rangle z_i\ast_{\check\Sigma} u.
  \end{equation}
  Applying \(\varphi\) to it gives
  \begin{equation}
    \begin{aligned}
      &s^{\ell-k}\theta_{\mu} u + s^{\ell-k-1}\sum_i\langle\mu,v_i\rangle z_i\ast_{\check\Sigma} u\\
      &\qquad =s^{\ell-k-1}\left(s\theta_{\mu}u+ \sum_i\langle\mu,v_i\rangle z_i\ast_{\check\Sigma}u\right)\\
      &\qquad =s^{\ell-k-1}\left(\sum_i\langle\mu,v_i\rangle z_i\ast_{\check\Sigma}u
      -\sum_i\langle\mu,v_i\rangle z_i\ast_{s}u\right)
    \end{aligned}  
    \end{equation}
    and the last term becomes zero after passing to \(s\mathcal{E}\).
    Indeed, let \(c\) be the lattice point underlying \(u\).
If \(v_i\) and \(c\) lie in a common cone, then
\begin{equation}
    z_i\ast_s u=z_i\ast_{\check\Sigma}u,
\end{equation}
so the corresponding summand is zero. Otherwise, we have
\begin{equation}
    z_i\ast_{\check\Sigma}u=0,~\mbox{and}~
  z_i\ast_s u=s^{b(v_i,c)}w_i,
\end{equation}
where \(w_i\) is a homogeneous monomial of degree \(k+1\)
and \(b(v_i,c)\geq1\). Its contribution is 
\begin{equation}
    -s^{\ell-k-1+b(v_i,c)}w_i
  =-s^{b(v_i,c)}\mathrm{H}(w_i)
  \in s\mathcal E.
\end{equation}
It is clear that this is surjective.
To show that it is an isomorphism, we compute
their dimension.
Consider the Koszul complex \(K^{\bullet}\)
on \(S_{\check\Sigma}\) defined by the commuting operators
\begin{equation}
  \Xi_\mu=\theta_\mu+G_{\mu,a}\ast_{\check\Sigma}.
\end{equation}
Filter it by
\[
  F_rK^p=
  F_{r+p}S_{\check\Sigma}
  \otimes\bigwedge^p\widetilde M_{\mathbf C},
\]
where \(F_jS_{\check\Sigma}\) consists of elements of
degree at most \(j\) as usual. 
The associated graded differential
is multiplication by the \(G_{\mu,a}\).
These elements form a regular sequence, so the associated
graded complex has cohomology only in degree \(\ell\).
The argument in Theorem \ref{thm:comparison} gives
\begin{equation}
  \operatorname{Gr}V\cong S_{\check\Sigma}/(G_{\mu,a}).
\end{equation}
Then \(\dim V=\dim\operatorname{Gr}V=\dim \mathrm{H}_{\mathrm{CR}}^{\bullet}(\check{\mathcal{Z}};\mathbf{C})
=\dim \mathcal{E}\slash s\mathcal{E}\).
\end{proof}

Define a map
\begin{equation}
\label{eq:def-monodromy-ring}
  \mathfrak{M}\colon S_{\check{\Sigma}}\to \mathrm{H}^{\bullet}_{\mathrm{CR}}(\check{\mathcal{Z}};\mathbf{C})
\end{equation}
by the following rule. Recall that we set \(z_{i}=a_{i}[v_{i}]\) for \(0\le i\le p\). 
Use the minimal cone presentation, we define
\begin{equation}
\label{eq:def-m}
  \mathfrak{M}\left(z^{q}[b]\right):=
  \mathbf{1}_{b}\prod_{i=0}^{p}\prod_{k=0}^{q_{i}-1}(D_{i}-d_{i}-k)
\end{equation}
and extend it linearly. 
Note that the product is taken in the Chen--Ruan cohomology ring.
Here the empty product is understood as \(\mathbf{1}\equiv\mathbf{1}_{\mathbf{0}}\)
corresponding to \(b=\mathbf{0}\) and \(D_{i}\) is the element corresponding to \(z_{i}\).

\begin{lemma}
\label{lem:identity-def-m}
For \(z^{q}[b]\in S_{\check{\Sigma}}\) as above, we have the identity
\begin{equation}
  \mathfrak{M}(z_{i}\ast_{\check{\Sigma}} z^{q}[b])=(D_{i}-d_{i}-q_{i})\mathfrak{M}(z^{q}[b]).
\end{equation}
\end{lemma}
\begin{proof}
  There are two cases. If \(z_{i}\) and \(z^{q}[b]\) lie in a common cone, 
  the equation is nothing but the definition \eqref{eq:def-m}.
  If \(z_{i}\) and \(z^{q}[b]\) do not lie in a common cone, then the
  product \(z_{i}\ast_{\check{\Sigma}} z^{q}[b]\) is zero.
  On the other hand, in this case, we must have \(q_{i}=d_{i}=0\). 
  The product \(\mathfrak{M}\left(z^{q}[b]\right)\) must contain
  factors \(D_j\) for every active ray outside the box support, 
  while \(\mathbf{1}_b\) encodes the sector condition. 
  Multiplying by \(D_i\) annihilates \(\mathfrak{M}\left(z^{q}[b]\right)\) 
  by the corresponding Stanley--Reisner 
  relations on that sector.
\end{proof}

Now we check that \(\mathfrak{M}\) factors through the quotient \(V\).
\begin{lemma}
\label{lem:well-def-m}
We have \(\mathfrak{M}\left(\Xi_{\mu}u\right)=0\) for any \(u\in S_{\check{\Sigma}}\).
\end{lemma}
\begin{proof}
Let \(u=z^{q}[b]\) as above. Then
\begin{equation}
  \begin{aligned}
    \Xi_{\mu}u =\left(\sum_i(q_i+d_i)\langle\mu,v_i\rangle\right)u+
    \sum_i\langle\mu,v_i\rangle z_i\ast_{\check\Sigma} u.
  \end{aligned}
\end{equation}
Applying \(\mathfrak{M}\) yields
\begin{equation}
  \begin{aligned}
    \mathfrak{M}\left(\Xi_{\mu}u \right)= \left(\sum_{i}\langle \mu,v_{i}\rangle D_{i}\right)\mathfrak{M}\left(u\right)=0.
  \end{aligned}
\end{equation}
This completes the proof.
\end{proof}
Lemmas \ref{lem:identity-def-m} and \ref{lem:well-def-m} together
imply that \(\mathfrak{M}\) factors through \(V\). We 
use the same notation to denote the induced map
\begin{equation}
  \mathfrak{M}\colon \mathcal{E}\slash s\mathcal{E}\cong 
  V\to \mathrm{H}_{\mathrm{CR}}^{\bullet}(\check{\mathcal{Z}};\mathbf{C}).
\end{equation}
It is clear that this is surjective. By the dimension reason,
we conclude that \(\mathfrak{M}\) is an isomorphism.

\subsection{Residue and cup product}
For \(g\in\mathcal{S}_{\mathcal{O}}\), 
\(\operatorname{H}(g)\) is an element in \(\mathcal{E}\).
We can further project it to the stalk \(\mathcal{E}\slash s\mathcal{E}\)
and the image of \(\operatorname{H}(g)\) there is denoted by 
\(\overline{\operatorname{H}(g)}\).

In what follows, we set \(\psi=\psi_{D}\) where
\(D=\sum_{i=1}^{p} c_{i}D_{i}\) is an ample toric divisor with \(c_{i}>0\).
\begin{theorem}
\label{thm:monodromy-residue-and-cup-prod}
The isomorphism \eqref{eq:identifying-connection}
identifies the residue connection 
\(\Psi-\deg + \textstyle \sum_{i} a_{i}(\psi(v_{i})-1)[v_{i}]\ast_{\check{\Sigma}}\) with
the cup product 
\begin{equation}
  D\cup - \colon 
  \mathrm{H}_{\mathrm{CR}}^{\bullet}(
  \check{\mathcal{Z}};\mathbf{C})\to \mathrm{H}_{\mathrm{CR}}^{\bullet}(
  \check{\mathcal{Z}};\mathbf{C}).
\end{equation}
\end{theorem}
\begin{proof}
Note that \(\mathfrak{M}\) is surjective and hence is an isomorphism by the dimension reason.
We compute for each homogeneous element in \(\{g_{\alpha}\}\). Consider
\(\sum_{i=0}^{p} q_{i} v_{i} + b\) with \(b=\sum_{i=0}^{p} d_{i} v_{i}\),~\(0\le d_{i}<1\), 
\begin{equation}
\begin{aligned}
  \tau\left(\overline{\operatorname{H}(z^{q}[b])}\right)&=\sum_{i=0}^{p} (q_{i}+d_{i})\psi(v_{i})
  \overline{\operatorname{H}(z^{q}[b])}
  -\sum_{i=0}^{p}(q_{i}+d_{i})\overline{\operatorname{H}(z^{q}[b])}\\
  &\qquad +\sum_{i=0}^{p} (\psi(v_{i})-1)\overline{\operatorname{H}(z_{i}z^{q}[b])}\\
  &= \sum_{i=1}^{p}(\psi(v_{i})-1)(q_{i}+d_{i})\overline{\operatorname{H}(z^{q}[b])}\\
  &\qquad +\sum_{i=0}^{p} (\psi(v_{i})-1)\overline{\operatorname{H}(z_{i}z^{q}[b])}.
\end{aligned}
\end{equation}
Applying the map \(\mathfrak{M}\) to it yields
\begin{equation}
\begin{aligned}
  &\mathfrak{M}\left(\tau\left(\overline{\operatorname{H}(z^{q}[b])}\right)\right)\\
  &=\sum_{i=0}^{p} (\psi(v_{i})-1)(q_{i}+d_{i})\mathfrak{M}\left(\overline{\operatorname{H}(z^{q}[b])}\right)\\
  &\qquad +\sum_{i=0}^{p} (\psi(v_{i})-1)\mathfrak{M}\left(\overline{\operatorname{H}(z_{i}z^{q}[b])}\right)\\
  &=\sum_{i=0}^{p} (\psi(v_{i})-1)(q_{i}+d_{i})\mathfrak{M}\left(\overline{\operatorname{H}(z^{q}[b])}\right)\\
  &\qquad +\sum_{i=0}^{p}(\psi(v_{i})-1)(D_{i}-q_{i}-d_{i})
  \mathfrak{M}\left(\overline{\operatorname{H}(z^{q}[b])}\right)\\
  &= D\cup\mathfrak{M}\left(\overline{\operatorname{H}(z^{q}[b])}\right)
\end{aligned}
\end{equation}
since \(\psi(v_{0})=1\).
The residue connection corresponds to 
cup product with the divisor class \(D\) under the linear isomorphism 
\(\mathfrak{M}\).
\end{proof}

Theorem \ref{thm:monodromy-residue-and-cup-prod} provides
an expected identification predicted by mirror symmetry.
We will utilize \(\mathcal{K}^{\bullet}\) to 
compare their integral structures as the next step
in the future.

\appendix

\section{Example: \texorpdfstring{\(\mathcal{Z}=\) 
the mirror \(\mathbf{P}^2\)}{the toric endpoint is P2}}
We have \(\check Z=\mathbf{P}^2\) in this example. The Newton polygon is the
small triangle with vertices \((1,0),(0,1),(-1,-1)\); it is not the
anticanonical polygon of \(Z=\mathbf{P}^2\). 
The latter case is treated
in the next section.
Take
\begin{equation}
 f_a=a_0+a_1x+a_2y+a_3/(xy),
\end{equation}
with exponent vectors
\begin{equation}
 v_0=(0,0,1),~v_1=(1,0,1),~
 v_2=(0,1,1),~v_3=(-1,-1,1).
\end{equation}
They satisfy \(v_1+v_2+v_3=3v_0\), and
\begin{equation}
 S=\mathbf{C}[u_0,u_1,u_2,u_3]/(u_1u_2u_3-u_0^3).
\end{equation}
The three twisted operators in logarithmic coordinates are
\begin{equation}
 \begin{aligned}
 D_x&=x\partial_x+M_{t(a_1x-a_3/(xy))},\\
 D_y&=y\partial_y+M_{t(a_2y-a_3/(xy))},\\
 D_t&=t\partial_t+M_{t(a_0+a_1x+a_2y+a_3/(xy))}.
 \end{aligned}
\end{equation}
For the ordinary Jacobian quotient, put
\begin{equation}
 q=-\frac{a_1a_2a_3}{a_0^3},\quad z_i=a_i u_i.
\end{equation}
The derivative relations give
\begin{equation}
 z_1=z_2=z_3=h,\quad z_0=-3h,
\end{equation}
and the binomial relation becomes \((1-27q)h^3=0\). Thus, for nonzero
coefficients and \(q\ne1/27\), we have
\begin{equation}
 R_a\cong\mathbf{C}[h]/(h^3).
\end{equation}
In this particular example the Jacobian ring already agrees with the
classical toric cohomology ring. 

For the fan degeneration, we choose
\(\psi(v_0)=0\) and \(\psi(v_1)=\psi(v_2)=\psi(v_3)=1\), extending
linearly on the three star cones. Then \eqref{eq:product} has presentation
\begin{equation}
 \mathcal{S}\cong
 \mathbf{C}[s,u_0,u_1,u_2,u_3]/(u_1u_2u_3-s^3u_0^3).
\end{equation}
At \(s=0\) end, the relation is nothing but the Stanley--Reisner relation
\(u_1u_2u_3=0\). Combining it with the linear relations yields
the cohomology ring
\(\mathbf{C}[h]/(h^3)\), with \(z_0=-3h\), as required.

For the entire family, eliminating the three linear forms gives
\begin{equation}
 \mathcal{S}/(G_x,G_y,G_t)
 \cong\mathbf{C}[s,h]/\bigl((1-27qs^3)h^3\bigr).
\end{equation}
Thus one can take
\begin{equation}
 A=\mathbf{C}[s,(1-27qs^3)^{-1}],\quad
 \mathcal{R}_A=A[h]/(h^3).
\end{equation}
The hypothesis \(q\ne1/27\) ensures that both endpoints remain in this
base. The excluded parameters are genuine degenerations of the graded
Jacobian quotient; its relation in \(h^3\) disappears. 

For a polynomial \(g\in\mathcal{S}_A\), write
\(\operatorname{cl}(g)\) for its top twisted cohomology class. Lifting
the graded basis gives the following \(A\)-basis of \(\mathcal{H}\).
\begin{equation}
 b_0=\operatorname{cl}(1),\quad
 b_1=\operatorname{cl}(-z_0/3),\quad
 b_2=\operatorname{cl}(z_0\ast_s z_0/9).
\end{equation}
The product inside the last representative is taken in \(\mathcal{S}_A\),
not in its cohomology. Here \(\psi(\delta)=\psi(v_{0})=0\), so
\(z_0\ast_s z_0=a_0^2[2\delta]\). Specializing this basis gives the
filtered vector-space isomorphism
\begin{equation}\label{eq:p2basis}
 b_0|_{s=1}\longmapsto1,\quad
 b_1|_{s=1}\longmapsto h,\quad
 b_2|_{s=1}\longmapsto h^2.
\end{equation}

\section{Example: \texorpdfstring{\(Z=\mathbf{P}^2\)}{Z=P2} and its resolved toric mirror}
Here ``mirror of \(\mathbf{P}^2\)'' means the toric ambient variety
obtained by polar duality of the anticanonical polytope, followed by its
crepant toric resolution. 

\subsection{The cubic and the two polygons}
Let \([X\mathpunct{:}Y\mathpunct{:}Z]\) be homogeneous coordinates on \(\mathbf{P}^2\), and
let \(P_a(X,Y,Z)\) be a generic cubic with all ten coefficients nonzero.
Take \(x=X/Z\), \(y=Y/Z\) on its maximal torus, and set
\begin{equation}
 f_a(x,y)=\frac{P_a(x,y,1)}{xy},\quad
 Y_a=\{f_a=0\}\subset(\mathbf{C}^{\ast})^2.
\end{equation}
The anticanonical polygon and its dual are given by
\begin{equation}
 \begin{aligned}
 \Delta&=\operatorname{conv}\{(-1,-1),(2,-1),(-1,2)\},\\
 \Delta^\vee&=\operatorname{conv}\{(1,0),(0,1),(-1,-1)\}.
 \end{aligned}
\end{equation}
The normal fan of \(\Delta\) is the fan of \(\mathbf{P}^2\).
The singular mirror ambient surface \(\check Z_{\mathrm{sing}}\)
has the fan over the faces of \(\Delta\), equivalently the normal fan
of \(\Delta^\vee\). Its three two-dimensional cones have determinant
three. It is the toric cubic surface
\begin{equation}
 \check Z_{\mathrm{sing}}
 \cong\{ABC=W^3\}\subset\mathbf{P}^3,
\end{equation}
with three \(A_2\) singularities. Subdividing its boundary edges at
all lattice points gives the smooth crepant resolution \(\check Z\).

\subsection{The fan and its cohomology ring}
List the boundary lattice points of \(\Delta\) in cyclic order, giving
\begin{equation}\label{eq:ninerays}
 \begin{aligned}
 \rho_1&=(-1,-1),&\rho_2&=(0,-1),&\rho_3&=(1,-1),\\
 \rho_4&=(2,-1),&\rho_5&=(1,0),&\rho_6&=(0,1),\\
 \rho_7&=(-1,2),&\rho_8&=(-1,1),&\rho_9&=(-1,0).
 \end{aligned}
\end{equation}
\(\check Z\) is a smooth projective toric surface with these nine rays, and
\begin{equation}
 \dim \mathrm{H}^0(\check Z)=1,\quad
 \dim \mathrm{H}^2(\check Z)=7,\quad
 \dim \mathrm{H}^4(\check Z)=1.
\end{equation}
Write \(D_i\) for the divisor of \(\rho_i\).
An explicit ring presentation is
\begin{equation}\label{eq:mirrorp2ring}
 \mathrm{H}^{\mathrm{even}}(\check Z;\mathbf{C})
 =\mathbf{C}[D_1,\ldots,D_9]/(I_{\mathrm{SR}},L_x,L_y),
\end{equation}
where \(\deg D_i=2\), \(I_{\mathrm{SR}}\) is generated by
\(D_iD_j\) for distinct nonadjacent rays, and
\begin{equation}
 \begin{aligned}
 L_x&=-D_1+D_3+2D_4+D_5-D_7-D_8-D_9,\\
 L_y&=-D_1-D_2-D_3-D_4+D_6+2D_7+D_8.
 \end{aligned}
\end{equation}
If \(p\in \mathrm{H}^4(\check Z;\mathbf{C})\) is normalized by
\(\int_{\check Z}p=1\), then
\begin{equation}\label{eq:intersectionmirror}
 D_iD_{i+1}=p,\quad
 D_i^2=
 \begin{cases}
 -p,&i=1,4,7,\\
 -2p,&i=2,3,5,6,8,9.
 \end{cases}
\end{equation}
The six \((-2)\)-curves form the three exceptional \(A_2\) chains.
With \(e=\sum_{i=1}^9D_i=c_1(-K_{\check Z})\), we obtain
\begin{equation}
 eD_i=
 \begin{cases}
 p,&i=1,4,7,\\
 0,&i=2,3,5,6,8,9,
 \end{cases}
 \quad e^2=3p.
\end{equation}
Consequently the cup product by \(e\) has image precisely 
\(\mathbf{C} e\oplus\mathbf{C} p\) and is of dimension two. This identifies the
anticanonical operation on the nine-dimensional toric endpoint explicitly.

\subsection{The cone algebra and an explicit degeneration}
The homogenized lattice points are \(\delta=(0,0,1)\) and
\(v_i=(\rho_i,1)\), \(1\le i\le9\). Its semigroup algebra is the
third Veronese algebra:
\begin{equation}\label{eq:veronese}
 \begin{aligned}
 S&\cong\bigoplus_{k\ge0}\mathbf{C}[X,Y,Z]_{3k},\\
 [(m_1,m_2,k)]&\longmapsto
 X^{m_1+k}Y^{m_2+k}Z^{k-m_1-m_2}.
 \end{aligned}
\end{equation}
The summand of polynomial degree \(3k\) has semigroup degree \(k\).
The inequalities defining \(k\Delta\) say exactly that the three
displayed exponents are nonnegative. In particular, \([\delta]\)
corresponds to \(XYZ\), and \(F_a=tf_a\) corresponds to \(P_a\).
The three multiplication symbols correspond to
\begin{equation}
 G_x=XP_{a,X}-P_a,~
 G_y=YP_{a,Y}-P_a,~\mbox{and}~G_t=P_a.
\end{equation}
Euler's identity \(XP_{a,X}+YP_{a,Y}+ZP_{a,Z}=3P_a\) implies
\begin{equation}
 R_a\cong S/(XP_{a,X},YP_{a,Y},ZP_{a,Z}),~\mbox{and}~
 \operatorname{Hilb}(R_a;T)=1+7T+T^2.
\end{equation}

Use the star triangulation with triangles
\(\operatorname{Conv}(0,\rho_i,\rho_{i+1})\). A concrete integral
convex support function for its homogenized fan is specified by
\begin{equation}\label{eq:mirrorheights}
 \psi(\delta)=0,~\mbox{and}~
 \psi(v_i)=
 \begin{cases}
 4,&i=1,4,7,\\
 3,&i=2,3,5,6,8,9,
 \end{cases}
\end{equation}
and linear extension on the star cones. To check convexity, let
\(h_i=\psi(v_i)\) and write
\(\rho_{i-1}+\rho_{i+1}=a_i'\rho_i\). Then \(a_i'=1\) at
\(i=1,4,7\), and \(a_i'=2\) otherwise. The bending inequalities are
\begin{equation}
 h_{i-1}+h_{i+1}-a_i'h_i
 =\begin{cases}2,&i=1,4,7,\\1,&\text{otherwise},\end{cases}
 \quad >0.
\end{equation}
Thus the linearity domains are precisely the nine star cones.

Equations \eqref{eq:product} and \eqref{eq:familyD} now give the
family explicitly. At \(s=0\) its algebra is
\begin{equation}
 S_{\check\Sigma}=\mathbf{C}[u_0,u_1,\ldots,u_9]/I_{\mathrm{SR}},
\end{equation}
and the three symbols impose
\begin{equation}
 z_0+\sum_{i=1}^9z_i=0,~
 \sum_{i=1}^9(\rho_i)_1z_i=0,~
 \sum_{i=1}^9(\rho_i)_2z_i=0,
 ~ z_i=a_i u_i.
\end{equation}
Their quotient is \eqref{eq:mirrorp2ring}, with \(z_i\mapsto D_i\)
and \(z_0\mapsto-e\). Theorem~\ref{thm:comparison} gives a free
rank-nine top cohomology module after the stated localization, and a
filtered vector-space isomorphism with the toric cohomology ring.

There is also a direct dimension check on the relative side. For a
nondegenerate cubic, \(Y_a\) is a smooth elliptic curve with nine
points removed, so \(\dim \mathrm{H}^1(Y_a)=10\). The restriction
\(\mathrm{H}^1(T)\to \mathrm{H}^1(Y_a)\) has rank two: small loops about punctures on
two different coordinate divisors have linearly independent winding
vectors in \(T\). Since \(\mathrm{H}^2(Y_a)=0\), the exact sequence of the
pair yields
\begin{equation}
 \dim \mathrm{H}^2(T,Y_a)=\bigl(10-2\bigr)+1=9,
\end{equation}
in agreement with \(1+7+1\) on \(\check Z\). Using the unresolved
three-ray mirror surface and its ordinary cohomology would miss the
six exceptional classes. The resolution gives the full endpoint
using ordinary cohomology. A singular model would instead require its
orbifold sectors.

\section{A singular MPCP example: 
\texorpdfstring{\(\mathbf{P}(1,1,1,1,2)\)}{P(1,1,1,1,2)}}
\label{sec:example-p11112}
This example shows that the distinction between ordinary and orbifold
cohomology persists even after an MPCP subdivision. In \(M\cong \mathbf{Z}^4\), set
\begin{equation}
 \begin{gathered}
 \rho_i=e_i~\mbox{for}~1\le i\le4,~\mbox{and}~\rho_5=(-1,-1,-1,-2),\\
 \end{gathered}
\end{equation}
and let
\begin{equation}
   \Delta=\operatorname{Conv}\{\rho_1,\ldots,\rho_5\}.
\end{equation}
The relation
\begin{equation}
 \rho_1+\rho_2+\rho_3+2\rho_4+\rho_5=0
\end{equation}
identifies the fan over the faces of \(\Delta\) with that of
\(\check Z=\mathbf{P}(1,1,1,1,2)\), after reordering the weights.
The primitive inward facet normals are
\begin{equation}
 \begin{gathered}
 (5,-1,-1,-1),\quad(-1,5,-1,-1),\quad(-1,-1,5,-1),\\
 (-1,-1,-1,2),\quad(-1,-1,-1,-1).
 \end{gathered}
\end{equation}
Thus \(\Delta\) is reflexive whose lattice points are precisely
\begin{equation}
 \Delta\cap M=\{0,\rho_1,\rho_2,\rho_3,\rho_4,\rho_5\}.
\end{equation}
The star subdivision with center \(0\) is regular and
maximal, hence gives an MPCP model already. Let 
\(\sigma=\mathrm{Cone}\{\rho_{1},\ldots,\hat{\rho}_{i},\ldots,\rho_{5}\}\).
The multiplicities of \(\sigma_{1},\ldots,\sigma_{5}\) are
\begin{equation}
 1,\quad1,\quad1,\quad2,\quad1.
\end{equation}
The cone omitting \(\rho_4\) gives the isolated terminal Gorenstein
quotient singularity \((1,1,1,1)/2\). 

The normalized volume is six and correspondingly we have
\begin{equation}
 h_\Delta^{\ast}(T)=1+T+2T^2+T^3+T^4.
\end{equation}
Take a nondegenerate Laurent polynomial with these six exponents:
\begin{equation}
 f_a=a_0+a_1x_1+a_2x_2+a_3x_3+a_4x_4
       +\frac{a_5}{x_1x_2x_3x_4^2}.
\end{equation}
Then
\begin{equation}
 \dim_\mathbf{C}\mathrm{H}^4(T,Y_a)=6,
 \quad
 \mathrm{H}^\bullet(\check Z;\mathbf{C})\cong\mathbf{C}[h]/(h^5),
 \quad\deg h=2.
\end{equation}
Ordinary cohomology has dimension five, so it cannot be the full
endpoint of this cone complex.

The missing class can be exhibited directly. In the singular cone,
\begin{equation}
 b=\frac12(\rho_1+\rho_2+\rho_3+\rho_5)=-e_4
\end{equation}
is the nonzero box element. In the homogenized fan it becomes
\begin{equation}
 \widetilde b=(b,2)
 =\frac12(v_1+v_2+v_3+v_5),\quad v_i=(\rho_i,1).
\end{equation}
Its age is two. It contributes one point-sector class \(\varepsilon\)
in Chen--Ruan degree four. As a graded vector space,
\begin{equation}
 \mathrm{H}_{\mathrm{CR}}^\bullet(\check{\mathcal Z};\mathbf{C})
 =\mathrm{H}^\bullet(\check Z;\mathbf{C})\oplus\mathbf{C}\varepsilon,
 \quad\deg\varepsilon=4.
\end{equation}
Its dimension and grading agree with
\(h_\Delta^{\ast}(T)\) and the rank-six relative cohomology.

The anticanonical class is \(e=6h\). Its restriction to the point
sector is zero over \(\mathbf{C}\), so
\begin{equation}
 e\cup\varepsilon=0,\quad
 \operatorname{im}(e\cup-)=\operatorname{Span}_{\mathbf{C}}\{h,h^2,h^3,h^4\}.
\end{equation}
Thus the additional sector changes the full cohomology space while
the anticanonical image in this example has dimension four.


\section{Example: 
\texorpdfstring{\(\check{\mathcal{Z}}:=\mathbf{P}(1,2,1)\times\mathbf{P}^1\)}{}}
\label{appD:comparison}

We compute the extended hypersurface series 
of CCIT
in \cites{2015-Coates-Corti-Iritani-Tseng-a-mirror-theorem-for-toric-stacks,2019-Coates-Corti-Iritani-Tseng-some-applications-of-the-mirror-theorem-for-toric-stacks} and compare it with the raw indexed series for
\begin{equation}
   \Delta=\operatorname{Conv}\{(1,0,0),(0,1,0),(-1,-2,0),
   (0,0,1),(0,0,-1)\}.
\end{equation}
The distinction between the box element and its interior translate is
essential in this comparison.

Let \(\rho_1,\ldots,\rho_5\) be the displayed vertices, in that order.
The fan consisting of cones over the faces of \(\Delta\) defines
\begin{equation}
  \check{\mathcal{Z}}:=\mathbf{P}(1,2,1)\times\mathbf{P}^1.
\end{equation}
Let \(H\) (resp.~\(P\)) be the pullback of
\(c_1(\mathcal O_{\mathbf P(1,2,1)}(1))\)
(resp.~\(c_1(\mathcal O_{\mathbf P^1}(1))\)). Then
\begin{equation}
\label{appD:divisors}
 (D_1,D_2,D_3,D_4,D_5)=(H,2H,H,P,P),~\mbox{and}~
 e=4H+2P.
\end{equation}
The untwisted cohomology is \(\mathbf{C}[H,P]/(H^3,P^2)\).

\subsection{CCIT's \(S_{\mathrm{ext}}\)-extended \(I\)-function}
There is one nontrivial box element,
\begin{equation}
  b=(0,-1,0)=\frac{1}{2}(\rho_1+\rho_3),~\mbox{and}~
 \operatorname{age}(b)=1.
\end{equation}
If $\mathbf1_b$ denotes its
unit, then $H\cup\mathbf1_b=0$ and $P^2\cup\mathbf1_b=0$.

Set $v_i=(\rho_i,1)$ and $\delta=(0,1)$. We distinguish
\begin{equation}
\label{appD:lifts}
 b\equiv (b,1)=\frac{1}{2}(v_1+v_3),~\mbox{and}~
 \tilde{b}:=\delta+\beta=(b,2).
\end{equation}
The point \(\beta\) is on the boundary 
of \(C=\operatorname{Cone}(\Delta\times\{1\})\) and
therefore \(\tilde{b}\) is in the interior. 

Consider
\begin{equation}
   f_a=a_0+a_1x+a_2y+a_3x^{-1}y^{-2}+a_4z+a_5z^{-1}
\end{equation}
and assume $a_i\ne0$. Choose compatible logarithms and put
\begin{equation}\label{appD:q}
 q_1=\frac{a_1a_2^2a_3}{a_0^4},~
 q_2=\frac{a_4a_5}{a_0^2},~\mbox{and}~A=a_1a_3.
\end{equation}
Thus \(a^D=q_1^Hq_2^P\), where \(D_0=-e\).

Take the extension set $S_{\mathrm{ext}}=\{b\}$.
The \(\mathbf{Q}\) splitting of the extended relation sequence gives
\begin{equation}\label{appD:degrees}
 \lambda(d,m,k)=\left(d-\frac{k}{2},2d,d-\frac{k}{2},m,m,k\right),~\mbox{with}~
 d\in\frac{1}{2}\mathbf Z_{\ge0},~m,k\in\mathbf Z_{\ge0}.
\end{equation}
Here \(d\) stands for the degree from \(\mathbf{P}(1,2,1)\), 
\(m\) for the degree from \(\mathbf{P}^{1}\) and \(k\)
for the extended box.
The extension vector is $(-1/2,0,-1/2,0,0,1)$, since
$b=(\rho_1+\rho_3)/2$.
Define
\begin{equation}
  N=4d+2m,~\mbox{and}~
 \epsilon(d,k)=
 \begin{cases}0,&2d-k\in2\mathbf Z,\\
 1,&2d-k\in2\mathbf Z+1.
 \end{cases}
\end{equation}
The reduction function assigns $\lambda(d,m,k)$ to sector $0$ or $b$
according to $\epsilon(d,k)$. Set \(\mathbf{1}_0=1\) 
and \(\mathbf{1}_1=\mathbf{1}_b\).
Notice that $N$ is always an even nonnegative integer.

Let \(i\colon Y\rightarrow \check{\mathcal{Z}}\) be a 
smooth generic anticanonical hypersurface. 

For $l\in\mathbf Q$, define 
\begin{equation}
   \mathcal P(u,l;z)=
 \frac{\prod_{a\le0,~a-l\in\mathbf Z}(u+az)}
 {\prod_{a\le l,~a-l\in\mathbf Z}(u+az)}.
\end{equation}
The ambient-valued hypersurface modification is
\begin{equation}\label{appD:I-z}
\begin{split}
 \mathcal I(q_1,q_2,x,z)
 ={}&zq_1^{H/z}q_2^{P/z}
 \sum_{d,m,k}\frac{q_1^dq_2^mx^k}{k!z^k}
 \mathcal{P}(H,d-k/2;z)^2\\
 \qquad &\mathcal{P}(2H,2d;z)
 \cdot\mathcal{P}(P,m;z)^2
 \prod_{r=1}^{N}(e+rz)\,\mathbf1_{\epsilon(d,k)}.
\end{split}
\end{equation}
The hypersurface series is $I_Y=i^{\ast}\mathcal{I}$.
All sums are in the range in \eqref{appD:degrees}.

We now specialize to $z=1$ to compare with the raw series. Put
\begin{align}
 G_\epsilon&=\Gamma(1+H-\epsilon/2)^2
       \Gamma(1+2H)\Gamma(1+P)^2,\label{appD:G}\\
 A_{d,m,k}&=\Gamma(1+H+d-k/2)^2
       \Gamma(1+2H+2d)\Gamma(1+P+m)^2.\label{appD:denom}
\end{align}
Equation~\eqref{appD:I-z} becomes
\begin{equation}\label{appD:I-gamma}
 \mathcal I(q_1,q_2,x,1)=q_1^Hq_2^P
 \sum_{d,m,k}\frac{q_1^dq_2^mx^k}{k!}
 \frac{G_{\epsilon(d,k)}\Gamma(1+e+N)}
 {A_{d,m,k}\Gamma(1+e)}\mathbf1_{\epsilon(d,k)}.
\end{equation}
Write $\mathcal I_k=\partial_x^k\mathcal I|_{x=0,z=1}$, so that
$\mathcal I=\sum_{k\ge0}x^k\mathcal I_k/k!$ at $z=1$.

\subsection{The raw series \(\mathcal{B}_{c}^{\mathrm{raw}}(a)\)}
For \(c=\delta+k\beta\), the exponent vectors solving
\(\sum_{i=0}^5\ell_{i} v_i=-c\) are
\begin{equation}
   \ell=\left(-N-1,d-\frac{k}{2},2d,d-\frac{k}{2},m,m\right).
\end{equation}
Define the normalized raw series
\begin{equation}
   \mathcal{R}_k=a_0A^{k/2}\mathcal{B}_{\delta+k\beta}^{\mathrm{raw}}.
\end{equation}
Then
\begin{equation}
\label{appD:raw}
 \mathcal{R}_k=q_1^Hq_2^P\sum_{d,m}
 \frac{q_1^dq_2^m}
 {\Gamma(-e-N)A_{d,m,k}}\mathbf1_{\epsilon(d,k)}.
\end{equation}
Let $\mathsf{C}$ be the (sectorwise) linear automorphism of
$\mathrm H^\bullet_{\mathrm{CR}}(
\check{\mathcal{Z}};\mathbf C)$ given by
\begin{equation}\label{appD:C}
 \mathsf{C}|_\epsilon=\Gamma(1-e)\cdot G_\epsilon.
\end{equation}
Since $N$ is even, Euler's reflection formula gives
\begin{equation}\label{appD:reflection}
 \frac1{\Gamma(-e-N)}
 =-\frac{\sin(\pi e)}{\pi}\Gamma(1+e+N)
 =-\frac{e\,\Gamma(1+e+N)}
 {\Gamma(1-e)\Gamma(1+e)}.
\end{equation}
Comparison of \eqref{appD:I-gamma} and \eqref{appD:raw} therefore
gives
\begin{equation}
\label{appD:main}
      a_0A^{k/2}\mathcal{B}_{\delta+kb}^{\mathrm{raw}}
      \equiv -\mathsf{C}^{-1}\bigl(e\cup\mathcal{I}_k\bigr),
      ~k\ge0.
\end{equation}
In particular,
\begin{align}
 a_0\mathcal{B}_\delta^{\mathrm{raw}}
   &=-\mathsf{C}^{-1}(e\cup\mathcal{I}_0),\label{appD:delta}\\
 a_0A^{1/2}\mathcal{B}_{\tilde{b}}^{\mathrm{raw}}
   &=-\mathsf{C}^{-1}(e\cup\mathcal{I}_1).\label{appD:box-interior}
\end{align}
Thus the same degree-independent Gamma normalization applies to both
interior indices. Equivalently,
\begin{equation}\label{appD:packaged}
 e\cup\mathcal{I}(q_1,q_2,x,1)
 =-\mathsf{C}\sum_{k\ge0}\frac{x^k}{k!}
       a_0A^{k/2}\mathcal{B}_{\delta+kb}^{\mathrm{raw}}.
\end{equation}



Let us check our calculation on the box sector. Note
that on the box sector $e=2P$ and
\begin{equation}
   \mathsf{C}|_b=\pi\Gamma(1-2P)\Gamma(1+P)^2=\pi,
\end{equation}
because \(P^2=0\). The restriction 
of \(-K_{\check{\mathcal{Z}}}\) 
to its coarse
\(\mathbf{P}^1\) is \(\mathcal O(2)\), 
so the hypersurface restricted to
this sector \(Y_b\) consists of two
stacky points; in particular \(i^{\ast}P=0\) on \(Y_b\).
The coefficient of \(x\) is
\begin{equation}
\label{appD:box-I}
\begin{aligned}
   (i^{\ast}\mathcal{I}_1)_b&=
 \pi\sum_{d,m\ge0}
 \frac{(4d+2m)!}{\Gamma(d+\tfrac12)^2(2d)!(m!)^2}
 q_1^dq_2^m\,i^{\ast}\mathbf{1}_b\\
 &=(1+48q_1+2q_2+\cdots)\, i^{\ast}\mathbf{1}_b.
\end{aligned}
\end{equation}
Here \(d\) is integral because \(\epsilon(d,1)=1\).
The corresponding interior raw component is
\begin{equation}
\label{appD:box-raw}
 \begin{aligned}
   (\mathcal{R}_1)_{b}&=-2P\sum_{d,m\ge0}
 \frac{(4d+2m)!}{\Gamma(d+1/2)^2(2d)!(m!)^2}
 q_1^dq_2^m\,\mathbf{1}_b\\
 &=-\frac{2P}{\pi}(1+48q_1+2q_2+\cdots)\,\mathbf{1}_b.
 \end{aligned}
\end{equation}
This is exactly \eqref{appD:box-interior} on the box sector.
The pushforward is nonzero although restriction of this ambient 
$P\mathbf1_b$-valued class back to the zero-dimensional sector is zero.

For the distinguished index \(\delta\), the box degrees are 
$d\in 1/2+\mathbf{Z}_{\ge0}$, and
\begin{equation}
\begin{aligned}
  (\mathcal{R}_0)_{b}
  &=-2P\sum_{\substack{d\in\frac{1}{2}+\mathbf{Z}_{\ge0}\\m\ge0}}
  \frac{(4d+2m)!}{\Gamma(1+d)^2(2d)!(m!)^2}
  q_1^dq_2^m\mathbf{1}_b\\
  &=-\frac{16P}{\pi}q_1^{1/2}\mathbf{1}_b+\cdots.
\end{aligned}
\end{equation}
Thus, unlike the example in Appendix \ref{sec:example-p11112}
where the anticanonical hypersurface avoids the
stacky locus, the box sector contributes nontrivially to both of these
Euler-pushed-forward hypersurface series.

\bibliographystyle{amsxport}
\bibliography{reference.bib}

@article{Ganatra:2024aa,
	author = {Sheel Ganatra and Andrew Hanlon and Jeff Hicks and Daniel Pomerleano and Nick Sheridan},
	eprint = {2406.05272},
	month = {06},
	title = {Homological mirror symmetry for {B}atyrev mirror pairs},
	url = {https://arxiv.org/pdf/2406.05272.pdf},
	year = {2024}}

@article{Ganatra:2025aa,
	author = {Sheel Ganatra and Nick Sheridan},
	eprint = {2511.04498},
	month = {11},
	title = {The cyclic open-closed map and variations of {H}odge structures},
	url = {https://arxiv.org/pdf/2511.04498.pdf},
	year = {2025}}

@incollection{2006-Hosono-central-charges-symplectic-forms-and-hypergeometric-series-in-local-mirror-symmetry,
	author = {Hosono, Shinobu},
	booktitle = {Mirror symmetry {V}. {P}roceedings of the {BIRS} workshop on {C}alabi--{Y}au varieties and mirror symmetry, {D}ecember 6--11, 2003},
	isbn = {0-8218-4251-X},
	language = {English},
	pages = {405--439},
	publisher = {Providence, RI: American Mathematical Society (AMS); Somerville, MA: International Press},
	title = {Central charges, symplectic forms, and hypergeometric series in local mirror symmetry},
	year = {2006},
	zbl = {1114.14025},
	zbmath = {5153044}}

@article{2011-Iritani-quantum-cohomology-and-periods,
	author = {Iritani, Hiroshi},
	doi = {10.5802/aif.2798},
	fjournal = {Annales de l'Institut Fourier},
	issn = {1777-5310},
	journal = {Ann. Inst. Fourier},
	number = {7},
	pages = {2909--2958},
	publisher = {MathDoc/Centre Mersenne},
	title = {Quantum Cohomology and Periods},
	url = {http://dx.doi.org/10.5802/aif.2798},
	volume = {61},
	year = {2011}}

@article{2019-Coates-Corti-Iritani-Tseng-some-applications-of-the-mirror-theorem-for-toric-stacks,
	author = {Coates, Tom and Corti, Alessio and Iritani, Hiroshi and Tseng, Hsian-Hua},
	doi = {10.4310/atmp.2019.v23.n3.a4},
	fjournal = {Advances in Theoretical and Mathematical Physics},
	issn = {1095-0753},
	journal = {Adv. Theor. Math. Phys.},
	number = {3},
	pages = {767--802},
	publisher = {International Press of Boston},
	title = {Some applications of the mirror theorem for toric stacks},
	url = {http://dx.doi.org/10.4310/ATMP.2019.v23.n3.a4},
	volume = {23},
	year = {2019}}

@article{1996-Givental-equivariant-gromov-witten-invariants,
	author = {Givental, Alexander B.},
	doi = {10.1155/S1073792896000414},
	fjournal = {IMRN. International Mathematics Research Notices},
	issn = {1073-7928},
	journal = {Int. Math. Res. Not.},
	language = {English},
	number = {13},
	pages = {613--663},
	title = {Equivariant {Gromov}--{Witten} invariants},
	volume = {1996},
	year = {1996},
	zbl = {0881.55006},
	zbmath = {976024}}

@article{2004-Lee-Pandharipande-a-reconstruction-theorem-in-quantum-cohomology-and-quantum-k-theory,
	author = {Lee, Y.-P. and Pandharipande, R.},
	doi = {10.1353/ajm.2004.0049},
	fjournal = {American Journal of Mathematics},
	issn = {1080-6377},
	journal = {Amer. J. Math.},
	month = Dec,
	number = {6},
	pages = {1367--1379},
	publisher = {Project MUSE},
	title = {A reconstruction theorem in quantum cohomology and quantum {K}-theory},
	url = {http://dx.doi.org/10.1353/ajm.2004.0049},
	volume = {126},
	year = {2004}}

@article{2007-Coates-Givental-quantum-riemann-roch-lefschetz-and-serre,
	author = {Coates, Thomas and Givental, Alexander},
	doi = {10.4007/annals.2007.165.15},
	fjournal = {Annals of Mathematics},
	issn = {0003-486X},
	journal = {Ann. of Math.},
	month = Jan,
	number = {1},
	pages = {15--53},
	publisher = {Annals of Mathematics},
	title = {Quantum {R}iemann--{R}och, {L}efschetz and {S}erre},
	url = {http://dx.doi.org/10.4007/annals.2007.165.15},
	volume = {165},
	year = {2007}}

@book{2017-Huber-Muller-Stach-periods-and-nori-motives,
	author = {Huber, Annette and M\"{u}ller-Stach, Stefan},
	isbn = {978-3-319-50925-9; 978-3-319-50926-6},
	mrclass = {14F42 (11G05 14C15 14C30 19E15 32G20)},
	mrnumber = {3618276},
	mrreviewer = {Alberto\ Collino},
	note = {With contributions by Benjamin Friedrich and Jonas von Wangenheim},
	pages = {xxiii+372},
	publisher = {Springer, Cham},
	series = {Ergebnisse der Mathematik und ihrer Grenzgebiete. 3. Folge. A Series of Modern Surveys in Mathematics [Results in Mathematics and Related Areas. 3rd Series. A Series of Modern Surveys in Mathematics]},
	title = {Periods and {N}ori motives},
	volume = {65},
	year = {2017}}

@article{2003-Borisov-Mavlyutov-string-cohomology-of-calabi-yau-hypersurfaces-via-mirror-symmetry,
	author = {Borisov, Lev A. and Mavlyutov, Anvar R.},
	doi = {10.1016/s0001-8708(03)00007-0},
	fjournal = {Advances in Mathematics},
	issn = {0001-8708},
	journal = {Adv. Math.},
	month = Dec,
	number = {1},
	pages = {355--390},
	publisher = {Elsevier BV},
	title = {String cohomology of {C}alabi--{Y}au hypersurfaces via mirror symmetry},
	url = {http://dx.doi.org/10.1016/S0001-8708(03)00007-0},
	volume = {180},
	year = {2003}}

@incollection{1998-Stienstra-resonant-hypergeometric-systems-and-mirror-symmetry,
	author = {Stienstra, Jan},
	booktitle = {Integrable systems and algebraic geometry ({K}obe/{K}yoto, 1997)},
	isbn = {981-02-3266-7},
	mrclass = {14D05 (14J32 14M25 14N35 33C20)},
	mrnumber = {1672077},
	mrreviewer = {Bruce\ Hunt},
	pages = {412--452},
	publisher = {World Sci. Publ., River Edge, NJ},
	title = {Resonant hypergeometric systems and mirror symmetry},
	year = {1998}}

@incollection{2013-Horja-toric-deligne-mumford-stacks-and-the-better-behaved-version-of-the-gkz-hypergeometric-system,
	author = {Horja, R. Paul},
	booktitle = {Strings, gauge fields, and the geometry behind},
	isbn = {978-981-4412-54-4},
	mrclass = {14D23 (14M25 33C70 81T25)},
	mrnumber = {3289345},
	mrreviewer = {V.\ A.\ Golubeva},
	pages = {329--348},
	publisher = {World Sci. Publ., Hackensack, NJ},
	title = {Toric {D}eligne--{M}umford stacks and the better behaved version of the {GKZ} hypergeometric system},
	year = {2013}}

@article{1991-Candelas-de-la-Ossa-Greene-Parkes-a-pair-of-calabi-yau-manifolds-as-an-exactly-soluable-superconformal-theory,
	author = {Candelas, Philip and de la Ossa, Xenia C. and Green, Paul S. and Parkes, Linda},
	doi = {10.1016/0550-3213(91)90292-6},
	fjournal = {Nuclear Physics. B. Theoretical, Phenomenological, and Experimental High Energy Physics. Quantum Field Theory and Statistical Systems},
	issn = {0550-3213},
	journal = {Nuclear Phys. B},
	mrnumber = {1115626},
	number = {1},
	pages = {21--74},
	title = {A pair of {C}alabi--{Y}au manifolds as an exactly soluble superconformal theory},
	volume = {359},
	year = {1991}}

@article{1971-Deligne-theorie-de-hodge-ii,
	author = {Deligne, Pierre},
	fjournal = {Institut des Hautes \'{E}tudes Scientifiques. Publications Math\'{e}matiques},
	issn = {0073-8301},
	journal = {Publ. Math. Inst. Hautes {\'E}tudes},
	mrnumber = {498551},
	number = {40},
	pages = {5--57},
	title = {Th\'{e}orie de {H}odge. {II}},
	year = {1971}}

@article{1990-Greene-Plesser-duality-in-calabi-yau-moduli-space,
	author = {Greene, B. R. and Plesser, M. R.},
	doi = {10.1016/0550-3213(90)90622-K},
	issn = {0550-3213},
	journal = {Nuclear Physics. B. Theoretical, Phenomenological, and Experimental High Energy Physics. Quantum Field Theory and Statistical Systems},
	mrnumber = {1059831},
	number = {1},
	pages = {15--37},
	title = {Duality in {C}alabi--{Y}au moduli space},
	volume = {338},
	year = {1990}}

@article{1993-Batyrev-variations-of-the-mixed-hodge-structure-of-affine-hypersurfaces-in-algebraic-tori,
	author = {Batyrev, Victor V.},
	doi = {10.1215/S0012-7094-93-06917-7},
	fjournal = {Duke Mathematical Journal},
	issn = {0012-7094},
	journal = {Duke Math. J.},
	mrnumber = {1203231},
	number = {2},
	pages = {349--409},
	title = {Variations of the mixed {H}odge structure of affine hypersurfaces in algebraic tori},
	volume = {69},
	year = {1993}}

@article{1994-Batyrev-dual-polyhedra-and-mirror-symmetry-for-calabi-yau-hypersurfaces-in-toric-varieties,
	author = {Victor V. Batyrev},
	fjournal = {Journal of Algebraic Geometry},
	journal = {J. Algebraic Geom.},
	pages = {493-545},
	title = {{Dual polyhedra and mirror symmetry for Calabi--Yau hypersurfaces in toric varieties}},
	volume = {3},
	year = {1994}}

@article{1995-Hosono-Klemm-Theisen-Yau-mirror-symmetry-mirror-map-and-applications-to-calabi-yau-hypersurfaces,
	author = {Hosono, S. and Klemm, A. and Theisen, S. and Yau, S.-T.},
	fjournal = {Communications in Mathematical Physics},
	issn = {0010-3616},
	journal = {Comm. Math. Phys.},
	mrnumber = {1316509},
	number = {2},
	pages = {301--350},
	title = {Mirror symmetry, mirror map and applications to {C}alabi--{Y}au hypersurfaces},
	volume = {167},
	year = {1995}}

@article{1995-Hosono-Klemm-Theisen-Yau-mirror-symmetry-mirror-map-and-applications-to-complete-intersection-calabi-yau-spaces,
	author = {Hosono, S. and Klemm, A. and Theisen, S. and Yau, S.-T.},
	doi = {10.1016/0550-3213(94)00440-P},
	issn = {0550-3213},
	journal = {Nuclear Physics. B. Theoretical, Phenomenological, and Experimental High Energy Physics. Quantum Field Theory and Statistical Systems},
	mrnumber = {1319280},
	number = {3},
	pages = {501--552},
	title = {Mirror symmetry, mirror map and applications to complete intersection {C}alabi--{Y}au spaces},
	volume = {433},
	year = {1995}}

@article{1996-Batyrev-Borisov-mirror-duality-and-string-theoretic-hodge-numbers,
	author = {Victor V. Batyrev and Lev A. Borisov},
	fjournal = {Inventiones Mathematicae},
	journal = {Invent. Math.},
	pages = {183-203},
	title = {Mirror duality and string-theoretic {H}odge numbers},
	volume = {126},
	year = {1996}}

@incollection{1996-Batyrev-Borisov-on-calabi-yau-complete-intersections-in-toric-varieties,
	author = {Batyrev, Victor V. and Borisov, Lev A.},
	booktitle = {Higher-dimensional complex varieties ({T}rento, 1994)},
	mrnumber = {1463173},
	pages = {39--65},
	publisher = {de Gruyter, Berlin},
	title = {On {C}alabi--{Y}au complete intersections in toric varieties},
	year = {1996}}

@incollection{1998-Givental-a-mirror-theorem-for-toric-complete-intersections,
	author = {Givental, Alexander B.},
	booktitle = {Topological field theory, primitive forms and related topics ({K}yoto, 1996)},
	mrnumber = {1653024},
	pages = {141--175},
	publisher = {Birkh\"{a}user Boston, Boston, MA},
	series = {Progr. Math.},
	title = {A mirror theorem for toric complete intersections},
	volume = {160},
	year = {1998}}

@article{2000-Hosono-local-mirror-symmetry-and-type-IIA-monodromy-of-calabi-yau-manifolds,
	author = {Shinobu Hosono},
	fjournal = {Advances in Theoretical and Mathematical Physics},
	journal = {Adv. Theor. Math. Phys.},
	number = {2},
	pages = {335-376},
	title = {{Local mirror symmetry and type IIA monodromy of Calabi--Yau manifolds}},
	volume = {4},
	year = {2000}}

@article{2006-Borisov-Horja-mellin-barnes-integrals-as-fourier-mukai-transforms,
	author = {Borisov, Lev A. and Horja, R. Paul},
	doi = {10.1016/j.aim.2006.01.011},
	fjournal = {Advances in Mathematics},
	issn = {0001-8708},
	journal = {Adv. Math.},
	mrnumber = {2271990},
	number = {2},
	pages = {876--927},
	title = {Mellin--{B}arnes integrals as {F}ourier--{M}ukai transforms},
	volume = {207},
	year = {2006}}

@book{2011-Cox-Little-Schenck-toric-varieties,
	author = {Cox, David A. and Little, John B. and Schenck, Henry K.},
	doi = {10.1090/gsm/124},
	isbn = {978-0-8218-4819-7},
	mrnumber = {2810322},
	pages = {xxiv+841},
	publisher = {American Mathematical Society, Providence, RI},
	series = {Graduate Studies in Mathematics},
	title = {Toric varieties},
	volume = {124},
	year = {2011}}

@article{2013-Lian-Yau-period-integrals-of-cy-and-general-type-complete-intersections,
	author = {Lian, Bong H. and Yau, Shing-Tung},
	doi = {10.1007/s00222-012-0391-6},
	fjournal = {Inventiones Mathematicae},
	issn = {0020-9910},
	journal = {Invent. Math.},
	mrnumber = {3004778},
	number = {1},
	pages = {35--89},
	title = {Period integrals of {CY} and general type complete intersections},
	volume = {191},
	year = {2013}}

@article{2013-Lian-Song-Yau-periodic-integrals-and-tautological-systems,
	author = {Lian, Bong H. and Song, Ruifang and Yau, Shing-Tung},
	doi = {10.4171/JEMS/396},
	fjournal = {Journal of the European Mathematical Society},
	issn = {1435-9855},
	journal = {J. Eur. Math. Soc. (JEMS)},
	mrnumber = {3055766},
	number = {4},
	pages = {1457--1483},
	title = {Periodic integrals and tautological systems},
	volume = {15},
	year = {2013}}

@article{2014-Bloch-Huang-Lian-Srinivas-Yau-on-the-holonomic-rank-problem,
	author = {Bloch, Spencer and Huang, An and Lian, Bong H. and Srinivas, Vasudevan and Yau, Shing-Tung},
	fjournal = {Journal of Differential Geometry},
	issn = {0022-040X},
	journal = {J. Differential Geom.},
	mrnumber = {3229048},
	number = {1},
	pages = {11--35},
	title = {On the holonomic rank problem},
	volume = {97},
	year = {2014}}

@article{2016-Huang-Lian-Zhu-period-integrals-and-riemann-hilbert-correspondence,
	author = {Huang, An and Lian, Bong H. and Zhu, Xinwen},
	fjournal = {Journal of Differential Geometry},
	issn = {0022-040X},
	journal = {J. Differential Geom.},
	mrnumber = {3557307},
	number = {2},
	pages = {325--369},
	title = {Period integrals and the {R}iemann--{H}ilbert correspondence},
	volume = {104},
	year = {2016}}

@article{1987-Gelfand-Graev-Zelevinski-holonomic-systems-of-equations-and-series-of-hypergeometric-type,
	author = {Gel'fand, Israel M. and Graev, Mark I. and Zelevinski\u{\i}, Andrei V.},
	fjournal = {Doklady Akademii Nauk SSSR},
	issn = {0002-3264},
	journal = {Dokl. Akad. Nauk SSSR},
	mrclass = {58G07 (22E45 32C38 32C42 33A30)},
	mrnumber = {902936},
	mrreviewer = {J. S. Joel},
	number = {1},
	pages = {14--19},
	title = {Holonomic systems of equations and series of hypergeometric type},
	url = {https://mathscinet.ams.org/mathscinet-getitem?mr=902936},
	volume = {295},
	year = {1987}}

@article{1989-Gelfand-Kapranov-Zelevinski-hypergeometric-functions-and-toral-manifolds,
	author = {Gel'fand, Israel M. and Kapranov, Mikhail M. and Zelevinski\u{\i}, Andrei V.},
	doi = {10.1007/BF01078777},
	fjournal = {Akademiya Nauk SSSR. Funktsional\cprime ny\u{\i} Analiz i ego Prilozheniya},
	issn = {0374-1990},
	journal = {Funktsional. Anal. i Prilozhen.},
	mrclass = {22E30 (14L32 32C42 33A30 58G07)},
	mrnumber = {1011353},
	mrreviewer = {A. Verona},
	number = {2},
	pages = {12--26},
	title = {Hypergeometric functions and toric varieties},
	url = {https://mathscinet.ams.org/mathscinet-getitem?mr=1011353},
	volume = {23},
	year = {1989}}

@article{1996-Hosono-Lian-Yau-gkz-generalized-hypergeometric-systems-in-mirror-symmetry-of-calabi-yau-hypersurfaces,
	author = {Hosono, Shinobu and Lian, Bong H. and Yau, Shing-Tung},
	fjournal = {Communications in Mathematical Physics},
	issn = {0010-3616},
	journal = {Comm. Math. Phys.},
	mrnumber = {1461942},
	number = {3},
	pages = {535--577},
	title = {G{KZ}-generalized hypergeometric systems in mirror symmetry of {C}alabi--{Y}au hypersurfaces},
	volume = {182},
	year = {1996}}

@article{1997-Lian-Liu-Yau-mirror-principle-i,
	author = {Lian, Bong H. and Liu, Kefeng and Yau, Shing-Tung},
	doi = {10.4310/AJM.1997.v1.n4.a5},
	fjournal = {Asian Journal of Mathematics},
	issn = {1093-6106},
	journal = {Asian J. Math.},
	mrclass = {14N10 (14D20 14J32)},
	mrnumber = {1621573},
	mrreviewer = {Andreas Gathmann},
	number = {4},
	pages = {729--763},
	title = {Mirror principle {I}},
	url = {https://mathscinet.ams.org/mathscinet-getitem?mr=1621573},
	volume = {1},
	year = {1997}}

@article{1997-Hosono-Lian-Yau-maximal-degeneracy-points-of-gkz-systems,
	author = {Shinobu Hosono and Bong H. Lian and Shing-Tung Yau},
	fjournal = {Journal of the American Mathematical Society},
	journal = {J. Amer. Math. Soc.},
	month = {April},
	number = {2},
	pages = {427-443},
	title = {{Maximal degeneracy points of GKZ systems}},
	volume = {10},
	year = {1997}}

@article{2005-Borisov-Chen-Smith-the-orbifold-chow-ring-of-toric-deligne-mumford-stacks,
	author = {Borisov, Lev A. and Chen, Linda and Smith, Gregory G.},
	doi = {10.1090/S0894-0347-04-00471-0},
	fjournal = {Journal of the American Mathematical Society},
	issn = {0894-0347},
	journal = {J. Amer. Math. Soc.},
	mrnumber = {2114820},
	number = {1},
	pages = {193--215},
	title = {The orbifold {C}how ring of toric {D}eligne--{M}umford stacks},
	volume = {18},
	year = {2005}}

@article{2010-Tseng-orbifold-quantum-riemann-roch-lefschetz-and-serre,
	author = {Tseng, Hsian-Hua},
	doi = {10.2140/gt.2010.14.1},
	fjournal = {Geometry \& Topology},
	issn = {1465-3060},
	journal = {Geom. Topol.},
	mrnumber = {2578300},
	number = {1},
	pages = {1--81},
	title = {Orbifold quantum {R}iemann--{R}och, {L}efschetz and {S}erre},
	volume = {14},
	year = {2010}}

@article{2015-Coates-Corti-Iritani-Tseng-a-mirror-theorem-for-toric-stacks,
	author = {Coates, Tom and Corti, Alessio and Iritani, Hiroshi and Tseng, Hsian-Hua},
	doi = {10.1112/S0010437X15007356},
	fjournal = {Compositio Mathematica},
	issn = {0010-437X},
	journal = {Compositio Math.},
	mrnumber = {3414388},
	number = {10},
	pages = {1878--1912},
	title = {A mirror theorem for toric stacks},
	volume = {151},
	year = {2015}}

@article{2020-Lee-Zhang-a-hypergeometric-systems-and-relative-cohomology,
	author = {Lee, Tsung-Ju and Zhang, Dingxin},
	doi = {10.1142/S0129167X2050113X},
	fjournal = {International Journal of Mathematics},
	issn = {0129-167X},
	journal = {Internat. J. Math.},
	mrnumber = {4192455},
	number = {13},
	pages = {2050113, 21},
	title = {{$A$}-hypergeometric systems and relative cohomology},
	volume = {31},
	year = {2020}}

@article{2021-Lian-Zhu-on-the-hyperplane-conjecture-for-periods-of-calabi-yau-hypersurfaces-in-pn,
	author = {Lian, Bong H. and Zhu, Minxian},
	doi = {10.4310/jdg/1620272942},
	fjournal = {Journal of Differential Geometry},
	issn = {0022-040X},
	journal = {J. Differential Geom.},
	mrclass = {14J33 (14D05)},
	mrnumber = {4255072},
	number = {1},
	pages = {101--146},
	title = {On the hyperplane conjecture for periods of {C}alabi--{Y}au hypersurfaces in {${\bf P}^n$}},
	url = {https://mathscinet.ams.org/mathscinet-getitem?mr=4255072},
	volume = {118},
	year = {2021}}
\end{document}